\documentclass[english,11pt,reqno]{amsart}
\title[Weak Fano threefolds arising as the blowup of a quadric along a curve]{Weak Fano threefolds arising as the blowup of a hyperquadric in \texorpdfstring{$\Pfour$}{P4} along a curve}
\author{Anne Schnattinger}
\address{Anne Schnattinger, Institut de mathématiques, Université de Neuchâtel
	\\ Rue Emile-Argand 11 CH-2000 Neuchâtel}
\email{anne.schnattinger@unine.ch}
\thanks{The author acknowledges support by the Swiss National Science Foundation Grant “Birational transformations of
	higher dimensional varieties” 200020-214999.}
\date{}
\usepackage{amsmath,amsfonts,amssymb,amsthm}
\usepackage{xcolor}
\usepackage{geometry}
\usepackage[colorlinks=true, allcolors=teal]{hyperref}
\usepackage{enumitem}
\usepackage{cleveref}
\usepackage{mathtools}
\usepackage{multicol}
\usepackage[makeroom]{cancel}
\usepackage{tikz-cd} 
\usepackage{float}
\usepackage[width=.9\textwidth]{caption}

\newcommand{\sub}{\subseteq}

\newcommand{\Pone}{\mathbb{P}^1}
\newcommand{\Ptwo}{\mathbb{P}^2}
\newcommand{\Pthree}{\mathbb{P}^3}
\newcommand{\Pfour}{\mathbb{P}^4}
\newcommand{\Pfive}{\mathbb{P}^5}
\newcommand{\Pn}{\mathbb{P}^n}
\newcommand{\PxP}{\mathbb{P}^1 \times \mathbb{P}^1}

\newcommand{\VPtwo}{V_{\Ptwo}}
\newcommand{\VPthree}{V_{\Pthree}}
\newcommand{\VPfour}{V_{\Pfour}}
\newcommand{\VPn}{V_{\Pn}}
\newcommand{\kk}{\mathbf{k}}
\newcommand{\kPthree}{\kk[x_0,\dots,x_3]}
\newcommand{\kPfour}{\kk[x_0,\dots,x_4]}
\newcommand{\kPn}{\kk[x_0,\dots,x_n]}
\newcommand{\kPxP}{\kk[s_0,s_1,t_0,t_1]}

\newcommand{\Z}{\mathbb{Z}}
\newcommand{\R}{\mathbb{R}}

\newcommand{\Pic}{\textup{Pic}}
\newcommand{\Div}{\textup{Div}}
\newcommand{\NE}{\overline{\textup{NE}}}

\newcommand{\dd}{\mathfrak{c}}

\newcommand{\bl}{\textup{bl}}
\newcommand{\Bl}{\textup{Bl}}
\newcommand{\supp}{\textup{Supp}}
\newcommand{\Ftwo}{\mathbb{F}_2}

\newcommand{\PGL}{\textup{PGL}}
\newcommand{\GL}{\textup{GL}}
\newcommand{\Aut}{\textup{Aut}}

\newtheorem{thm}{Theorem}[section]

\newtheorem{rmk}[thm]{Remark}
\newtheorem{lem}[thm]{Lemma}
\newtheorem{cor}[thm]{Corollary}
\newtheorem{prop}[thm]{Proposition}
\newtheorem{ex}[thm]{Example}

\newtheorem{mythm}{Theorem}
\renewcommand*{\themythm}{\Alph{mythm}}

\renewcommand\labelenumi{(\roman{enumi})}
\renewcommand\theenumi\labelenumi
\newlist{enummainthm}{enumerate}{1}
\setlist[enummainthm]{label=\upshape(\roman*), ref=\upshape\themainthm(\roman*)}
\crefname{enummainthmi}{Theorem}{Theorems}
\newlist{enumprop}{enumerate}{1}
\setlist[enumprop]{label=\upshape(\roman*), ref=\upshape\theprop(\roman*)}
\crefname{enumpropi}{Proposition}{Propositions}
\newlist{enumthm}{enumerate}{1}
\setlist[enumthm]{label=\upshape(\roman*), ref=\upshape\thethm(\roman*)}
\crefname{enumthmi}{Theorem}{Theorems}
\newlist{enumlem}{enumerate}{1}
\setlist[enumlem]{label=\upshape(\roman*), ref=\upshape\thelem(\roman*)}
\crefname{enumlemi}{Lemma}{Lemmas}
\newlist{enumcor}{enumerate}{1}
\setlist[enumcor]{label=\upshape(\roman*), ref=\upshape\thecor(\roman*)}
\crefname{enumcori}{Corollary}{Corollaries}
\newlist{enumex}{enumerate}{1}
\setlist[enumex]{label=\upshape(\roman*), ref=\upshape\theex(\roman*)}
\crefname{enumexi}{Example}{Examples}
\newlist{enumrmk}{enumerate}{1}
\setlist[enumrmk]{label=\upshape(\roman*), ref=\upshape\thermk(\roman*)}
\crefname{enumrmki}{Remark}{Remarks}
\newlist{enumdef}{enumerate}{1}
\setlist[enumdef]{label=\upshape(\roman*), ref=\upshape\thedef(\roman*)}
\crefname{enumdefi}{Definition}{Definitions}
\newlist{enummythm}{enumerate}{1}
\setlist[enummythm]{label=\upshape(\roman*), ref=\upshape\themythm(\roman*)}
\crefname{enummythmi}{Theorem}{Theorems}
\makeatletter
\newcounter{subcreftmpcnt} %
\newcommand\alphsubformat[1]{(\roman{#1})} 
\newcommand\subcref[2][\alphsubformat]{%
	\ifcsname r@#2@cref\endcsname
	\cref@getcounter {#2}{\mylabel}%
	\setcounter{subcreftmpcnt}{\mylabel}%
	\hyperref[#2]{\alphsubformat{subcreftmpcnt}}%
	\else ?? \fi}   
\makeatother

\begin{document}
	\maketitle
	\begin{abstract}
		We characterize smooth irreducible curves $C$ on a smooth hyperquadric $Y$ of $\Pfour$ such that the blowup of $Y$ along $C$ is a weak Fano threefold. These are precisely the smooth irreducible curves $C$ of degree $d$ and genus $g$ lying on a smooth hypercubic section of $Y$ such
		that (i) $C$ has no 4-secant line and no 7-secant conic; (ii) $d< 18$ and $(g,d)\not \in \{(4,7),\:(10, 11)\}$; (iii) either $3d-26<g\leq\frac{d^2-1}{12}$ or $(g,d)\in \{(4,6),\:(13,12)\}$. \par We prove the geometric realizability of each case, thereby proving the existence of weak Fano threefolds and Sarkisov links constructed from them, which were previously known only as numerical possibilities.
	\end{abstract}
	%
	%
	\tableofcontents
	\section{Introduction} \label{section:introduction}
	Recall that a smooth projective variety $X$ is said to be Fano if its anticanonical divisor $-K_X$ is ample, and $X$ is said to be weak Fano if $-K_X$ is big and nef. Working over an algebraically closed field $\kk$ of characteristic 0, the aim of this paper is to prove the following theorem:
	\begin{mythm} \label{thm:A}
		Let $C \sub Y$ be a smooth irreducible curve of genus $g$ and degree $d$ that lies on a smooth hyperquadric $Y \sub \Pfour$ and denote the blowup of $Y$ along $C$ by $X$. Then $X$ is weak Fano if and only if
		\begin{enummythm} 
			\item $C$ is contained in a smooth hypercubic section of $Y$, \label{thm:A_hypercubic}
			\item $C$ has no 4-secant line and no 7-secant conic, \label{thm:A_secants}
			\item either $(g,d)\in \{(4,6),\:(13,12)\}$, or $d < 18$, $3d-26<g\leq\frac{d^2-1}{12}$ and $(g,d)\not \in \{(4,7),\:(10, 11)\}$. \label{thm:A_g_and_d}
		\end{enummythm}
	\end{mythm}
	We introduce the following notations:
	\begin{itemize}
		\item $\mathcal{P}_{\textup{none}}:=\{(0,1),\:(0,2),\:(0,3),\:(0,4),\:(0,5),\:(1,4),\:(1,5),\:(1,6),\:(2,5),\:(2,6),\:(2,7)\} \cup \{(3,7),\:(4,6),\:(4,8),\:(5,8),\:(6,9),\:(8,10),\: (9,11),\:(13,12),\:(14,13)\};$
		\item $\mathcal{P}_{\textup{line}}:=\{(0,6),\:(0,7),\:(1,7),\:(2,8),\:(3,8),\:(4,9),\: (5,9),\:(6,10),\:(7,10),\:(8,11),\: (11,12)\}$;
		\item $\mathcal{P}_{\textup{conic}}:= \{(0,8),\:(1,8),\:(2,9),\:(3,9),\:(5,10)\}$;
		\item $\mathcal{P}:=\mathcal{P}_{\textup{none}} \cup \mathcal{P}_{\textup{line}} \cup \mathcal{P}_{\textup{conic}}$.
	\end{itemize}
	With this notation, condition \subcref{thm:A_g_and_d} of \Cref{thm:A} is equivalent to $(g,d)\in \mathcal{P}$ and we can reformulate and specify \Cref{thm:A} as follows:
	\begin{mythm} \label{thm:B}
		Let $C \sub Y$ be a smooth irreducible curve of genus $g$ and degree $d$ that lies on a smooth hyperquadric $Y \sub \Pfour$ and denote the blowup of $Y$ along $C$ by $X$. Then $X$ is weak Fano if and only if $C$ is contained in a smooth hypercubic section and one of the following applies:
		\begin{enummythm}
			\item $(g,d) \in \mathcal{P}_{\textup{none}}$, \label{thm:B_no_secants}
			\item $(g,d)\in \mathcal{P}_{\textup{line}}$ and $C$ does not have a 4-secant line, \label{thm:B_4_secant_line}
			\item $(g,d)\in \mathcal{P}_{\textup{conic}}$, $C$ has no 4-secant line and no 7-secant conic. \label{thm:B_7_secant_conic}
		\end{enummythm}
	\end{mythm}
	The classification of (weak) Fano varieties of dimension 2, which are called (weak) del Pezzo surfaces, is a very classical result (see e.g.\ \cite[Chapter 8.1]{Dolgachev}): $X$ is a weak del Pezzo surface if and only if $X$ is isomorphic to $\PxP$, to the second Hirzebruch surface $\Ftwo$ or to the blowup of $\Ptwo$ in at most 8 (possibly infinitely near) points in almost general position. The latter condition means that there is no line through 4 of the points, no  conic through 7 of the points and no exceptional divisor through two infinitely near points.\par 
	Moreover, $X$ is a del Pezzo surface if and only if $X$ is either isomorphic to $\PxP$ or to the blowup of $\Ptwo$ in at most 8 points in general linear position. That is to say, there is no line through 3 of the points, no conic though 6 of them, and no cubic through 8 points that is singular at one of them. \par 
	In dimension 2, this provides a complete classification of (weak) del Fano varieties into 10 (respectively 11) deformation families. Furthermore, the characterization of the families that are given by the blowup of a variety along a subvariety depends only on geometric properties of the blown-up subvariety. \par 
	In dimension 3, there is also a complete classification of Fano threefolds into 105 deformation families. It was established by Iskovskikh (\cite{Iskovskikh-Fano_threefolds_I} and \cite{Iskovskikh-Fano_threefolds_II}) for Fano threefolds of Picard rank 1 and by Mori and Mukai for Fano threefolds of a higher Picard rank (\cite{Mori_Mukai-Fano_threefolds_Picard_rank_2}). \par
	Nonetheless, the classification of weak Fano threefolds is far from being complete. \par 
	A weak Fano threefold $Z$ of Picard rank one is a Fano threefold: Being a projective variety, $Z$ admits an ample divisor $A$, which generates the one-dimensional Picard group $\Pic(Z)$. Since the anticanonical divisor $-K_Z$ is big, it must be a positive multiple of the ample divisor $A$, meaning that $-K_Z$ must itself be ample. So the weak Fano threefolds of Picard rank one can be found in Iskovskikh's classification of Fano threefolds.\par
	Multiple authors have worked on the classification of weak Fano threefolds of Picard rank two, starting with \cite{Weak_Fano_threefolds_I-Jahnke_Peternell_Radloff} and \cite{Weak_Fano_threefolds_II-Jahnke_Peternell_Radloff}. The classification was partially completed by \cite{Takeuchi-Weak_Fano_threefolds_del_Pezzo_fibration} and continued by a series of papers \cite{Cutrone_Marshburn_Classification}, \cite{Arap_Cutrone_Marshburn_weak_Fano_threefolds_Existence} and \cite{Cutrone_Marshburn_Update}. The authors of these papers have numerically classified the weak Fano threefolds of Picard rank two based on the type of extremal contractions that can appear on the weak Fano threefold. For more background on the type of extremal contractions that can appear and which of the authors have classified which type of extremal rays, we refer the reader to the introduction of the paper \cite{Cutrone_Marshburn_Classification} and \Cref{section:Sarkisov_links}.\par  
	
	The six previously mentioned papers not just give a finite list of the families of weak Fano threefolds of Picard rank 2, which can numerically occur. The geometric realizability of these weak Fano threefolds is also proven for all cases in \cite{Takeuchi-Weak_Fano_threefolds_del_Pezzo_fibration}, for all but four cases in \cite{Cutrone_Marshburn_Classification}, \cite{Arap_Cutrone_Marshburn_weak_Fano_threefolds_Existence} and \cite{Cutrone_Marshburn_Update}, as well as for many, but not all, cases in \cite{Weak_Fano_threefolds_I-Jahnke_Peternell_Radloff} and \cite{Weak_Fano_threefolds_II-Jahnke_Peternell_Radloff}. \par 
	A common way to obtain weak Fano threefolds of Picard rank 2 is to blowup a Fano threefold of Picard rank one along a subvariety. Studying such subfamilies of weak Fano threefolds allows us to obtain a classification of these weak Fano threefolds, depending only on some geometric conditions of the blown-up subvariety. The authors of the papers \cite{Blanc_Lamy_Weak_Fano_threefolds} and \cite{Blanc_Lamy_Cubic} obtained such a classification of weak Fano threefolds arising as the blowup of $\Pthree$ or a smooth hypercubic of $\Pfour$ along a smooth irreducible curve $C$ of a given degree and genus. Their classification depends only on the degree and genus of the curve $C$, the existence of certain surfaces containing $C$ and the absence of curves of low degrees intersecting $C$ too many times. \par 
	Replacing $\Pthree$ or a smooth hypercubic in $\Pfour$ by a smooth hyperquadric in $\Pfour$, we obtain similar geometric obstructions on the curve $C$ as in \cite{Blanc_Lamy_Cubic} and \cite{Blanc_Lamy_Weak_Fano_threefolds}. One peculiarity of our case, which does not appear in the other two papers, is that weak Fano threefolds arising from the blowup of a hyperquadric along curves of the same degree $d$ and genus $g$ can exhibit very different behaviors: For some curves of degree $d$ and genus $g$, they give rise to a Sarkisov link, and for other curves of degree $d$ and genus $g$, they do not. This dichotomy appears for the four pairs $(g,d)\in\{(0,4),\:(3,8),\:(6,10),\:(11,12)\}$. \par
	Using a result from \cite{Knutsen-Smooth_curves_on_projective_K3_surface} concerning the existence of integral smooth curves on smooth K3-surfaces, we are able to prove the geometric realizability of all families of weak Fano threefolds arising from \Cref{thm:A}. This guarantees the existence of weak Fano threefolds and Sarkisov links of type I arising from them, whose existence was not proven in \cite{Weak_Fano_threefolds_II-Jahnke_Peternell_Radloff}, and was was therefore unknown previously. These are the weak Fano threefolds corresponding to the pairs $(g,d)\in\{(0,7),\:(5,9)\}$.\par 
	As an auxiliary tool, we also generalize the genus bound for integral smooth curves on smooth K3 surfaces of degree 6, as given in \cite{Knutsen-Smooth_curves_on_projective_K3_surface}, to arbitrary reduced curves, which may be reducible or singular. \par 
	The structure of the paper is as follows. \par 
	In \Cref{section:arithmetic_genus}, we collect preliminary results that will be used for the main proofs of the rest of the paper. These will mostly concern the arithmetic genus of (potentially reducible) curves in $\Pfour$, since the proof of \Cref{thm:A} relies heavily on the relationship between the genus and degree of the curve, as well as on the existence of hypersurfaces and surfaces of certain degrees containing the curve. \par
	\Cref{section:hypercubic} entails the proof of \Cref{thm:A} and \ref{thm:B}. In \S\ref{subsection:blowups} and \S\ref{subsection:anticanonical_system_base_point_free}, we will respectively show the necessity of conditions \subcref{thm:A_secants} and \subcref{thm:A_hypercubic} in \Cref{thm:A}. In particular, we will see that if $X$ is weak Fano, the curve $C$ is always contained in a smooth K3-surface of degree 6. We also show that condition \subcref{thm:A_g_and_d} in \Cref{thm:A} is necessary in \S\ref{subsection:bounding_genus}, by bounding the genus of a curve contained in such a sextic K3-surface. This will also be the main tool to prove the sufficiency of the conditions of \Cref{thm:A} in \S\ref{subsection:curves_yielding_weak_Fano}. Subsection \ref{subsection:existence} addresses the geometric realizability of a corresponding weak Fano threefold for each pair $(g,d)\in\mathcal{P}$. \par
	In \Cref{section:hyperplane}, we can weaken or modify the assumptions of \Cref{thm:B} for some pairs $(g,d)\in \mathcal{P}$. More precisely, for the six pairs
	$\{(0,1),\:(0,2),\:(0,3),\:(1,4),\:(2,5),\:(4,6)\},$ and the pair $(0,4)$ corresponding to a curve lying on a hyperplane, we prove that a corresponding curve will always yield a weak Fano threefold without any further assumptions on the curve $C$ (\S\ref{subsection:hyperplane_weak_Fano}). Whereas for the 14 pairs in the set
	$$\{ (0,4),\:(0,5),\:(0,6),\:(1,5),\:(1,6),\:(2,6),\:(2,7),\:(3,7),\:(3,8),\:(4,8),\:(5,8),\:(6,9),\:(8,10),\:(13,12)\},$$  the assumption that $C$ is contained in a smooth hyperquadric section of $Y$, and has no 4-secant line if $(g,d)\in \{(0,6),\:(3,8)\}$, are sufficient to show that $X$ is weak Fano (\S\ref{subsection:hyperquadric_weak_Fano}). To prove these two modifications of \Cref{thm:B}, we will have to study smooth and singular hyperplane sections of $Y$ (\S\ref{subsection:hyperplane_preliminaries}), as well as smooth del Pezzo surfaces of degree 4 (\S\ref{subsection:hyperquadric_preliminaries}). \par
	Given the genus $g$ and degree $d$ of the blown-up curve $C$, the last \Cref{section:Sarkisov_links} categorizes the weak Fano threefold $X$ depending on the properties of the birational morphism $\psi$ associated to the big anticanonical divisor (\S\ref{subsection:extremal_contractions}). To do this, we make use of the classification of Fano threefolds in \cite{Mori_Mukai-Fano_threefolds_Picard_rank_2}, as well as the numerical classification of weak Fano threefolds of Picard rank two from the papers \cite{Weak_Fano_threefolds_I-Jahnke_Peternell_Radloff}, \cite{Weak_Fano_threefolds_II-Jahnke_Peternell_Radloff}, \cite{Takeuchi-Weak_Fano_threefolds_del_Pezzo_fibration}, \cite{Cutrone_Marshburn_Classification}, \cite{Arap_Cutrone_Marshburn_weak_Fano_threefolds_Existence} and \cite{Cutrone_Marshburn_Update}. If $\psi$ is an isomorphism, then $X$ is a Fano threefold which gives rise to a Sarkisov link of type I or II (\S\ref{subsection:Fano_threefolds}). If $\psi$ is a divisorial contraction, then it does not produce a Sarkisov link (\S\ref{subsection:no_flop}). Otherwise, the anticanonical map $\psi$ is small. In this case, we can perform a flop to again obtain a Sarkisov link of type I or II (\S\ref{subsection:flop}).
	
	\subsubsection*{Acknowledgements} 
	I would like to thank my supervisor Jérémy Blanc for suggesting the problem to me as well as the continuous support throughout the writing process. I would also like to thank Stéphane Lamy, Sokratis Zikas and Henrik Wehrheim for the helpful comments and discussions. 
	%
	%
	%
	\section{Arithmetic genus of curves on a smooth hyperquadric} \label{section:arithmetic_genus}
	For the rest of the paper, let $\kk$ denote an algebraically closed field of characteristic 0. Furthermore, we assume all varieties to be projective, but potentially reducible or non-reduced unless stated otherwise. \par 
	This section compiles auxiliary lemmas on the elementary properties of the arithmetical genus of curves (\S\ref{subsection:arithmetic_genus}) and on how the genus depends on the curve being contained in larger varieties of a certain degree (\S\ref{subsection:genus_curves_hypersurfaces}). 
	\subsection{Arithmetic genus of (possibly singular) curves} \label{subsection:arithmetic_genus}
	Recall that the arithmetic genus of a projective variety $Z$, whose irreducible components all have dimension $m$, is defined as $$p_a(Z)=(-1)^m(\chi(\mathcal{O}_Z)-1),$$
	where the Euler characteristic $\chi(\mathcal{O}_Z)= \sum_{i= 0}^m (-1)^i \dim_{\kk}(H^i(Z,\:\mathcal{O}_Z))$ is given by the constant term of the Hilbert polynomial of $Z$. \par 
	Note, in particular, that this definition does not depend on the embedding of $Z$ into projective space. However, given a reduced projective variety $Z$ and an integer $l\in \Z$, we can consider the non-reduced variety $Z'=l\cdot Z$ as a subvariety of a projective variety of dimension $\dim(Z)+1$. Then $Z'=l \cdot Z$ is a divisor on this larger variety and so its structure (as a subscheme) and thus its arithmetic genus do depend on the embedding of $Z'$.
	
	\begin{rmk} \label{rmk:genus_H1}
		If $Z$ is a connected and reduced projective variety (e.g.\ if $Z$ is integral), recall that $\mathcal{O}_Z=\kk$ and thus, $\dim_{\kk}(H^0(Z, \mathcal{O}_Z))=1$ (see e.g.\ \cite[Cor.\ 3.3.21]{Qing_Liu-Algebraic_Geometry}). Moreover, $Z$ is a (possibly non-reduced, singular or reducible) complete intersection in $\Pn$ of positive dimension for some $n\geq 1$, then we also have $\dim_{\kk}(H^0(Z, \mathcal{O}_Z))=1$ (\cite[18.6.U]{Rising_Sea-Vakil}). In particular, $Z$ is a projective curve that is either connected and reduced or a complete intersection, then the arithmetic genus of $Z$ is just given by $p_a(Z)=\dim_{\kk}(H^1(Z, \mathcal{O}_Z))$.
	\end{rmk}
	In the following, the genus of a (possibly singular, reducible or non-reduced) curve $\Gamma$ will always refer to the arithmetic genus $p_a(\Gamma)=1-\chi(\mathcal{O}_{\Gamma})$ of $\Gamma$. Recall that if $\Gamma$ is a smooth, irreducible curve, then its arithmetic and geometric genus coincide.\\
	The following lemma generalizes the well-known equality $2p_a(\Gamma)-2=\Gamma \cdot (\Gamma + K_S)$ for smooth irreducible curves $\Gamma$ on a smooth surface $S$ to an arbitrary curve $\Gamma \sub S$ that is possibly singular, reducible or non-reduced (compare \cite[Exercice V.1.3]{Hartshorne}).
	\begin{lem} \label{lem:genus_formula_union_of_curves}
		Let $S$ be a smooth surface with canonical divisor $K_S$ and $\Gamma \sub S$ a curve. Then the arithmetic genus $p_a(\Gamma)$ of $\Gamma$ satisfies 
		$$ 2p_a(\Gamma)-2 = \Gamma \cdot (\Gamma + K_S).$$
		In particular, given a curve $\Gamma'\sub S$, the union $\Gamma \cup \Gamma'$ (seen as a divisor on $S$) satisfies
		$$ p_a(\Gamma + \Gamma')= p_a(\Gamma)+p_a(\Gamma')+\Gamma\cdot \Gamma'-1.$$
	\end{lem}
	\begin{proof}
		In (the proof of) \cite[Thm.\ V.1.6]{Hartshorne}, it is shown that any Cartier divisor $D\sub S$ satisfies $\chi(\mathcal{O}_X(D))=\frac{1}{2}D\cdot (D-K_S)+1+p_a(S)$. For the divisor $D=-\Gamma$, this equality yields 
		$$\chi(\mathcal{O}_X(-\Gamma))=\frac{1}{2}\cdot\Gamma\cdot (\Gamma+K_S)+1+p_a(S).$$
		By definition of the arithmetic genus, we have $p_a(S)=\chi(\mathcal{O}_S)-1$. Furthermore, we make use of the fact that $\mathcal{O}_X(-\Gamma) \simeq \mathcal{I}_\Gamma$ (\cite[Prop.\ II.6.18]{Hartshorne}), where $ \mathcal{I}_\Gamma$ denotes the ideal sheaf of $\Gamma$. The short exact sequence 
		$ 0 \rightarrow \mathcal{I}_\Gamma \rightarrow \mathcal{O}_S \rightarrow \mathcal{O}_{\Gamma} \rightarrow 0$ now implies that $\chi(\mathcal{O}_S)=\chi({\mathcal{O}_{\Gamma}}) + \chi({\mathcal{I}_{\Gamma}}).$ Putting everything together, we deduce that 
		$$ p_a(\Gamma)=1-\chi(\mathcal{O}_{\Gamma})=1-\chi(\mathcal{O}_S)+\chi(\mathcal{I}_{\Gamma})=-p_a(S)+\chi(\mathcal{O}_X(-\Gamma))=\frac{1}{2}\cdot\Gamma\cdot (\Gamma+K_S)+1.$$
		From this, it follows directly that
		\begin{align*}
			p_a(\Gamma + \Gamma')
			&=\frac{1}{2}(\Gamma + \Gamma')\cdot(\Gamma + \Gamma' + K_S)+1 \\
			& = \frac{1}{2}\cdot\Gamma\cdot (\Gamma+K_S) + \frac{1}{2}\cdot\Gamma'\cdot (\Gamma'+K_S) + \Gamma \cdot \Gamma' +1 \\
			& = p_a(\Gamma)+p_a(\Gamma')+\Gamma\cdot \Gamma'-1.
		\end{align*}
	\end{proof}
	We will mostly apply \Cref{lem:genus_formula_union_of_curves} for the case where $S$ is a smooth K3-surface and thus $K_S=0$. Then any curve $\Gamma \sub S$ satisfies $\Gamma^2 = 2p_a(\Gamma)-2$. We use this fact without always referring to  \Cref{lem:genus_formula_union_of_curves}. \par 
	Such a K3-surface $S$ will show up as the intersection of a hyperquadric and a hypercubic in $\Pfour$. To study curves on $S$, we are interested in curves that are equal to or contained in a complete intersection of $S$ with another hypersurface of $\Pfour$. The genus of such a (potentially reducible) complete intersection depends only on the degree of the intersected hypersurfaces as the next lemma shows.
	\begin{lem} \label{lem:genus_formula_complete_intersection}
		Let $\Gamma = H_1 \cap {\dots} \cap H_{n-1} \sub \Pn$ be a reduced curve that is a complete intersection of $n-1$ irreducible hypersurfaces $H_1,\dots,H_{n-1}\sub \Pn$ of degree $\deg(H_i)=d_i$ for all $i\in\{1,\dots,n-1\}$. Then 
		$$p_a(\Gamma)=\frac{1}{2}\prod_{j=1}^{n-1} d_j \left(-n-1+\sum_{i=1}^{n-1} d_i\right) +1.$$
	\end{lem}
	\begin{proof}
		Since $\Gamma$ is a complete intersection, we can apply \cite[Cor.\ 7.3.31]{Qing_Liu-Algebraic_Geometry} to obtain that $2p_a(\Gamma)-2=\deg(\omega_{\Gamma})$, where $\omega_{\Gamma}$ denotes the dualizing sheaf of $\Gamma$ (which is isomorphic to its canonical sheaf as $\Gamma$ is a complete intersection, see e.g.\ \cite[Exercise 6.4.11, Chapter 6]{Qing_Liu-Algebraic_Geometry}). \par 
		Applying iteratively the adjunction formula (\cite[Prop.\ 29.4.8]{Rising_Sea-Vakil}), we obtain that the dualizing sheaf $\omega_{\Gamma}$ of the complete intersection $\Gamma$ is given by $$\omega_{\Gamma} \simeq \mathcal{O}_{\Gamma}\left(\sum_{i=1}^{n-1} d_i - n-1\right).$$
		(Note that we can apply \cite[Prop.\ 29.4.8]{Rising_Sea-Vakil}, since $\Gamma$ is a projective Cohen-Macaulay scheme by \cite[Prop.\ II.8.23]{Hartshorne}.)
		Taking the degree on both sides gives the desired equality $$2p_a(\Gamma)-2= \deg(\omega_{\Gamma})=\deg(\Gamma)\cdot \left(-n-1+\sum_{i=1}^{n-1} d_i\right)=\prod_{j=1}^{n-1} d_j \left(-n-1+\sum_{i=1}^{n-1} d_i \right).$$

\end{proof}
\subsection{Genus of curves on hypersurfaces of a given degree} \label{subsection:genus_curves_hypersurfaces}
Given a smooth irreducible curve $C$ in $\Pfour$ of degree $d$ and genus $g$, it is in general hard to determine the dimension of the linear system of hypersurfaces of a given degree that contain $C$ without additional information about the curve. However, we are able to give a lower bound on the dimension that depends only on $g$ and $d$. In general, this lower bound does not have to be attained.
\begin{lem} \label{lem:dimension_linear_system_hypersurfaces}
	Let $e,n\geq 1$ and $C \sub \Pn$ be a smooth irreducible curve of genus $g$ and degree $d$. If $2g - 2 < e d$, then the projective dimension of the linear system of hypersurfaces in $\Pn$ of degree $e$ containing $C$ is at least equal to
	$$\binom{n + e}{e} + g - ed -2 .$$
	If  $2g - 2\geq e d$, we get the less specific bound that the projective dimension of the corresponding linear system is at least equal to
	$$\binom{n + e}{e}-2+\min\left\{g - ed,\: -\frac{1}{2} ed\right\}.$$
	For $n=4$, we obtain that the projective dimension of the linear system of\\
	$\begin{cases}
		hyperplanes & \\
		hyperquadrics & \\
		hypercubics
	\end{cases}$
	passing through $C$ is at least equal to
	$\begin{cases}
		3+g-d& \\
		13+g-2d& \\
		33+g-3d
	\end{cases}$ if 
	$\begin{cases}
		2g-2<d.& \\
		g-1<d. & \\
		2g-2<3d.
	\end{cases}$
\end{lem}
\begin{proof}
	Let $Q\in \kPn_e\setminus\{0\}$ be a homogeneous polynomial of degree $e$ such that $Q$ does not vanish on $C$. Consider the divisor $H\in \Div(\Pn)$ associated to the hypersurface $\VPn(Q)$. Then $D:=H|_C\in \Div(C)$ is a divisor of degree $\deg(D)=\deg(H)\cdot \deg(C)=ed$ inducing a $\kk$-linear map
	$$ \eta: \kPn_e \rightarrow \mathcal{L}(D),\quad P \mapsto \frac{P}{Q}\Big|_C.$$
	The projective dimension of the linear system of hypersurfaces in $\Pn$ of degree $e$ containing $C$ is one less than the vectorial dimension $\ker(\eta)$ of the linear system. We want to bound $\dim(\ker(\eta))$ using $\dim( \mathcal{L}(D))$. We first note that, by the adjunction formula, we have  
	$$\deg(K_C - D)=\deg(K_C) - \deg(D) = -2 +2g-ed,$$
	and by Riemann-Roch,
	$$ \dim(\mathcal{L}(D))=: \ell(D) =-g+\deg(D)+1+\ell(K_C-D)=-g+ed+1+\ell(K_C-D).$$
	If $2g-2<ed$, then $\deg(K_C - D)<0$ and thus, $\ell(K_C-D)=0$. Riemann-Roch then yields
	$\ell(D) =-g+ed+1.$ It follows that the projective dimension $\dim(\ker(\eta)) -1$ can be bounded by
	$$\dim(\ker(\eta)) -1 \geq \dim(\kPn_e) - \dim(\mathcal{L}(D)) -1 = \binom{n + e}{e} + g -ed -2.$$
	If $2g-2= ed$, then there are two possibilities: Either $\ell(K_C-D)=0$ and then the projective dimension can be bounded from below as before by
	$$ \binom{n + e}{e} + g -ed -2.$$ 
	Otherwise, $\ell(K_C-D)>0$, then the divisor $D$ is a special effective divisor. Using Clifford's theorem (\cite[Thm.\ IV.5.4]{Hartshorne}), we obtain $ \ell(D)\leq \frac{1}{2}\deg(D)+1 = \frac{1}{2} ed+1 $, which gives
	$$\dim(\ker(\eta)) -1 \geq \dim(\kPn_e) - \ell(D) -1 \geq \binom{n + e}{e} -\frac{1}{2}ed-2.$$
	
\end{proof}
The following lemma provides us with two useful facts about the smooth hyperquadric $Y\sub \Pfour$ we consider throughout the paper: $Y$ is irreducible and does not contain any plane. 
\begin{lem} \label{lem:smooth_quadric_irreducible_contains_no_plane}
	Let $Y \sub \Pn$ be a projective variety.
	\begin{enumlem}
		\item If $Y$ is smooth, $n\geq 2$ and the codimension of all irreducible components of $Y$ is at most $n/2$, then $Y$ is irreducible. \label{lem:smooth_implies_irred}
		\item If $Y$ is a smooth hyperquadric, $H$ a hyperplane and $n \geq 4$, then the hyperplane section $Y \cap H$ is irreducible. In particular, $Y$ does not contain any linear subspace of dimension $n-2$. \label{lem:hyperplane_section_irred}
	\end{enumlem}
\end{lem}
\begin{proof}
	The first statement follows immediately from the fact that any point contained in more than one irreducible component of $Y$ is singular and from the fact that the intersection of two irreducible components of $Y$ (or more generally, of two subsets of $\Pn$ with $n\geq2$, whose codimensions sum up to at most $n$) is non-empty if $n\geq 2$.\par
	The second statement can be checked in coordinates: Up to a change of coordinates, we can assume that $H=\VPn(x_0)$. Moreover, $Y$ is irreducible by \subcref{lem:smooth_implies_irred} (as all irreducible components of $Y$ are of codimension 1). So we can write $Y= \VPn(f)$, where $f = p_2 + p_1x_0 + p_0 x_0^2\in \kPn_2$ is an irreducible polynomial of degree 2 and the $p_i\in \kk[x_1,\dots,x_n]_i$ are homogeneous polynomials of degree $i\in \{0,1,2\}$. \par
	We want to show that $p_2$ is irreducible in $\kk[x_1,\dots,x_n]$. Suppose to the contrary that there are $\ell_1,\: \ell_2 \in \kk[x_1,\dots,x_n]_1$ with $p_2 = \ell_1  \ell_2$. Then $f\in (x_0,\:\ell_1,\:\ell_2,\:p_1)^2$, which implies that $\VPn(x_0,\:\ell_1,\:\ell_2,\:p_1)$ is contained in the singular locus of $Y$. 
If $n\geq 4$, then $\dim(\VPn(x_0,\:\ell_1,\:\ell_2,\:p_1)) \geq n-4 \geq 0$ contradicts the assumption that $Y$ is smooth. So $p_2$ is indeed irreducible in $\kk[x_1,\dots,x_n]$. Therefore, $(p_2,\:x_0)$ is a prime ideal in $ \kPn$ and $H \cap S = \VPn(f,x_0) = \VPn(p_2,x_0)$ is an irreducible variety. 
\end{proof}
\begin{rmk} \label{rmk:singular_hyperquadric_contains_plane}
By \Cref{lem:hyperplane_section_irred}, a hyperquadric $Y\sub \Pfour$ is singular if it contains a plane. The converse is also true: 
Given a quadratic form $f\in\kPfour_2$ defining $Y$, we can find a symmetric bilinear form $B:\kk^5\rightarrow \kk$ such that $f(x)=B(x,x)$ for all $x\in\kk^5$. Since $\textup{char}(\kk)\neq 2$, the symmetric bilinear form $B$ is diagonalizable (see e.g.\ \cite[Thm.\ 6.5]{Jacobson-Quadratic_form_diagonalizable}). So we may assume that, up to a change of coordinates, we can write $f(x)=B(x,x)=\sum_{i=0}^r x_i^2$ with $x=(x_0,\dots,x_4) \in \kk^5$ and $r\in\{0,\dots,4\} $. \par If $r\in \{0,\:1\}$, then $Y$ is reducible and contains the plane $\VPfour(x_0,\:x_1)$. If $r=2$, then the plane $\VPfour(x_0-\omega x_1,\:x_2)$ is contained in the singular quadric $Y=\VPfour(x_0^2+x_1^2+x_2^2)$, where $\omega \in \kk$ with $\omega^2=-1$. Similarly, $r=3$, then $Y=\VPfour(x_0^2+x_1^2+x_2^2+x_3^2)$ is singular and contains the plane $\VPfour(x_0-\omega x_1,\: x_2-\omega x_3)$. The case $r=4$ is the only one for which $Y$ is smooth.
\end{rmk}
The fact that a smooth hyperquadric in $\Pfour$ does not contain a plane implies that every integral cubic curve in $Y$ is rational, as the following lemma shows. This statement will become useful in some of the following proofs. 
\begin{lem} \label{lem:curve_degree3_genus1}
Let $Y\sub \Pfour $ be a smooth hyperquadric. If $D\sub Y$ is an integral curve of degree at most 3, then $D$ is smooth and $p_a(D) = 0$.
\end{lem}
\begin{proof}
It is a well-known fact that every line and every integral conic in projective space is isomorphic to $\Pone$ and is therefore smooth of genus 0. So suppose that $D$ is an integral (that is, an irreducible and reduced) curve in $Y$ of degree 3. Then the degree of $D$ is strictly smaller than the dimension of the ambient projective space $\Pfour$. Thus, $D$ must be degenerate, as is shown in e.g.\ \cite[p.\ 173]{Griffiths_Harris_Principles_of_Alg_Geom}. That is to say, $D$ is contained in a hyperplane, which allows us to see $D$ as an integral cubic curve in $\Pthree$. \par 
If $D$ were contained in a hyperplane of $\Pthree$, then $D$ would lie on a plane $P\sub \Pfour$. Since the smooth hyperquadric $Y$ cannot contain $P$ (\Cref{lem:hyperplane_section_irred}), Bézout's theorem would yield the contradiction $\deg(D)\leq 2$. \par 
So $D$ is a non-degenerate irreducible curve in $\Pthree$ of degree 3. It is hence projectively isomorphic to the rational normal curve, which is the image of the Veronese embedding $\Pone \hookrightarrow \Pthree$ (\cite[p.\ 179]{Griffiths_Harris_Principles_of_Alg_Geom}). We deduce that $D$ is isomorphic to $\Pone$ and, in particular, $D$ is smooth and $p_a(D)=0$.
\end{proof}
%
%
%

%
\section{Curves contained on a smooth hypercubic section} \label{section:hypercubic}
This section is devoted to the proof of the main theorems \ref{thm:A} and \ref{thm:B}. We consider the blowup $X$ of a smooth hyperquadric $Y$ in $\Pfour$ along a smooth irreducible curve $C$. The linear system $|-K_X|$ is given by the strict transforms of hypercubic sections $S$ of $Y$ through $C$ that are smooth along a general point of $C$, as we will show in \Cref{lem:anticanonical_divisor}. This implies, in particular, that there is an irreducible hypercubic $Z\sub \Pfour$ with $C\sub Z$ and $S=Y\cap Z$. \par
We will prove in \Cref{prop:weak_Fano_implies_contained_in_smooth_hypercubic_section} that a general surface $S$ of this form is smooth if $X$ is weak Fano. In that case, the smooth hypercubic section $S=Y\cap Z$ is a smooth sextic K3-surface. We will often use this fact without explicitly referring to the following remark. 
\begin{rmk} \label{rmk:Riemann_Roch_formula}
The smooth complete intersection $S$ of a hyperquadric and hypercubic in $\Pfour$ is a smooth K3 surface of degree 6. Conversely, every smooth K3-surface $S$ of degree 6 in $\Pfour$ is the smooth complete intersection of a hyperquadric with a hypercubic in $\Pfour$ (shown e.g.\ in \cite[Chapter 4.5, p.\ 592]{Griffiths_Harris_Principles_of_Alg_Geom}). \par 
In particular, we then have $K_S = 0$. So every curve $D\sub S$ satisfies $D^2=2p_a(D)-2$ (\Cref{lem:genus_formula_union_of_curves}).
Moreover, the Euler characteristic of $S$ is equal to $\chi(\mathcal{O}_S)=2$ (using Serre duality and the fact that $h^0(S,\mathcal{O}_S)=1$ and $h^1(S,\mathcal{O}_S)=0$). The Riemann-Roch inequality for surfaces hence reads $ \ell(D) + \ell(-D) \geq \frac{1}{2} D^2 + 2$ for all divisors $D$ on $S$ (\cite[Chapter 4, Section 2.4]{Shafarevich_Basic_Alg_Geom_1}).\par 
\end{rmk}
With the help of these facts about K3-surfaces, we will bound the genus of an arbitrary (potentially reducible) curve on a smooth sextic K3-surface in terms of its degree (\Cref{prop:bound_curves_in_K3surface_of_degree_6}). This result will be essential for proving both directions of \Cref{thm:A}. \par 
Another central result of this section is the proof that for every pair $(g,d)\in\mathcal{P}$, there exists a weak Fano threefold arising as the blowup of a smooth hyperquadric of $\Pfour$ along a smooth irreducible curve of degree $d$ and genus $g$ (\Cref{cor:existence_gdpairs_yielding_weak_Fano}). 
\subsection{Properties of the anticanonical linear system} \label{subsection:blowups}

Recall that a projective threefold $X$ is said to be weak Fano if its anticanonical divisor $-K_X$ is big and nef. This is equivalent to requiring that $(-K_X)^3>0$ and $-K_X\cdot \Gamma \geq 0$ for any effective curve $\Gamma \sub X$, as follows from the following lemma.
\begin{lem} 
Let $Z$ be a smooth irreducible projective variety of dimension at least 2 and $D,D_1,D_2\in \Div(Z)$ divisors on $Z$. Then:
\begin{enumlem}
\item If the base locus of $|D|$ does not contain an irreducible curve $\Gamma \sub Z$ with $D\cdot \Gamma<0$ (so in particular if $|D|$ is base-point-free), then $D$ is nef. \label{lem:base_point_free_implies_nef}
\item If $|D_1|$ and $|D_2|$ are base-point-free, then $|D_1+D_2|$ is base-point-free. \label{lem:sum_of_base_point_free_linear_systems}
\item If $D$ is nef, then $D$ is big if and only if $D^{\dim(Z)}>0$. \label{lem:nef_implies_big_iff_top_intersection_positive}
\end{enumlem}
\end{lem}

\begin{proof}
\subcref{lem:base_point_free_implies_nef}: Suppose by contradiction that $D$ is not nef. Then there is an irreducible curve $\Gamma \sub Z$ such that $D \cdot \Gamma <0$. By assumption, $\Gamma$ is not contained in the base locus of $|D|$. So there is a point $p\in \Gamma$ and an effective divisor $D'\in |D|$ with $p\not \in \supp(D')$. In particular, $\Gamma \not \sub \supp(D')$ and we must thus have $D'\cdot \Gamma \geq 0$, which contradicts $D'\cdot \Gamma =D\cdot \Gamma <0$.
\par
\subcref{lem:sum_of_base_point_free_linear_systems}: To prove the claim, it suffices to show that $\bl(D_1+D_2)\sub \bl(D_1) \cup \bl(D_2)$, where $\bl(D)$ denotes the base locus of the complete linear system $|D|$ on $Z$. \par
To do this, let $p\in \bl(D_1+D_2)$ with $p \not \in \bl(D_1)$, meaning there is a divisor $D'\in |D_1|$ with $p\not\in \supp(D')$. 
We want to show that $p\in \bl(D_2)$. Let $D'' \in |D_2|$ and consider $D:=D'+D''$. Since $D\in |D_1+D_2|$, we have $p\in \supp(D)=\supp(D')\cup \supp(D'')$. Now the fact that $p\not \in\supp(D')$ yields $p\in \supp(D'')$. As $D''\in|D_2| $ was arbitrary, it follows that $p\in \bl(D_2)$.\par
\subcref{lem:nef_implies_big_iff_top_intersection_positive}: This is exactly \cite[Thm.\ 2.2.16]{Lazarsfeld_positivity_in_AlgGeom_I}.
\end{proof}
By the previous lemma, the fact that a complete linear system is base-point-free implies that the corresponding divisor is nef. In our setting, this implication holds for the divisor $-K_X$ even under the weaker assumption that a suitable restriction of the linear system $|-K_X|$ is base-point-free.
\begin{lem} \label{lem:restricted_divisor_bpf_implies_nef}
Let $C \sub Y$ be a smooth irreducible curve of genus $g$ and degree $d$ lying on a smooth hyperquadric $Y\sub \Pfour$ and denote the blowup of $Y$ along $C$ by $\pi:X\rightarrow Y$.
Suppose that there is an irreducible hypersurface $Z\sub \Pfour$ of degree at most 3 such that $T=Y\cap Z$ is a normal surface containing $C$. \par 
Denote the strict transform of $T$ under $\pi$ by $\widetilde{T} \sub X$. 
If the restricted linear system $\big|-K_X|_{\widetilde{T}}\big|$ is base-point-free, then $-K_X$ is nef.
\end{lem}
\begin{proof}
Since the linear system $\big|-K_X|_{\widetilde{T}}\big|$ is assumed to be base-point-free, the restricted divisor $-K_X|_{\widetilde{T}}$ is nef (\Cref{lem:base_point_free_implies_nef}). \par 
Suppose to the contrary that $-K_X$ is not nef. Then there is an irreducible curve $\Gamma \sub X$ such that $(-K_X)\cdot \Gamma <0$. We want to show that $\Gamma \sub \widetilde{T}$. We first note that since $T$ is normal and $C\sub T$, the surface $T$ must be smooth along a general point of $C$. So $\widetilde{T}\in |lH_X-E|$ where $l=\deg(T)$ and $H_X$ denotes the pullback of a general hyperplane section of $Y$ under $\pi$. \par 
If $Z$ is a hypercubic, the inclusion $\Gamma \sub \widetilde{T}$ follows immediately from the fact that $\widetilde{T}\in |3H_X-E|=|-K_X|$ (\Cref{lem:formula_for_anticanonical_divisor}) and thus, $\widetilde{T} \cdot \Gamma= (-K_X)\cdot \Gamma <0$. If $Z$ is a hyperquadric, then $\widetilde{Z}\in |2H_X-E|$. There is an effective divisor $P\in |H_X|$ with $\Gamma \not \sub \supp(P)$ because $|H_X|$ is a base-point-free linear system. Since $\widetilde{T}+P \in |3H_X-E|= |-K_X|$, we then have $(\widetilde{T}+P)\cdot \Gamma= (-K_X)\cdot \Gamma <0$. Together with the fact that $\Gamma \not \sub \supp(P)$, this implies that $\Gamma \sub \widetilde{T}$. If $Z$ is a hyperplane, we can replace $P$ by an element $Q\in|2H_X|$ not containing $\Gamma$ in its support in order to argue analogously that $\Gamma \sub \widetilde{T}$.\par 
In all cases, we can restrict the inequality $(-K_X)\cdot \Gamma$ to $\widetilde{T}$ and obtain that $(-K_X|_{\widetilde{T}})\cdot \Gamma <0$. This contradicts the nefness of the divisor $-K_X|_{\widetilde{T}}$.
\end{proof}
We will show in the next subsection that the converse of the previous lemma holds under the additional assumption that $-K_X$ is big. In order to prove this, we need the following lemma. 
\begin{lem} \label{lem:bpf_iff_restriction_bpf}
Let $X$ be a smooth weak Fano threefold and $T\in |-K_X|$ an integral member. Then the restriction map 
$$ H^0(X, \mathcal{O}_X(-K_X)) \rightarrow H^0(X,\mathcal{O}_T(-K_X))$$
is surjective. \par
Moreover, if $X$ is weak Fano, then $|-K_X|$ is base-point-free if and only if $|-K_X|_T|$ is base-point-free (where the restriction $-K_X|_T$ is seen as the restriction of line bundles).
\end{lem}
\begin{proof}
Since the ideal sheaf of $T$ is given by $\mathcal{O}_X(-T)$ (\cite[Prop.\ II.6.18]{Hartshorne}), we have the short exact sequence 
$$ 0 \rightarrow \mathcal{O}_X(-T) \rightarrow \mathcal{O}_X \rightarrow \mathcal{O}_T \rightarrow 0.$$
Fixing a general member $T'\in |-K_X|$ with $T'\neq T$, we can tensor by $\mathcal{O}_X(T')$ to obtain the short exact sequence
$$ 0 \rightarrow \mathcal{O}_X(T'-T)=\mathcal{O}_X \rightarrow \mathcal{O}_X(T') \rightarrow \mathcal{O}_T(T') \rightarrow 0. $$
It induces a long exact sequence, which, in particular, gives us the short exact sequence
$$ H^0(X,\mathcal{O}_X(T'))  \overset{\phi}{\rightarrow} H^0(T,\mathcal{O}_T(T')) \rightarrow H^1(X, \mathcal{O}_X),$$ 
where the map $\phi$ is precisely the restriction map. \par 
Applying Kawamata-Viehweg vanishing (\cite[Thm.\ 4.3.1]{Lazarsfeld_positivity_in_AlgGeom_I}) to the big and nef line bundle $-K_X$, we know that $H^1(X, \mathcal{O}_X) = 0$. So by the exactness of the sequence, the restriction map $\phi$ is indeed surjective. \par 
This implies directly that the base loci $\bl(|-K_X|)$ and $\bl(|-K_X|_T|)$ of $|-K_X|$ and $|-K_X|_T|$ coincide since
$$ \bl(|-K_X|)=\bigcap_{R\in |-K_X|} \supp(R) = \bigcap_{\substack{R\in |-K_X|,\\ R\neq T}} \supp(\phi(R)) = \bigcap_{S\in |-K_X|_T|} \supp(S) =\bl(|-K_X|_T|).$$ 
In particular, $|-K_X|$ is base-point-free if and only if $|-K_X|_T|$ is base-point-free.
\end{proof}
Since we assume our hyperquadric $Y\sub \Pfour$ and the curve $C\sub Y$ to be smooth, the blowup $X$ of $Y$ along $C$ will always be a smooth threefold. So we can always apply the following lemma which, guarantees that a general member of $|-K_X|$ is a smooth K3-surface.
\begin{lem} \label{lem:weak_fano_dim_linear_system}
If $X$ is a smooth weak-Fano threefold, then a general member of $|-K_X|$ is a smooth K3-surface and 
$ \dim(|-K_X|)\geq 4.$
\end{lem}
\begin{proof}
By \cite[Prop.\ 1.2]{Weak_Fano_threefolds_I-Jahnke_Peternell_Radloff}, we have $h^0(X,-K_X)=\dim(|-K_X|)= \frac{(-K_X)^3}{2} +3$. As $-K_X$ is big, it follows directly that $(-K_X)^3>0$ and thus, $\dim(|-K_X|)\geq 4$.\par
As $X$ is smooth, \cite[Thm.\ 0.4 + 0.5]{Shin_3dimensional_Fano_varieties_with_canonical_singularities} implies that a general member of $ |-K_X|$ is a K3-surface, which is smooth along the base locus of $|-K_X|$. Together with Bertini's theorem and the assumption that $X$ is smooth, we deduce that a general member of $-K_X$ is also globally smooth.
\end{proof}
By the previous lemma, we know that a general member of $|-K_X|$ is a smooth K3-surface. In the following lemma, we specify what such a general member looks like: As mentioned in the beginning of this section, it is the strict transform of a hypercubic section $S$ of $Y$ that contains $C$ and is smooth along a general point of $C$. Note that a priori $S$ may be singular at certain special points of $C$. However, we will show in the next subsection that a general such $S$ is globally smooth.  
\begin{lem} \label{lem:anticanonical_divisor}
Let $C \sub Y$ be a smooth irreducible curve of genus $g$ and degree $d$ lying on a smooth hyperquadric $Y\sub \Pfour$. Denote the blowup of $Y$ along $C$ by $\pi:X\rightarrow Y$ and its exceptional divisor by $E=\pi^{-1}(C)$. Then:
\begin{enumlem}
\item \label{lem:formula_for_anticanonical_divisor}
In $\Pic(X)$, we have $-K_X = 3 H_X - E$, where $H_X\sub X$ denotes the pullback of a general hyperplane section of $Y$ under $\pi$. In particular, a general member of the linear system $|-K_X|$ is the strict transform of a hypercubic section of $Y$ that contains $C$ and is smooth along a general point of $C$.
\item \label{lem:cube_of_K_X}
The anticanonical divisor $-K_X$ of $X$ satisfies $(-K_X)^3 = 52-6d+2g.$
\item \label{lem:K_X_squared_times_E}
On $X$, we have $(-K_X)^2 \cdot E = 3d-2g+2$.
\item \label{lem:3n+1_secants} Let $\Gamma \sub Y$ be an irreducible curve of degree $n$ and denote its strict transform under $\pi$ by $\widetilde{\Gamma}\sub X$. If $\Gamma$ satisfies $E \cdot \widetilde{\Gamma}  =3n$, then $\widetilde{\Gamma}\cdot K_X=0$. If $\Gamma$ is a $3n+1$-secant of $C$, meaning that $E \cdot \widetilde{\Gamma}  \geq 3n+1$, then $-K_X$ is not nef. 
\end{enumlem}
\end{lem}
\begin{proof}
Considering a general hyperplane section $H_Y$ of $Y$ and denoting $H_X=\pi^*(H_Y)$, the adjunction formula yields $$K_Y = (K_{\Pfour}+Y)|_Y = (-5 +\deg(Y)) \: H_Y = -3 H_Y.$$
It is a classical fact that we then have
$K_X = \pi^*(K_Y) + (\textup{codim}_Y(C)-1)E = -3 H_X + E$.
This gives \subcref{lem:formula_for_anticanonical_divisor} and implies that 
$$
(-K_X)^3 = 27 H_X^3 - 27H_X^2 \cdot E + 9 H_X \cdot E^2 - E^3, \quad (-K_X)^2 \cdot E = 9H_X^2\cdot E - 6 H_X \cdot E^2 + E^3.
$$
By the projection formula, we have $H_X^3 = H_Y^3 = \deg(Y) = 2$. Moreover, \cite[Lemma 2.2.14]{Shafarevich_Parshin_AlgGeom_Fano_varieties} gives $H_X^2 \cdot E = 0$, $H_X \cdot E^2 = -\pi^*(H_Y) \cdot \pi^*(C) = - H_Y \cdot C = -d $ and $E^3 = -2g+2+K_Y\cdot C = -2g+2 -3d$. Altogether, this yields
\begin{gather*}
(-K_X)^3 = 27\cdot 2 -9d+2g-2+3d = 52 - 6d +2g, \\
(-K_X)^2 \cdot E = (-6)(-d) -2g+2-3d = 3d-2g+2.
\end{gather*}
This gives \subcref{lem:cube_of_K_X} and \subcref{lem:K_X_squared_times_E}. \par
Lastly, suppose there is an irreducible curve $\Gamma \sub Y$ of degree $n:=\deg(\Gamma) = H_Y \cdot \Gamma$. Then it follows from \subcref{lem:formula_for_anticanonical_divisor} and the intersection form on $\Pic(X)$ that
$$ -K_X \cdot \widetilde{\Gamma} = \big(3 \pi^*(H_Y) - E\big) \cdot \big(\pi^*(\Gamma) - (E\cdot \widetilde{\Gamma}) E\big) = 3 H_Y \cdot \Gamma - E \cdot \widetilde{\Gamma}= 3n - E\cdot \widetilde{\Gamma}.$$
This directly implies \subcref{lem:3n+1_secants}.
\end{proof}
Recall that in the notation of \Cref{lem:anticanonical_divisor}, an irreducible curve $\Gamma \sub Y$ is said to be an $m$-secant of $C$ for some $m\in \Z_{>0}$ if $\widetilde{\Gamma}\cdot E \geq m$. Also note that if $C,\Gamma \sub S$ for some surface $S\sub Y$, then $\widetilde{\Gamma}\cdot E=C\cdot \Gamma$, where the latter intersection is seen as the intersection of the divisors $C, \Gamma \in \Pic(S)$.
\subsection{Weak Fano implies that the curve lies on a smooth hypercubic section} \label{subsection:anticanonical_system_base_point_free}
The main goal of this section will be to prove that if $X$ is weak Fano, then the curve $C$ is contained in a smooth hypercubic section of $Y$. Such a surface is a smooth K3-surface of degree 6 (\Cref{rmk:Riemann_Roch_formula}).\par
By \Cref{lem:formula_for_anticanonical_divisor} and \Cref{lem:weak_fano_dim_linear_system}, a general member of $|-K_X|$ is a smooth K3-surface and  the strict transform of a hypercubic section of $Y$ that contains $C$ and is smooth along a general point of $C$. We can use the smoothness of such a strict transform "above in the blowup" to deduce the smoothness of a general member "below" in the linear system of hypercubic sections of $Y$ through $C$. \par 
The key to proving this lies in showing that $X$ is weak Fano if and only if $-K_X$ is big and the linear system $|-K_X|$ is base-point-free. One direction is given by the fact that a divisor is nef if the corresponding complete linear system is base-point-free (\Cref{lem:base_point_free_implies_nef}). The other direction is proved in the following proposition. 
\begin{prop} \label{prop:weak_Fano_implies_base_point_free}
Let $C \sub Y$ be a smooth irreducible curve of genus $g$ and degree $d$ lying on a smooth hyperquadric $Y\sub \Pfour$. Denote the blowup of $Y$ along $C$ by $\pi:X\rightarrow Y$. If $X$ is weak Fano, then $|-K_X|$ is base-point-free.
\end{prop}
\begin{proof}
Let $T\in |-K_X|$ be a general member. Then $T$ is a smooth K3-surface (\Cref{lem:weak_fano_dim_linear_system}). Due to \Cref{lem:bpf_iff_restriction_bpf} and the assumption that $X$ is weak Fano, it suffices to show that the restricted linear system $|-K_X|_T|$ is base-point-free. We suppose to the contrary that $|-K_X|_T|$ is not base-point-free and want to derive a contradiction.\par 
We first note that the fact that $-K_X$ is nef implies directly that $-K_X|_T$ is also nef. Then the bigness of $-K_X|_T$ follows from the bigness of $-K_X|_T$ because of \Cref{lem:nef_implies_big_iff_top_intersection_positive} and the inequality $$(-K_X|_T)^2 = (-K_X)\cdot(-K_X)\cdot T = (-K_X)^3 >0.$$
That is to say, $-K_X|_T$ is a big and nef linear divisor on the smooth K3-surface $T$ such that the linear system $|-K_X|_T|$ is not base-point-free. By (the proof of) \cite[Lemma 2.1]{Shin_3dimensional_Fano_varieties_with_canonical_singularities}, there must then be a smooth elliptic curve $e\sub T$, a rational (-2)-curve $r\sub T$ and $b\geq 2$ such that $-K_X|_T=be+r$ and $e\cdot r=1$. Note that we have $e^2=2p_a(e)-2=0$ as $T$ is a smooth K3-surface (\Cref{lem:genus_formula_union_of_curves}).\par 
On the one hand, we calculate that
\begin{equation*} \tag{$\circ_1$} \label{eq:relation_d_g_b_1}
3d-2g+2\overset{\textup{\ref{lem:K_X_squared_times_E}}}{=} (-K_X)^2\cdot E= (-K_X)\cdot T \cdot E = (-K_X|_T)\cdot E = be\cdot E+r\cdot E.
\end{equation*}
On the other hand, we can also relate $b,\:g$ and $d$ by
\begin{equation*} \tag{$\circ_2$} \label{eq:relation_d_g_b_2}
52-6d+2g\overset{\textup{\ref{lem:cube_of_K_X}}}{=}(-K_X)^3 = (-K_X|_T)^2=(be+r)^2=2b-2.
\end{equation*}
To obtain more conditions on the intersection $e\cdot E$, we make the following observation (which will also become relevant later in the proof):\par 
Recall that in $\Pic(X)$, we have $T=-K_X=3H_X-E$ , where $H_X=\pi^*(H_Y)$ denotes the pullback of a general hyperplane section $H_Y$ of $Y$ under $\pi$ (\Cref{lem:formula_for_anticanonical_divisor}). This shows that $T$ is the strict transform of a hypercubic section $S$ of $Y$ through $C$ and that $S$ is smooth along a general point of $C$. In particular, $S$ is singular at maximally finitely many points of $C$. The support of the restriction $T|_E$ is thus the union of a section and at most finitely many fibers of $\pi|_E:E\rightarrow C$. \par 
We can therefore deduce that: If the irreducible curves $r,e\sub T$ are contained in $E$, then they must either be a section or a fiber of $\pi|_E$. \par 
This observation allows us to show that $e$ is birational to an integral elliptic curve $e'\sub Y$ satisfying $e\cdot H_X = \deg(e')$: If $e\not \sub E$, then $e$ is the strict transform of an integral elliptic curve $e'\sub Y$ and we get directly that $\deg(e')=e'\cdot H_Y=e\cdot H_X$. Otherwise, if $e\sub E$, then $e$ cannot be a fiber of $\pi|_E$ as all fibers are rational. So $e$ must be a section and is hence birational to the curve $e':=\pi_*(e)=C$. Due to the birationality, $e'=C$ must then also be an integral elliptic curve and we obtain $e\cdot H_X=e\cdot \pi^*(H_Y)=\pi_*(e)\cdot H_Y=C\cdot H_Y=\deg(C)$. \par
We have now proven that we can indeed always find an integral elliptic curve $e'\sub Y$ such that $e\cdot H_X =\deg(e')$. Noting furthermore that $Y$ cannot contain an integral elliptic curve of degree 3 (\Cref{lem:curve_degree3_genus1}), we deduce that 
\begin{equation} \label{eq:deg_pi_e} \tag{$\circ_3$}
e\cdot H_X=\deg(e')\geq 4.
\end{equation}
Observing that $e\cdot (-K_X)=e\cdot (-K_X|_T)$ (as $e\sub T$), we find that the degree of $e'$ must also satisfy
\begin{equation*} \tag{$\circ_4$} \label{eq:relation_eE_degpi}
3\deg(e')-e\cdot E=3 e \cdot H_X - e\cdot E= e\cdot (-K_X)=e\cdot (-K_X|_T)= e \cdot (be+r) = 1.
\end{equation*}
Together with the lower bound on $\deg(e')$, we obtain
\begin{equation} \label{eq_eE} \tag{$\circ_5$} 
e\cdot E \overset{\eqref{eq:relation_eE_degpi}}{=}3\deg(e') -1 \overset{\eqref{eq:deg_pi_e}}{\geq} 11. 
\end{equation}
We now show that we must have equalities in \eqref{eq:deg_pi_e} and \eqref{eq_eE}, as well as $(g,d)\in\{(2,9),(5,10)\}$, and we will derive a contradiction from this. In order to do this, we make a case distinction depending on whether $r$ is section of $\pi|_E$ or not. In both cases, we obtain an additional inequality, which, together with the previous equalities, will give us the desired $(g,d)$-pairs. Here we will again make use of the previously deduced observation that if $r\sub E$ or $e\sub E$, then they are either a fiber or a section of $\pi|_E$.
\begin{itemize}
\item \underline{Case 1: $r$ is not a section of $\pi|_E$.} If $r\sub E$, then (by the assumption that $r$ is not a section) $r$ must be a fiber of $\pi|_E$ and thus, $r\cdot E=-1$ (\cite[Lemma 2.2.14]{Shafarevich_Parshin_AlgGeom_Fano_varieties}). Otherwise if $r\not \sub E$, then $r\cdot E\geq 0$ . That is to say, in our first case, we have
\begin{equation} \label{eq:case_1} \tag{$\bullet_1$}
	r\cdot E \geq -1.
\end{equation} 

Equation \eqref{eq:relation_d_g_b_2} together with the fact that $b\geq 2$ implies that $3d\leq 25+g$ and hence, $3d+3-2g\leq 28-g\leq 28.$ But then \eqref{eq:relation_d_g_b_1} and \eqref{eq:case_1} become
$$b\leq\frac{3d+3-2g}{e\cdot E} \overset{\eqref{eq_eE}}{\leq}\frac{28}{11 }, $$
which yields $b=2$. Plugging this into \eqref{eq:relation_d_g_b_2} gives $g=3d-25$. \par
On the one hand, this provides the lower bound $d\geq \frac{25}{3}$. On the other hand, plugging $-2g=50-6d$ into \eqref{eq:relation_d_g_b_1}, together with the inequalities \eqref{eq:case_1} and \eqref{eq_eE}, we obtain that $d\leq\frac{31}{3}$. This upper bound, together with the lower bound for $d$ and the fact that $g=3d-25$, shows that the only possible $(g,d)$-pairs are $(g,d)=(5,10)$ and $(g,d)=(2,9)$. \par 
Moreover, equation \eqref{eq:relation_d_g_b_1} and \eqref{eq:case_1} then imply that $e\cdot E \leq 13$. Since we also have that $e\cdot E\geq 11$ and $e\cdot E = 3\deg(e')-1$ by \eqref{eq_eE}, the only possibility is that $e\cdot E =11$ and $\deg(e')=4$. \par 
\item \underline{Case 2: $r\sub E$ is a section of $\pi|_E$.} In this case, the elliptic curve $e$ cannot also be a section, since a section of $\pi|_E:E\rightarrow C$ is birational to $C$ and thus irreducible. But $e$ can also not be a fiber of $\pi|_E$, since the fibers of $\pi|_E$ are rational. Hence, the elliptic curve $e$ cannot be contained in $E$ and must be the strict transform of an integral elliptic curve $e'\sub Y$. \par 
Denote $T_1:=T$. Since $\dim(|-K_X|)\geq 4$ (\Cref{lem:weak_fano_dim_linear_system}), we can find another general member $T_2\in |-K_X|$ with $T_2\neq T_1$. By \Cref{lem:formula_for_anticanonical_divisor}, the $T_i$ are the strict transforms of hypercubic sections $Y\cap Z_i$ for some hypercubics $Z_i\sub \Pfour$ through $C$ and $i\in\{1,2\}$. The fact that $be$ is contained in the support of $-K_X|_{T_1}=T_2|_{T_1}$ and $e\not \sub E$ implies that $C+be'$ must be contained in the support of the intersection $Y\cap Z_1\cap Z_2$ (both seen as divisors on $Y\cap Z_1$). In particular, we have
\begin{equation} \label{eq:case_2} \tag{$\bullet_2$}
	d+b \cdot \deg(e')\leq \deg(Z_1)\cdot \deg(Z_2)\cdot \deg(Y)=18.
\end{equation} 
Plugging this into \eqref{eq:relation_d_g_b_2}, while also using \eqref{eq_eE}, yields $b\leq \frac{27-g}{e\cdot E}\leq \frac{27}{11}$. The only possibility is thus $b=2$. From this and \eqref{eq:relation_d_g_b_2}, it follows that $g=3d-25$. On the one hand, this gives the lower bound $d\geq \frac{25}{3}$. On the other hand, plugging $\deg(e')\geq4$ into \eqref{eq:case_2}, we get that $9\leq d\leq10$. From this and \eqref{eq:case_2}, we also deduce that $\deg(e')\leq \frac{9}{2}$. So the only solutions to the previous (in)equalities in the second case are also $\deg(e')=4$, $e\cdot E=11$ and $(g,d)\in\{(2,9),\:(5,10)\}$.
\end{itemize}
In both cases, we have shown that we must have $(g,d)\in\{(2,9),\:(5,10)\}$, $\deg(e')=4$ and $e\cdot E=11$. From this, we we will be able to deduce a final contradiction. We first note that since $g\in\{2,5\}$, the curve $C$ is not elliptic. We can thus not be in the case that $e$ is a section of $\pi|_E:E\rightarrow C$ and $e'=C$. So $e$ must be the strict transform of an integral elliptic curve $e'\sub Y$. \par 
We first argue that $e'$ has to be contained in a hyperplane. Suppose to the contrary that the integral quartic curve $e'\sub \Pfour$ were non-degenerate. Then it follows from \cite[p.\ 179]{Griffiths_Harris_Principles_of_Alg_Geom} that $e'$ is the image of a Veronese embedding $\Pone \hookrightarrow \Pfour$. In particular, $e'$ must be smooth. Then we can apply \Cref{lem:dimension_linear_system_hypersurfaces} to deduce that $e'$ lies on a hyperplane, contradicting our previous assumption. This shows that $e'$ must always be degenerate, meaning $e'$ is indeed contained in a hyperplane $H\sub \Pfour$. \par 
We now show that $C$ cannot be contained in $H$ which allows us to derive a final contradiction. To do this, recall that our general member $T\in|-K_X|$ is a smooth K3-surface and the strict transform of a hypercubic section $Y\cap Z$ containing $C$ for some hypercubic $Z\sub \Pfour$. Note that $Z$ must be irreducible as $Y\cap Z$ is smooth and hence irreducible (\Cref{lem:smooth_implies_irred}). If we had $C\sub H$, Bézout's theorem would imply that $\deg(C)\leq \deg(Y)\cdot \deg(Z)\cdot \deg(H)=6$ contradicting $d=\deg(C)\in\{9,10\}$. (We can apply Bézout's theorem since $Y\cap Z\not \sub H$. Otherwise, we have $Y\cap Z=Y\cap H$ and $T$ would not be a K3-surface.) \par 
That is to say, the curve $C$ cannot lie on the hyperplane $H$. So we can apply once again Bézout's theorem (together with the fact that $e'\sub H$ and that $e$ is the strict transform of $e'$) to find the final contradiction
$$ 11=e\cdot E \leq e'\cdot C \leq H \cdot C \leq \deg(C) \in \{9,10\}.$$
\end{proof}

The previous proposition allows us to deduce equivalent characterizations of $-K_X$ being big and nef. They will be used in later proofs to show that $X$ is weak Fano.
\begin{cor} \label{cor:weak_Fano_big_nef_basepoinfree}
Let $C \sub Y$ be a smooth irreducible curve of genus $g$ and degree $d$ lying on a smooth hyperquadric $Y\sub \Pfour$ and denote the blowup of $Y$ along $C$ by $\pi:X\rightarrow Y$. Then the following statements are equivalent:
\begin{enumcor}
\item $X$ is weak Fano. \label{cor:X_weak_Fano}
\item $26-3d+g>0$ and $-K_X$ is nef. \label{cor:X_big_nef}
\item $26-3d+g>0$ and $|-K_X|$ is base-point-free. \label{cor:X_basepointfree_and_big}
\item $26-3d+g>0$ and there is an irreducible hypercubic $Z\sub \Pfour$ such that $T=Y\cap Z$ is a normal surface containing $C$ and the restricted linear system $\big|-K_X|_{\widetilde{T}}\big|$ is base-point-free, where $\widetilde{T}$ denotes the strict transform of $T$.
\label{X_restricted_linear_system_basepointfree_and_big}
\end{enumcor}
\end{cor}
\begin{proof}
By definition, $X$ is weak Fano if and only if $-K_X$ is big and nef. Due to \Cref{lem:nef_implies_big_iff_top_intersection_positive} and the fact that $\dim(X)=3$, this condition is equivalent to saying that $-K_X$ is nef and $(-K_X)^3>0$. The equivalence of \subcref{cor:X_weak_Fano} and \subcref{cor:X_big_nef} now follows directly from the equality $(-K_X)^3=52-6d+2g$ (\Cref{lem:cube_of_K_X}). \par 
The implication \subcref{cor:X_weak_Fano} $\Rightarrow$ \subcref{cor:X_basepointfree_and_big} results immediately from \Cref{prop:weak_Fano_implies_base_point_free}. Conversely, the direction \subcref{cor:X_basepointfree_and_big} $\Rightarrow$ \subcref{cor:X_big_nef} is a direct consequence of \Cref{lem:base_point_free_implies_nef}.\par 
We have shown in \Cref{lem:restricted_divisor_bpf_implies_nef} that \subcref{X_restricted_linear_system_basepointfree_and_big} implies \subcref{cor:X_big_nef}. For the converse implication \subcref{cor:X_big_nef} $\Rightarrow$ \subcref{X_restricted_linear_system_basepointfree_and_big}, we can make use of the fact that we have already argued that \subcref{cor:X_big_nef}, $\Rightarrow$ \subcref{cor:X_basepointfree_and_big} and then the result follows from \Cref{lem:bpf_iff_restriction_bpf}.
\end{proof}
We are now able to show that the assumption of \Cref{thm:A} about the curve $C$ being contained in a smooth hypercubic section of $Y$ is a necessary condition. 
\begin{prop} \label{prop:weak_Fano_implies_contained_in_smooth_hypercubic_section}
Let $C \sub Y$ be a smooth irreducible curve lying on a smooth hyperquadric $Y\sub \Pfour$. Suppose that the blowup $\pi:X\rightarrow Y$  of $Y$ along $C$ is a weak Fano threefold. Let $\dd|_Y$ denote the linear system, whose general member is a hypercubic section of $Y$ which contains $C$ and is smooth along a general point of $C$. \par 
Then the general member of $\dd|_Y$ is globally smooth and $\dim(\dd|_Y)\geq 4$. That is to say, there are at least four irreducible hypercubics $Z_1,\dots, Z_4\sub \Pfour$ such that for all $i\in\{1,\dots,4\}$, the hypercubic sections $Z_i\cap Y$ are smooth, contain $C$ and are pairwise not linearly equivalent (as divisors on $Y$).
\end{prop}
\begin{proof}
Let $S\in \dd_Y$ be a general member. By assumption, $S$ is smooth along a general point of $C$ and $S=Y\cap Z$ for some hypercubic $Z\sub \Pfour$ containing $C$. The strict transform $\widetilde{S}$ of $S$ satisfies thus $\widetilde{S} = 3H_X - E$, where $H_X$ denotes the pullback of a general hyperplane section of $Y$ under $\pi$. Together with \Cref{lem:formula_for_anticanonical_divisor}, this is equivalent to saying that $|\widetilde{S}|=|-K_X|$. \par 
The statement $\dim(\dd|_Y)\geq 4$ now follows directly from the fact that $\dim(|-K_X|)\geq 4$ (\Cref{lem:weak_fano_dim_linear_system}). \par
In order to show that $S$ is globally smooth, we make use of the fact that the linear system $|-K_X|$ is base-point-free by \Cref{prop:weak_Fano_implies_base_point_free}. \par 
On the one hand, this implies that a general member of $|-K_X|$ is smooth by Bertini's theorem and the fact that $X$ is smooth. In particular, the strict transform $\widetilde{S}\in |-K_X|$ must be smooth. So $S$ has to be smooth outside of $C$. It is left to show that $S$ is also smooth along $C$.\par
On the other hand, the fact that $|-K_X|$ is base-point-free shows that the restricted linear system $|-K_X |_{|E}$ is also a base-point-free linear system on the smooth surface $E$. Hence, again by Bertini's theorem, a general member of $|-K_X |_{|E}$ is smooth. In particular, the restriction $\widetilde{S}|_E\in |-K_X |_{|E}$ is smooth.\par 
Recall that $E$ is a ruled surface over $C$ via the restriction $\pi|_E:E\rightarrow C$ and hence, the Picard group of $E$ is generated by sections and fibers of $\pi|_E$ (\cite[Prop.\ V.2.3]{Hartshorne}). Since $\widetilde{S}|_E$ is an effective divisor on $E$, we can find a section $s\sub E$ and finitely many fibers $f_1,\dots,f_r\sub E$ such that $\widetilde{S}|_E=as+\sum_{i=1}^r b_i f_i$ for $a,\:b_i\in \Z_{\geq0}$. Moreover, we have $f_i\cdot s=1$ and $f_i\cdot f_j=0$ for all $i,\:j\in\{1,\dots,r\}$ (again by\cite[Prop.\ V.2.3]{Hartshorne}). We calculate
$$ \widetilde{S}\cdot f_1 = (-K_X)\cdot f_1 \overset{\textup{\ref{lem:formula_for_anticanonical_divisor}}}{=} (3H_X-E)\cdot f_1 = -E\cdot f_1 = 1.$$
As $f_1\sub E$, this remains valid for the restriction to $E$ and thus,
$$1=\widetilde{S}|_E \cdot f_1= \left(as+\sum_{i=1}^r b_i f_i\right)\cdot f_1=a.$$
This means that $s$ is an irreducible component of $\widetilde{S}|_E$. If $S$ were now singular in some point $p\in C \sub S$, then the fiber $f=\pi^{-1}(\{p\})\sub E$ would also be an irreducible component of $\widetilde{S}|_E$. But then the intersection point of $f$ and $s$ (which exists since $f\cdot s =1$ in $\Pic(E)$) would be a singular point of $\widetilde{S}|_E$. This contradicts the fact that $\widetilde{S}|_E$ is smooth. It follows that $S|_C$ and hence $S$ must be smooth. \par 
Lastly, this implies that the hypercubics $Z\sub \Pfour$ yielding a smooth hypercubic section $Y\cap Z\in \dd|_Y$ must be irreducible since $Y\cap Z$ is smooth and thus irreducible (\Cref{lem:smooth_implies_irred}). \par 
\end{proof}
%
%
%
%
%
%

%
\subsection{Bounding the genus of curves on a smooth hypercubic section} \label{subsection:bounding_genus}
In \cite[Thm. 1.1]{Knutsen-Smooth_curves_on_projective_K3_surface}, a necessary and sufficient condition for the existence of a smooth K3-surface in $\mathbb{P}^{n+1}$ of degree $2n$ containing an integral smooth curve of a given degree $d$ and genus $g$ is derived by bounding $g$ in terms of $d$ and $n$. For the case $n=3$, we obtain the following result:

\begin{prop}
\label{thm:knutsen_existence_curves_in_K3_surface_degree6}
Let $g\geq 0$ and $d\geq 1$ be two integers. There exists a hyperquadric $Q\sub\Pfour$ and a hypercubic $Z\sub \Pfour$ such that their intersection $S=Q\cap Z$ is a smooth K3-surface and contains an irreducible smooth curve $C$ of degree $d$ and genus $g$ if and only if one of the following cases applies:
\begin{enumprop}
\item $g=\frac{d^2}{12}+1$,\label{thm:knutsen_existence_curves_in_K3_surface_degree6_1}
\item $g=\frac{d^2}{12}+\frac{1}{4}$, \label{thm:knutsen_existence_curves_in_K3_surface_degree6_2}
\item $g<\frac{d^2}{12}$ and $(g,d)\neq (4,7)$. \label{thm:knutsen_existence_curves_in_K3_surface_degree6_3}
\end{enumprop}
Furthermore, we can assume that $\Pic(S)=\Z H_S$ in case \subcref{thm:knutsen_existence_curves_in_K3_surface_degree6_1}, and that $\Pic(S)=\Z H_S \oplus \Z C$ in the other two cases, where $H_S$ denotes a general hyperplane section of $S$.
\end{prop}
\begin{proof}
Recall that smooth K3-surfaces in $\Pfour$ of degree 6 are precisely the smooth intersections of a hyperquadric and a hypercubic in $\Pfour$ (\Cref{rmk:Riemann_Roch_formula}).  \cite[Thm. 6.1.2]{Knutsen-Smooth_curves_on_projective_K3_surface} now shows that the conditions \subcref{thm:knutsen_existence_curves_in_K3_surface_degree6_1}-\subcref{thm:knutsen_existence_curves_in_K3_surface_degree6_3} are equivalent to the existence of a smooth K3-surface $S$ of degree 6 containing an irreducible smooth curve $C$ of degree $d$ and genus $g$. The assumption about the Picard group of $S$ follows from the case $n=3$ of \cite[Thm. 1.1]{Knutsen-Smooth_curves_on_projective_K3_surface}.
\end{proof}
If we assume additionally that the hyperquadric $Q$ in \Cref{thm:knutsen_existence_curves_in_K3_surface_degree6} is smooth, \Cref{prop:bound_curves_in_K3surface_of_degree_6} shows that we cannot have case \subcref{thm:knutsen_existence_curves_in_K3_surface_degree6_2}. \par 
Before attending to said proposition, we prove two auxiliary lemmas. The first one gives sufficient conditions for a divisor on a smooth K3-surface to be effective.
\begin{lem} 
Let $S\sub \Pn$ be a smooth K3-surface for some $n\in \Z_{>0}$, $D\in \Pic(S)$ a divisor and $H_S$ a general hyperplane section of $S$. 
\begin{enumlem}
\item If $d:=H_S\cdot D\geq 0$ and $D^2 > -4$, then $D\sub S \sub \Pn$ is an effective (but potentially non-reduced or reducible) curve of degree $d$ and genus $\frac{1}{2}D^2+1$. \label{lem:Riemann_Roch_effective_divisor}
\item If $H_S \cdot D = 0$ and $D^2>-4$, then $D=0$ is the zero divisor. \label{lem:Hodge_index_thm}
\end{enumlem}
\end{lem}
\begin{proof}
Since $S$ is a smooth K3-surface, the Riemann-Roch inequality for surfaces (\Cref{rmk:Riemann_Roch_formula}) yields
$$\ell(D)+\ell(-D) \geq \frac{1}{2}D^2 + \chi(\mathcal{O}_S) > -\frac{4}{2} + 2=0.$$ 
This implies that $D$ or $-D$ is effective because otherwise, $\ell(D)=\ell(-D)=0$. \par
Since $H$ is an ample divisor, we have $H\cdot\Gamma >0$ for all effective, non-zero divisors $\Gamma$ on $S$ by the Nakai-Moishezon criterion. In particular, if $H\cdot D=0$ and $D$ or $-D$ is effective, we must have $D=0$. \par 
If $H\cdot D >0$, then $D$ must be effective and non-zero. The degree of $D$ is therefore given by $H_S\cdot D=d$. The equality $p_a(D)=\frac{1}{2} D^2 +1$ for the genus of $D$ follows directly from the formula of \Cref{lem:genus_formula_union_of_curves} and the fact that $K_S=0$.
\end{proof}
Recall that in \Cref{lem:curve_degree3_genus1}, we have shown that every integral curve of degree at most 3 on a smooth hyperquadric in $\Pfour$ is of genus 0. We can generalize this result to an arbitrary, potentially reducible or non-reduced, curve of degree at most 3 lying on a smooth sextic K3-surface that is contained in a smooth hyperquadric in $\Pfour$.
\begin{lem} \label{lem:divisor_degree_3_genus_1}
Let $Q\sub \Pfour$ be a smooth hyperquadric and $S\sub Q$ a smooth K3-surface of degree 6. Any (possibly singular or non-integral) curve $D\in\Pic(S)$ of degree at most 3 satisfies $p_a(D)\leq 0$.
\end{lem}
\begin{proof}
If $D$ is integral, we have seen in \Cref{lem:curve_degree3_genus1} that $p_a(D)=0$ (and, in fact, $D$ is smooth). In particular, any line or integral conic on $S$ has genus 0 and hence self-intersection $-2$ in $S$ by \Cref{lem:genus_formula_union_of_curves}. Also recall the formula $p_a(\Gamma + \Gamma')=p_a(\Gamma)+p_a(\Gamma')+\Gamma \cdot \Gamma'-1$ from \Cref{lem:genus_formula_union_of_curves} for two curves $\Gamma, \Gamma'\sub S$. \par
Supposing that $D$ is now not integral, we are in one of the following cases:
\begin{itemize}
\item $D=mL$ for a line $L\sub S$ and $m\in \{2,3\}$: Then $$p_a(D)=mp_a(L)+\left(\sum_{i=1}^{m-1}i\right)L^2-(m-1) = \frac{(m-1)m}{2}\cdot (-2) -(m-1)=1-m^2<0.$$
\item $D=mL_1+L_2$ for two distinct lines $L_1,L_2 \sub S$ and $m\in\{1,2\}$: The line $L_1$ is contained in three distinct hyperplanes $H, H',H''\sub \Pfour$ (\Cref{lem:dimension_linear_system_hypersurfaces}). By Bézout's theorem, the intersection $H\cap H' \cap H''$ is a curve of degree 1, which contains the line $L_1$ and can thus not also contain $L_2$. It follows that $L_2\not \sub H\cap H'\cap H''$. Up to replacing $H$ by $H'$ or $H''$, we may suppose that $L_2\not \sub H$.  Then the intersection number of $L_1$ and $L_2$ satisfies $L_1 \cdot L_2 \leq H \cdot L_2 = 1$. We deduce that $$p_a(D)=p_a(mL_1)+p_a(L_2)+ mL_1 \cdot L_2 -1\leq 1-m^2+m-1=m(1-m)\leq 0.$$
\item $D=L_1+L_2+L_3$ for three distinct lines $L_1,L_2,L_3 \sub S$: Analogously to the second case, we can argue that $L_i\cdot L_j \leq 1$ for all $i,j\in\{1,2,3\}$ with $i\neq j$. We want to show that $L_1\cdot L_2 + L_1\cdot L_3 + L_2\cdot L_3 \leq 2$. Suppose to the contrary that $L_i\cdot L_j = 1$ for all $i,j\in\{1,2,3\}$ with $i\neq j$. Then there are two possibilities:\par 
The first one is that $L_1, L_2$ and $L_3$ all intersect in the same point $p\in S$. If the lines do not lie on a common plane, then the tangent space $T_p(S)$ is of dimension at least 3, which contradicts the assumption that $S$ is smooth. Otherwise, if $L_1,L_2$ and $L_3$ would lie on a plane $P=H_1\cap H_2$ for two hyperplanes $H_1,H_2\sub \Pfour$, they must be contained in the intersection $H_1\cap H_2\cap Q$. However, this intersection must be of degree 2 by Bézout's Theorem since the smooth hyperquadric $Q$ cannot contain $P$ (\Cref{lem:smooth_quadric_irreducible_contains_no_plane}). So $H_1\cap H_2 \cap Q$ cannot contain the reducible curve $L_1\cup L_2 \cup L_3$ of degree 3. \par 
The second possibility to have $L_i\cdot L_j = 1$ for all $i,j\in\{1,2,3\}$ with $i\neq j$ is that $L_1, L_2$ and $L_3$ intersect in 3 distinct points $p_1,p_2,p_3\in S$. Consider a plane $H\cap H'$ through $p_1,p_2$ and $p_3$ that is given as an intersection of two hyperplanes $H,H'\sub \Pfour$. Since two points of the set $\{p_1,p_2,p_3\}\sub H\cap H'$ are contained in $L_i$, we apply Bézout's theorem to deduce that $L_i \sub H\cap H'$ for all $i\in\{1,2,3\}$. As the smooth hyperquadric $Y$ cannot contain the plane $H\cap H'$ (\Cref{lem:smooth_implies_irred}), we can again use Bézout's theorem to argue that $H\cdot H' \cdot Y\leq 2$. However, this contradicts the fact that $L_1\cup L_2 \cup L_3 \sub H \cap H' \cap Y$. \par
We have thus indeed shown that $L_1\cdot L_2 + L_1\cdot L_3 + L_2\cdot L_3 \leq 2$.
We conclude $$p_a(D)= p_a(L_1)+p_a(L_2)+p_a(L_3) + L_1\cdot L_2 + L_1\cdot L_3 + L_2 \cdot L_3-2\leq 2 -2 =0.$$ 
\item $D=\Gamma + L$ for a line $L\sub S$ and an integral conic $\Gamma \sub S$: The conic $\Gamma$ is contained in two distinct hyperplanes $H, H'\sub \Pfour$ (\Cref{lem:dimension_linear_system_hypersurfaces}). The smooth hyperquadric $Y$ cannot contain $H\cap H'$ (\Cref{lem:smooth_implies_irred}). So by Bézout's theorem, the intersection $H\cap H' \cap Y$ is a curve of degree 2, which contains the conic $\Gamma$ and can therefore not also contain $L$. It follows that $L\not \sub H\cap H'$. Up to exchanging $H$ and $H'$, we may suppose that $L\not \sub H$.  Then the intersection number of $\Gamma$ and $L$ satisfies $\Gamma \cdot L \leq H \cdot L = 1$. We deduce that $p_a(D)=p_a(\Gamma)+p_a(L)+\Gamma \cdot L-1=\Gamma \cdot L -1\leq 0
$.
\end{itemize}
\end{proof}
The next proposition generalizes the necessity of the bound from \Cref{thm:knutsen_existence_curves_in_K3_surface_degree6} on the genus of a curve $C$ in a smooth sextic K3-surface in the following sense: We do not assume that $C$ is irreducible or smooth.
\begin{prop} \label{prop:bound_curves_in_K3surface_of_degree_6} 
Let $S=Y\cap Z\sub \Pfour$ be the smooth intersection of a smooth hyperquadric $Y\sub \Pfour$ with a hypercubic $Z\sub \Pfour$. Let $D \sub S$ be a reduced (possibly singular or reducible) curve of degree $d$ and (arithmetic) genus $g$. We define an integer $B(d)\in\Z_{\geq0}$ by:
\begin{equation*} \label{eq:upper_bound_genus_curve_in_K3_surface_degree6}
B(d):= 
\begin{cases}
	&\frac{d^2-r^2}{12} \quad \:\textup{ if } d\equiv_6 r \textup{ with } r\in\{-2,-1,1,2,3\},\\
	&\frac{d^2}{12}-1 \quad \textup{ if } d\equiv_6 0.
\end{cases}
\end{equation*}
Then
\begin{enumprop}
\item $D$ is the complete intersection of $S$ with another hypersurface if and only if there is $e\in \Z_{>0}$ with $d=6e$ and $g=\frac{d^2}{12}+1=3e^2+1$. \label{prop:bound_curves_not_in_K3surface_of_degree_6_mod_0}
\item If $D$ is not the complete intersection of $S$ with another hypersurface, then $$g\leq B(d)\leq \frac{d^2-1}{12}.$$
Moreover, we have $(g,d) = (4,7)$ if and only if $D=L+\Gamma$ where $L\cdot \Gamma=1$, $L\sub S$ is a line and $\Gamma \sub S$ is a curve of degree 6 and genus 4 (or equivalently, $\Gamma$ is the complete intersection of $S$ with a hyperplane). \label{prop:bound_curves_not_in_K3surface_of_degree_6_mod_12345}
\end{enumprop}
\begin{table}[H]
\caption{For $d\leq 18$, we obtain the following upper bound $B(d)$ for the genus $g$ of a curve $D$ of degree $d$ that lies on $S$ and is not the complete intersection of $S$ with a hypersurface of degree $1,2$ or 3:}
\label{table:genus_bounds}
\begin{minipage}{\textwidth}
	\centering
	\vspace{0.25cm}
	\begin{tabular}{|l|l|l|l|l|l|l|l|l|l|l|l|l|l|l|l|l|l|l|}
		\hline
		$d$ & 1&2&3&4&5&6&7&8&9&10&11&12&13&14&15&16&17&18\\
		\hline
		$B(d)$ &0&0&0&1&2&2&4&5&6&8&10&11&14&16&18&21&24&26\\
		\hline
	\end{tabular}
\end{minipage}
\end{table}
\end{prop}
\begin{proof}
Since $S$ is a smooth K3-surface of degree 6, we have $K_S=0$ and $\Gamma^2= 2p_a(\Gamma)-2$ for each effective divisor $\Gamma \sub S$ (\Cref{rmk:Riemann_Roch_formula}). Denoting a general hyperplane section of $S$ by $H_S$, we moreover have $H_S^2 = \deg(S) = 6$, as well as $H_S \cdot D = \deg(D)=d$.\par
We first suppose that there is an irreducible hypersurface $W\sub \Pfour$ of degree $e\in \Z_{>0}$ such that $D=S \cap W = Y \cap Z \cap W$. Then $d=\deg(Y)\cdot \deg(Z)\cdot \deg(W) = 6e$ and by \Cref{lem:genus_formula_complete_intersection}, the genus of the complete intersection $D= Y \cap Z \cap W$ is then equal to $$p_a(D)=\frac{1}{2} \cdot \deg(W)\cdot\deg(Y)\cdot\deg(Z)\cdot (-5+\deg(W)+\deg(Y)+\deg(Z))+1= 3e^2+1=\frac{d^2}{12}+1.$$
This gives one direction of \subcref{prop:bound_curves_not_in_K3surface_of_degree_6_mod_0}. Conversely, we now prove the following implication:
\begin{gather*} \tag{$\ast$} \label{eq:implication}
\textup{If } g> \frac{d^2}{12}-1 \textup{ and } d=6e \textup{ for some } e\in \Z_{>0}, \\ \textup{then } D \textup{ is the complete intersection of } S \textup{ with another hypersurface of degree } e.
\end{gather*}
This will prove the other direction of \subcref{prop:bound_curves_not_in_K3surface_of_degree_6_mod_0}, but also give the bound of \subcref{prop:bound_curves_not_in_K3surface_of_degree_6_mod_12345} in the case that $6|d$.\par 
Denoting $D':= D-eH_S\in \Pic(S)$, we find that
\begin{gather*}
H_S\cdot D' = H_S \cdot (D-eH_S) = d-6e = 0 \textup{ and } \\
(D')^2 = D\cdot D' - e \underbrace{H_S \cdot D'}_{=0} = D\cdot( D - eH_S) = \underbrace{2g}_{\mathclap{>\frac{1}{6}d^2-2=ed-2}}-2 - ed>-4.
\end{gather*}
From this and \Cref{lem:Hodge_index_thm}, we deduce that $D':=D-eH_S=0$, or equivalently, $D=eH_S$. We have thus indeed shown \eqref{eq:implication}. The two assumptions of the implication \eqref{eq:implication} are, in particular, satisfied if $d=6e$ and $g=\frac{d^2}{12}+1 > \frac{d^2}{12}-1$, proving the second direction of \subcref{prop:bound_curves_not_in_K3surface_of_degree_6_mod_0}. Moreover, we can form the negation of \eqref{eq:implication}: If $d\equiv_6 0$ and $C$ is not the complete intersection of $S$ with another hypersurface, then $g \leq \frac{d^2}{12}-1=B(d)$. This proves \subcref{prop:bound_curves_not_in_K3surface_of_degree_6_mod_12345} in the case that $6|d$. \par
We now want to show $g\leq B(d)$ for the first case of \subcref{prop:bound_curves_not_in_K3surface_of_degree_6_mod_12345}, where $d\equiv_6 r$ with  $r\in \{-2,\:-1,\:1,\:2,\:3\}$. Suppose to the contrary that $g> B(d)= \frac{d^2}{12}- \frac{r^2}{12}$.\par 
Let $m\in \Z_{\geq0}$ with $d=6m+r$. If $r\in\{1,\:2,\:3\}$, we define $D':=D-mH_S$. Otherwise, if $r\in \{-1,\:-2\}$, let $D':=mH_S-D$. In both cases, we have $H_S\cdot D'=|d-6m|=|r|\in\{1,\:2,\:3\}$ and
\begin{align*}
(D')^2 &= (\pm(D - mH_S))^2 = D^2 - 2mD\cdot H_S + m^2 H_S^2 = 2g - 2 -2\left(\frac{d-r}{6}\right) \cdot d + 6 \left(\frac{d-r}{6}\right)^2 \\
&= \underbrace{2g}_{> \frac{d^2}{6} - \frac{r^2}{6}} -2 - \frac{d^2}{3} + \cancel{\frac{rd}{3}} + \frac{d^2}{6} - \cancel{\frac{rd}{3}} + \frac{r^2}{6} > -2.
\end{align*} 
It follows from \Cref{lem:Riemann_Roch_effective_divisor} that $D'$ is a curve of degree $H_S\cdot D'\in\{1,\:2,\:3\}$ and genus $\frac{1}{2} (D')^2+1>0$. This contradicts, however, the fact that any effective divisor of degree at most 3 in $S\sub Y$ must be of non-positive genus by \Cref{lem:divisor_degree_3_genus_1}. So we must have $g\leq B(d)$ whenever $D$ is not the complete intersection of $S$ with another hypersurface.
\par
It is left to prove the second part of \subcref{prop:bound_curves_not_in_K3surface_of_degree_6_mod_12345}, which concerns the equivalence of the case $(g,d)=(4,7)$. Suppose first that $D$ is of degree $D\cdot H=7$ and genus $p_a(D)=4$.
Consider the divisor $L:=D-H_S\in\Pic(S)$. We have 
\begin{gather*}
H_S \cdot L = H_S\cdot D-H_S^2=7-6=1, \quad L^2= D^2-2H_S\cdot D+H_S^2=2p_a(D)-2-14+6=-2, \\
L\cdot D=D^2-H_S\cdot D=2p_a(D)-2-7=-1.
\end{gather*}
Now \Cref{lem:Riemann_Roch_effective_divisor} implies that $L$ is a line. Furthermore, the fact that the intersection $L\cdot D=-1$ of two effective divisors is negative and that $\deg(L)<\deg(D)$ implies that we must have $L\sub D$. So $D=L+\Gamma$ for an effective divisor $\Gamma\in\Pic(S)$ of degree 6. We calculate that
$$ L\cdot \Gamma= L\cdot D-L^2=-1+2=1 \quad \textup{and} \quad 4\overset{\ref{lem:genus_formula_union_of_curves}}{=} p_a(D)=p_a(L)+p_a(\Gamma)+L\cdot \Gamma -1=p_a(\Gamma).$$
Conversely, let $\Gamma\sub S$ be a curve of degree 6 and genus 4. By \subcref{prop:bound_curves_not_in_K3surface_of_degree_6_mod_0}, this is equivalent to requiring that $\Gamma$ is the complete intersection of $S$ with a hyperplane. Let $L\sub S$ be a line with $L\cdot \Gamma=1$. Then we can use \Cref{lem:genus_formula_union_of_curves} to obtain
$$ p_a(L+\Gamma)=p_a(L)+p_a(\Gamma)+L\cdot \Gamma -1 = 4+1-1=4.$$
\end{proof}
\begin{rmk}
Comparing the bounds of \Cref{thm:knutsen_existence_curves_in_K3_surface_degree6} with those of \Cref{prop:bound_curves_in_K3surface_of_degree_6}, we see that \Cref{thm:knutsen_existence_curves_in_K3_surface_degree6_2} cannot appear if we require the hyperquadric $Y$ to be smooth. In \Cref{thm:knutsen_existence_curves_in_K3_surface_degree6}, the curve $C$ is assumed to be irreducible and smooth. Allowing reducible, singular curves, the only new pair in \Cref{prop:bound_curves_in_K3surface_of_degree_6} that not appear in \Cref{thm:knutsen_existence_curves_in_K3_surface_degree6} is $(g,d)=(4,7)$. \par 
It is left to see that a curve corresponding to this pair actually exists. In other words, we want to show that the bound $B(7)=4$ from \Cref{prop:bound_curves_in_K3surface_of_degree_6} is tight. Let $S\sub \Pfour$ be a smooth K3-surface of degree 6 and $H\sub \Pfour$ a general hyperplane. Then the divisor $H_S=S\cap H\in \Pic(S)$ is the complete intersection of $H$ and $S$. It is hence a curve of degree 6 and genus 4 (\Cref{prop:bound_curves_not_in_K3surface_of_degree_6_mod_0}). \par 
Fix a point $p\in H_S$ and a point $q\in \Pfour \setminus H$. Let $L\sub \Pfour$ be the unique line through $p$ and $q$. Since $q\not \in H$, we have $L\not \sub H$. So we can use Bézout's theorem to deduce that $H_S\cdot L \leq  H \cdot L =1$. In fact, we must have $H_S\cdot L=1$ as $p\in L\cap H_S$. Applying the formula from \Cref{lem:genus_formula_union_of_curves}, we have found a (reducible) curve $D:=H_S \cup L$ of degree 7 and genus $p_a(D)=p_a(H_S)+p_a(L)+H_S\cdot L -1= 4$.
\end{rmk}
\begin{rmk}
\cite{De_Cataldo-Genus_of_curves_on_3d_quadric} gives an upper bound on the geometric genus of an integral curve $C$ in terms of the degree of $C$, given that $C$ lies on an integral surface that is the intersection of a smooth hyperquadric in $\Pfour$ with a hypersurface of degree $k$. For $k=3$, this gives almost the same situation as in \Cref{prop:bound_curves_in_K3surface_of_degree_6}. However, in \Cref{prop:bound_curves_in_K3surface_of_degree_6}, we only require the curve $C$ to be reduced instead of integral, and we use the arithmetic genus instead of the geometric genus. Recalling that integral implies reduced and that the geometric genus of a curve is always smaller than or equal to the arithmetic genus, \Cref{prop:bound_curves_in_K3surface_of_degree_6} implies the case $k=3$ of \cite[Thm.\ 1.4]{De_Cataldo-Genus_of_curves_on_3d_quadric}. \par 
Moreover, our bounds from \Cref{prop:bound_curves_in_K3surface_of_degree_6} coincide with the bounds given in  \cite[Thm.\ 1.4]{De_Cataldo-Genus_of_curves_on_3d_quadric} for all $d\leq 18$, except for the case $d=14$: \cite[Thm.\ 1.4]{De_Cataldo-Genus_of_curves_on_3d_quadric} gives $g=15$ as an upper bound. Our bound, however, is $B(14)=16$. This is sharp as we will show in \Cref{thm:existence_curves_in_K3_surface_degree6_smooth_quadric} (see also \Cref{rmk:counterexample}). 
\end{rmk}
\subsection{Curves on a smooth hypercubic section yielding a weak Fano threefold} \label{subsection:curves_yielding_weak_Fano}
In this subsection, we will prove the main theorems \ref{thm:A} and \ref{thm:B}.
After proving an auxiliary lemma, we collect necessary conditions  on the curve $C$ for $X$ to be weak Fano, thereby showing one direction of \Cref{thm:A}. 
\begin{lem} \label{lem:numerical_description_of_P}
The following are equivalent:
\begin{itemize}
\item $(g,d)\in \mathcal{P}$ (where $\mathcal{P}$ is the set of pairs defined in the introduction);
\item either $(g,d)\in \{(4,6),\:(13,12)\}$, or $d < 18$, $3d-26<g\leq\frac{d^2-1}{12}$ and $(g,d)\not \in \{(4,7),\:(10, 11)\}$ (this is precisely condition \subcref{thm:A_g_and_d} of \Cref{thm:A}).
\end{itemize}
\end{lem}
\begin{proof}
By definition of $\mathcal{P}$, we have $\{(4,6),\:(13,12)\}\sub \mathcal{P}$ and $(4,7),\:(10, 11) \not \in \mathcal{P}$. For all pairs $(g,d)\in \mathcal{P}\setminus \{(4,6),\:(13,12)\}$, a direct calculation shows that the bounds $d < 18$ and $3d-26<g\leq\frac{d^2-1}{12}$ are satisfied. Conversely, a numerical check verifies that the only pairs $(g,d)\in \Z_{\geq 0} \times \Z_{>0}$ satisfying the bound $d < 18$ and $3d-26<g\leq\frac{d^2-1}{12}$ are $(g,d)\in \mathcal{P} \cup \{(4,7),\:(10, 11)\}$.
\end{proof} 

\begin{lem} \label{lem:general_case_weak_Fano_implies_in_list}
Let $C \sub Y$ be a smooth irreducible curve of genus $g$ and degree $d$ that lies on a smooth hyperquadric $Y \sub \Pfour$ and denote the blow up $\pi:X\rightarrow Y$ of $Y$ along $C$ by $X$. \par 
If $X$ is weak Fano, then:
\begin{enumerate}[label=\textup{(\arabic*)}]
\item $|-K_X|$ is base-point-free and $C$ is contained in a smooth hypercubic section of $Y$.
\item In $Y$, there are no irreducible curves of degree $n\in \Z_{>0}$ that are $3n+1$-secant to $C$.
\item Either $(g,d)\in \{(4,6),\:(13,12)\}$, or $d < 18$, $3d-26<g\leq\frac{d^2-1}{12}$ and $(g,d)\not \in \{(4,7),\:(10, 11)\}$.
\item $26-3d+g>0$ and $3d+2-2g\geq0$.
\end{enumerate}	
In particular, we have $(g,d)\in \mathcal{P}$. 
\end{lem}
\begin{proof}
Suppose that $X$ is weak Fano. Then $|-K_X|$ is base-point-free (\Cref{prop:weak_Fano_implies_base_point_free}) and $-K_X $ is nef and big. In particular, there are no curves in $Y$ of degree $n\in\Z_{>0}$ that are $3n+1$-secants to $C$ (\Cref{lem:3n+1_secants}). Moreover, we have $26-3d+g>0$, or equivalently, $3d-26<g$ (\Cref{cor:weak_Fano_big_nef_basepoinfree}). \par
By \Cref{prop:weak_Fano_implies_contained_in_smooth_hypercubic_section}, there is a smooth hypercubic section $S$ of $Y$ that contains $C$. Since $S$ is smooth, it is, in particular, smooth along a general point of $C$. So the strict transform $\widetilde{S}\sub X$ of the smooth hypercubic section $S$ of $Y$ satisfies $\widetilde{S} \in |3 H_X-E|=|-K_X|$ (\Cref{lem:formula_for_anticanonical_divisor}). \par 
As $S$ is a smooth surface containing $C$, the restriction $\widetilde{S}|_E$ is a section of $\pi|_E:E\rightarrow C$. So $\widetilde{S}|_E$ is isomorphic to $C$ and hence an irreducible curve. By the nefness of $-K_X$, the following inequality must thus be satisfied:
$$0\leq (-K_X) \cdot \widetilde{S}|_E = (-K_X)\cdot(-K_X|_E) = K_X^2 \cdot E \overset{\textup{\ref{lem:K_X_squared_times_E}}}{=} 3d+2-2g. $$
\Cref{prop:weak_Fano_implies_contained_in_smooth_hypercubic_section} furthermore guarantees that $C$ is contained in at least 3 non-linearly equivalent hypercubic sections of $Y$ that are smooth and thus also irreducible (\Cref{lem:smooth_implies_irred}). This implies that $C$ is contained in an irreducible surface $S$ of degree 6 and in at least one irreducible hypercubic $Z\sub \Pfour$ not containing $S$. With this, it follows from Bézout's theorem that $d\leq \deg(S) \cdot \deg(Z)= 18$.\par
In particular, $C$ can only be the complete intersection of $S$ with a hypersurface of degree $e\leq 3$. If this is the case, \Cref{lem:genus_formula_complete_intersection} implies that $d=e\cdot\deg(S)=6e$ and $g= 3e^2+1$. This gives the pairs $(4,6), \:(13,12)$ and $(28,18)$. The latter pair is, however, not possible since it does not satisfy $3d-26<g$.\par
This shows that $d<18$. Furthermore, if $C$ is not the complete intersection of $S$ with another hypersurface, it follows from \Cref{prop:bound_curves_in_K3surface_of_degree_6} that $g\leq \frac{d^2-1}{12}$ and $(g,d)\neq(4,7)$. \par 
It is left to eliminate the case $(g,d)=(10,11)$. The remaining cases will satisfy $(g,d)\in\mathcal{P}$ by \Cref{lem:numerical_description_of_P}. \par 
If $C$ is a smooth irreducible curve of degree $d=11$ and genus $g=10$, then $g-1<d$ and $13+g-2d=1$. So \Cref{lem:dimension_linear_system_hypersurfaces} implies that $C$ is contained in a quartic surface $Y\cap Q$ of $Y$ where $Q\sub \Pfour$ is another hyperquadric. If $Q$ were reducible, then $C \sub H$ for some hyperplane $H\sub \Pfour$. Since $S=Y\cap Z$ is smooth and thus an irreducible surface (\Cref{lem:smooth_implies_irred}), the intersection $Y\cap H \cap Z= S \cap H$ would be a curve of degree $\deg(Y)\cdot \deg(H)\cdot \deg(Z)=6$ by Bézout's theorem. This contradicts the fact that $C\sub S\cap H$ is a curve of degree 11. \par 
Thus, $Q$ must be irreducible. Using the same arguments as before with $Q$ in place of $H$, the intersection $S \cap Q$ is then a curve of degree $\deg(Y)\cdot \deg(Q)\cdot \deg(Z)=12$. Since $C\sub Y\cap Q \cap Z$, we must have $Y\cap Q \cap Z=C\cup L$, where $L\sub \Pfour$ is a line. Applying \Cref{prop:bound_curves_in_K3surface_of_degree_6} to the reducible curve $C\cup L \sub S$ of degree 12, we see that the fact that $C\cup L$ is the complete intersection of $S$ with a hyperquadric $Q$ implies that $p_a(C\cup L)=13$. We deduce that
$$13=p_a(C \cup L) \overset{\ref{lem:genus_formula_union_of_curves}}{=} p_a(C)+p_a(L)+C\cdot L-1= g+0+C\cdot L -1 = 9+C\cdot L.$$
This is only possible if $C\cdot L=4$. Therefore, $L$ is a 4-secant-line of $C$, which contradicts the assumption that $-K_X$ is nef (\Cref{lem:3n+1_secants}). So we have shown that $(g,d)=(10,11)$ is indeed not possible.
\end{proof}
We now want to prove the second direction of \Cref{thm:A}. That is, $X$ is weak Fano, given that the curve $C$ satisfies condition (1), (2) and (3) of the previous \cref{lem:general_case_weak_Fano_implies_in_list}. In fact, we will see that a weakened version of condition (2) is also sufficient: namely, that $C$ has no 4-secant-lines or 7-secant conics.  \par 
The next lemma shows that the nefness of $-K_X$ depends only on the existence of rational curves whose intersection with $C$ is too high. 
\begin{lem} \label{lem:not_nef_implies_3n+1_secant}
Let $C \sub Y$ be a smooth irreducible curve of genus $g$ and degree $d$ that lies on a smooth hyperquadric $Y \sub \Pfour$ and denote the blow up $\pi:X\rightarrow Y$ of $Y$ along $C$ by $X$. 
Suppose that \begin{itemize}
\item $d < 18$,
\item $C$ is contained in a smooth hypercubic section $S$ of $Y$,
\item $26-3d+g>0$ and $3d+2-2g\geq0$. 
\end{itemize} 
If $-K_X$ is not nef, then there is a rational curve $\Gamma \sub S$ of degree $n\leq 18-d$, which is a $3n+1$-secant of $C$ (meaning that $\Gamma \cdot C\geq 3n+1$ as an intersection of divisors in $S$).
\end{lem}
\begin{proof}
Denote the strict transform of the smooth K3-surface $S$ by $T:=\widetilde{S}\sub X$ . By assumption, the divisor $-K_X$ is not nef. So there is an irreducible curve $D\sub X$ such that ($-K_X)\cdot D <0$.\par
We first assume that $D$ is contained in the exceptional divisor $E=\pi^{-1}(C)$ and want to derive a contradiction from this. If $D\sub E$, we can look at the intersection $(-K_X) \cdot D$ in $E$ and deduce that $(-K_X|_E)\cdot D <0$.\par
As $S$ is a smooth surface containing $C$, the restriction $T|_E=\widetilde{S}|_E$ is a section of $\pi|_E$ and hence, in particular, an effective divisor. So $T|_E\cdot D=(-K_X|_E)\cdot D<0$ implies that $D\sub \supp(T|_E)$. However, the section $T|_E$ is irreducible as it is isomorphic to $C$. We must thus have $-K_X|_E=T|_E=D$. Noting that $(-K_X)\cdot (-K_X|_E)=(-K_X|_E)\cdot (-K_X|_E)$, we deduce that
$$0>(-K_X|_E) \cdot D = (-K_X|_E)^2 =(-K_X)\cdot (-K_X|_E) = (-K_X)^2 \cdot E \overset{\textup{\ref{lem:K_X_squared_times_E}}}{=} 3d+2-2g. $$
This is, however, not possible as $3d+2-2g\geq0$ by assumption. \par
The irreducible curve $D$ is hence not contained in $E$ and thus, $D=\widetilde{\Gamma}$ is the strict transform of an irreducible curve $\Gamma \sub Y$ with $\Gamma \neq C$. \par 
We now want to show that $\Gamma$ is contained in $S$ and is a $3n+1$-secant of $C$. Denoting the pullback of a general hyperplane section $H_Y$ of $Y$ under $\pi$ with $H_X=\pi^*(H_Y)$, we deduce that
$$ 0 > (-K_X) \cdot \widetilde{\Gamma} \overset{\textup{\ref{lem:formula_for_anticanonical_divisor}}}{=} (3 H_X - E) \cdot \widetilde{\Gamma} = 3 H_Y \cdot \Gamma - E \cdot \widetilde{\Gamma}= 3n - E \cdot \widetilde{\Gamma}.$$
This gives $E \cdot \widetilde{\Gamma} \geq 3n+1$. That is to say, $\Gamma$ is a $3n+1$-secant of $C$.\par
Any hypersurface containing $C$, but not $\Gamma$, intersects hence $\Gamma $ at least $3n+1$ times. Using Bézout's theorem, it follows that $\Gamma$ must be contained in any hyperplane, hyperquadric and hypercubic of $\Pfour$ that contain $C$. In particular, $\Gamma$ must be contained in $S$ since $S$ is, by assumption, an intersection of the hyperquadric $Y$ and a hypercubic containing $C$. Therefore, it makes sense to look at the intersection $C\cdot \Gamma = E \cdot \widetilde{\Gamma} \geq 3n+1$, where the intersection $C\cdot \Gamma$ is taken in $S$.\par
It is left to show that $\Gamma$ is a rational curve of degree $n\leq 18-d$. Recall that $S$ and thus
$T:=\widetilde{S}$ is a smooth K3-surface satisfying $T\in |-K_X|$ (\Cref{lem:formula_for_anticanonical_divisor}). As $26-3d+g>0$ by assumption, we first calculate that
$$(T|_T)^2= T\cdot T|_T = T^3= (-K_X)^3\overset{\textup{\ref{lem:cube_of_K_X}}}{=}52-6d+2g >0.$$
Furthermore, we can note that (as in the proof of \Cref{lem:anticanonical_divisor}) $H_X^3=\deg(Y)=2, H_X^2\cdot E = 0$ and $H_X\cdot E^2 =-d$. With this, we obtain
$$ T|_T \cdot H_X = T\cdot T \cdot H_X = (-K_X)^2\cdot H_X \overset{\textup{\ref{lem:formula_for_anticanonical_divisor}}}{=}9H_X^2-6H_X^2\cdot E+ H_X\cdot E^2=18-d> 0.$$ 
Using \Cref{lem:Riemann_Roch_effective_divisor}, it follows that $T|_T$ is an effective  divisor on $T$ of degree $18-d$. \par 
Since we have shown that $\Gamma \sub S$, we know that $\widetilde{\Gamma} \sub \widetilde{S}=T$. In particular, we can look at the intersection $(-K_X)\cdot \widetilde{\Gamma} = T\cdot \widetilde{\Gamma}<0$ in $T$ to deduce that $T|_T\cdot \widetilde{\Gamma}<0$. However, $T|_T$ is effective and $\widetilde{\Gamma}$ is an effective prime divisor. So the intersection can only be negative if $\widetilde{\Gamma}\sub \supp(T|_T)$ and $(\widetilde{\Gamma})^2<0$. \par
As $T$ is a smooth K3-surface, we have $2p_a(\widetilde{\Gamma})-2=(\widetilde{\Gamma})^2<0$. But $\widetilde{\Gamma}$ is irreducible and thus $p_a(\widetilde{\Gamma})$ is non-negative. It follows that $p_a(\widetilde{\Gamma})=0$ and thus, again since $\widetilde{\Gamma}$ is irreducible, $\widetilde{\Gamma}$ is isomorphic to $\Pone$ (\cite[Cor.\ 7.7]{Mumford-Algebraic_Geometry_I}). In other words, $\Gamma$ is a rational curve with $p_a(\Gamma)=0$.\par
Lastly, observe that the blowup $\pi:X\rightarrow Y$ induces an isomorphism $S\simeq \widetilde{S}=T$ (as $S$ is a smooth surface containing $C$, see \Cref{rmk:universal_property_of_blowups}). Since moreover, $\widetilde{\Gamma}\sub T$, it follows that 
$$ \deg(\Gamma)= \Gamma \cdot H_Y= \widetilde{\Gamma} \cdot H_X \leq T|_T \cdot H_X =18-d.$$ 
\end{proof}
In the previous lemma, we have seen that under the assumptions of \Cref{thm:A}, the only possibility for $-K_X$ to not be nef is the existence of a rational curve of degree $n$, which is a $3n+1$-secant of $C$. In the following, we will prove  that it suffices to take $n\in \{0,1,2\}$. The exact value of $n$ will depend on the pair $(g,d)\in \mathcal{P}$ (see \Cref{table:max_degree_of_3n+1_secant}). This will also prove the second direction of \Cref{thm:A} as well as \Cref{thm:B}.
\begin{prop} \label{thm:general_case_in_list_implies_weak_Fano}
Let $C \sub Y$ be a smooth irreducible curve of genus $g$ and degree $d$ that lies on a smooth hyperquadric $Y \sub \Pfour$ and denote the blow up $\pi:X\rightarrow Y$ of $Y$ along $C$ by $X$. \par
Suppose that $(g,d)\in \mathcal{P}$ and that $C$ is contained in a smooth hypercubic section of $Y$. If $C$ has no $3n+1$-secants of degree $n\leq n_{\max}$, where $n_{\max}\in\{0,1,2\}$ is given by the fourth column in \Cref{table:max_degree_of_3n+1_secant}, then $X$ is weak Fano.
\end{prop}
\vspace{-0.5cm}
\begin{table}[H]
\caption{Maximal degree $n\leq n_{\max}$ of a curve that is $3n+1$-secant to the smooth irreducible curve $C$ of genus $g$ and degree $d$. The numbers in the fourth column that are not crossed out are obtained from the non-negativity of the polynomial $P_{g,d}$. Some of these numbers can be improved using additional arguments, as indicated by the crossed-out values.}
\label{table:max_degree_of_3n+1_secant}
\begin{minipage}{\textwidth}
\centering
\vspace{0.25cm}
\begin{multicols}{2}
	\begin{tabular}{|l|l|l|l|}
		\hline
		$g$ & $d$ & $P_{g,d}(n)$ & $n_{\max}$ \\
		\hline
		0 & 1 & $n^2-34n$ & 0 \\
		\hline
		0 & 2 & $n^2-32n+3$ & 0 \\
		\hline
		0 & 3 & $n^2-30n+8$ & 0\\
		\hline
		0 & 4 & $n^2-28n+15$ & 0 \\
		\hline
		0 & 5 & $n^2-26n+24$ & 0 \\
		\hline 
		0 & 6 & $n^2-24n+35$ & 1\\
		\hline
		0 & 7 & $n^2-22n+48$ & \cancel{2} 1\\
		\hline
		0 & 8 & $n^2-20n+63$ & \cancel{3} 2 \\
		\hline
		1 & 4 & $n^2-28n+3$ & 0 \\
		\hline
		1 & 5 & $n^2-26n+12$ & 0 \\
		\hline 
		1 & 6 & $n^2-24n+23$ & \cancel{1} 0 \\
		\hline
		1 & 7 & $n^2-22n+36$ & 1 \\
		\hline
		1 & 8 & $n^2-20n+51$ & \cancel{3} 2 \\
		\hline
		2 & 5 & $n^2-26n$ & 0 \\
		\hline
		2 & 6 & $n^2-24n+11$ & 0 \\
		\hline 
		2 & 7 & $n^2-22n+24$ & \cancel{1} 0\\
		\hline
		2 & 8 & $n^2-20n+39$ & \cancel{2} 1 \\
		\hline
		2 & 9 & $n^2-18n+56$ & \cancel{4} 2 \\
		\hline
	\end{tabular}
	\\
	\begin{tabular}{|l|l|l|l|l|}
		\hline
		$g$ & $d$ & $P_{g,d}(n)$ & $n_{\max}$ \\
		\hline
		3 & 7 & $n^2-22n+12$ & 0 \\
		\hline
		3 & 8 & $n^2-20n+27$ & 1 \\
		\hline
		3 & 9 & $n^2-18n+44$ & 2 \\
		\hline
		4 & 6 & $n^2-24n-13$ & 0 \\
		\hline
		4 & 8 & $n^2-20n+15$ & 0 \\
		\hline 
		4 & 9 & $n^2-18n+32$ & \cancel{2} 1 \\
		\hline
		5 & 8 & $n^2-20n+3$ & 0\\
		\hline
		5 & 9 & $n^2-18n+20$ & 1 \\
		\hline
		5 & 10 & $n^2-16n+39$ & \cancel{3} 2 \\
		\hline
		6 & 9 & $n^2-18n+8$ & 0 \\
		\hline
		6 & 10 & $n^2-16n+27$ & 1 \\
		\hline
		7 & 10 & $n^2-16n+15$ & 1 \\
		\hline
		8 & 10 & $n^2-16n+3$ & 0 \\
		\hline
		8 & 11 & $n^2-14n+24$ & \cancel{2} 1 \\
		\hline
		9 & 11 & $n^2-14n+12$ & 0 \\
		\hline
		11 & 12 & $n^2-12n+11$ & 1 \\
		\hline
		13 & 12 & $n^2-12n-13$ & 0 \\
		\hline
		14 & 13 & $n^2-10n$ & 0 \\
		\hline
		
	\end{tabular}
\end{multicols}
\end{minipage}
\end{table}
\begin{proof}
All pairs $(g,d)\in \mathcal{P}$ satisfy $26-3d+g>0$. It is therefore sufficient to show that $-K_X$ is nef by \Cref{cor:weak_Fano_big_nef_basepoinfree}. Suppose to the contrary that $-K_X$ were not nef. \par
Since all pairs $(g,d)\in \mathcal{P}$ satisfy $d<18, 26-3d+g>0$, as well as $3d+2-2g>0$ and since $C$ is contained in a smooth hypercubic section $S$ of $Y$ by assumption, we can apply \Cref{lem:not_nef_implies_3n+1_secant}. This lemma guarantees the existence of a rational curve $\Gamma \sub S$ of degree $n\leq 18-d$ with $C\cdot \Gamma \geq 3n+1$ (where the intersection is taken in $S$). \par
e will use these properties and \Cref{prop:bound_curves_in_K3surface_of_degree_6} to bound the genus of $\Gamma$ in terms of $g,\:d$ and $n$. This will give us a range of values that $n$ can attain, namely $1\leq n\leq4$. Using additional arguments for several particular pairs, which involve finding contradictory divisors on $S$, we will be able to bring this bound down to $n\leq n_{\max}$, where $n_{\max}$ is given by \Cref{table:max_degree_of_3n+1_secant}.  \par 
The smooth hypercubic section $S\sub Y\sub \Pfour$ is a K3-surface. It thus satisfies $2g(B)-2=B^2$ for all curves $B\sub S$ (\Cref{lem:genus_formula_union_of_curves}) and, in particular, for $B\in\{C,\:\Gamma, \:C\cup \Gamma\}$. Since $C\cdot \Gamma \geq 3n+1$ and $p_a(\Gamma)=0$, it follows that
\begin{equation} \label{eq:genus_of_C_union_Gamma} \tag{$\diamond$}
p_a(C\cup \Gamma)\overset{\ref{lem:genus_formula_union_of_curves}}{=}g+C\cdot\Gamma -1 \geq g + 3n.
\end{equation}
We first suppose that $C\cup \Gamma$ is a complete intersection of $S$ with another hypersurface of degree $e\in\Z_{>0}$ and want to derive a contradiction from this assumption. Since $n\leq18-d$, we must have $e\in\{1,2,3\}$. Moreover, it follows from \Cref{prop:bound_curves_not_in_K3surface_of_degree_6_mod_0} that $6e=\deg(C\cup\Gamma)=d+n$ and $3e^2+1=p_a(C\cup \Gamma) = g+C\cdot \Gamma -1$ by \eqref{eq:genus_of_C_union_Gamma}. Let $H$ denote a general hyperplane section of $S$ and consider the divisor $D:=eH-C-\Gamma$.  We observe that we have 
\begin{gather*}
H^2 = \deg(S)=6, \quad H\cdot C = d,\quad H \cdot \Gamma = n,\\
C \cdot \Gamma \geq 3n+1, \quad C^2 = 2g-2, \quad \Gamma^2 = 2p_a(\Gamma)-2 = -2.
\end{gather*}
Using these intersection numbers, we can calculate that 
$$ H\cdot D = eH^2-H\cdot C-H\cdot \Gamma = 6e-(d+n)=0.$$
Furthermore, we have
$$ D^2 = e H\cdot D-C\cdot D- \Gamma \cdot D= -e\underbrace{(d+n)}_{= 6e}+2g-4+2C\cdot \Gamma=-2p_a(C\cup \Gamma)+ 2(g-1+C\cdot \Gamma)=0.$$
From \Cref{lem:Hodge_index_thm}, it now follows that $D=0$ in $\Pic(S)$ and hence $D\cdot D'=0$ for all divisors $D'$ on $S$. In particular, we have
$$ 0 = D\cdot \Gamma = eH\cdot \Gamma - C\cdot \Gamma - \Gamma^2 = en-C\cdot \Gamma +2,$$
and hence $ C\cdot \Gamma = en+2$. This intersection number satisfies $C\cdot \Gamma \geq 3n+1$ only if either $e=3$, or $e=2$ and $n=1$. In the latter case, we have $d=6e-n=11$ and $C\cdot \Gamma=en+2=4$.
This contradicts, however, the fact that $$g\overset{\textup{\eqref{eq:genus_of_C_union_Gamma}}}{=} p_a(C\cup\Gamma)-C\cdot \Gamma + 1= 13-4 + 1 =10,$$ 
and the fact that the only pairs $(g,d)\in\mathcal{P}$ with $d=11$ satisfy $g\in\{8,9\}$. So we must have $e=3$. But then $n=18-d$ and we have
$$28= 3e^2+1= p_a(C\cup \Gamma) \overset{\textup{\eqref{eq:genus_of_C_union_Gamma}}}{\geq} g+ 3n = g+ 54-3d.$$ 
This is equivalent to $26-3d+g\leq 0$. However, all pairs $(g,d)\in\mathcal{P}$ satisfy $26-3d+g>0$. \par
Therefore, $C\cup \Gamma \sub S $ cannot be a complete intersection of $S$ with another hypersurface. So we can apply \subcref{prop:bound_curves_not_in_K3surface_of_degree_6_mod_12345} of \Cref{prop:bound_curves_in_K3surface_of_degree_6} and obtain that 
$$g + 3n \overset{\textup{\eqref{eq:genus_of_C_union_Gamma}}}{\leq} p_a(C\cup \Gamma) \overset{\textup{\ref{prop:bound_curves_in_K3surface_of_degree_6}}}{\leq} \frac{\deg(C\cup \Gamma)^2 - 1}{12} = \frac{d^2 + 2dn + n^2 - 1}{12}.$$
Rewriting this inequality yields
$$0 \leq n^2 - (36-2d)n+ d^2-12g-1 =: P_{g,d}(n).$$
For all $(g,d)\in \mathcal{P}$, we can compute the polynomial $P_{g,d}(n)$ and determine for which values of $0<n<18-d$ this polynomial is non-negative. We denote the largest such $n$ by $n_{\max}$. Then we obtain the values in \Cref{table:max_degree_of_3n+1_secant}, which are the leftmost number in the fourth column (so in this column, they either appear as a single or a crossed-out number). \par
For $(g,d)=(1,6)$, we can eliminate the case $n=1$ as follows: Suppose that $L$ is a 4-secant line of a curve $C$ of degree 6 and genus 1. Then the divisor $D=C+L$ is a divisor of degree 7 and genus $p_a(D)=p_a(C)+p_a(L)+C\cdot \Gamma -1=4$. It follows from \Cref{prop:bound_curves_not_in_K3surface_of_degree_6_mod_12345} that the divisor $D$ is of the form $D=L'+\Gamma$ where $L' \sub S$ is a line and $\Gamma\sub S$ is a curve of degree 6 and genus 4 satisfying $\Gamma\cdot L'=1$. This is contradicting the fact that $D=L+C$ and that $C$ has genus $1$. \par
We can improve some other bounds from \Cref{table:max_degree_of_3n+1_secant} by finding a contradictory divisor $D\in \Z H \oplus \Z C \oplus \Z \Gamma \sub \Pic(S)$, specified in \Cref{table:finding_contradictory_divisors}, as follows: \par 
Using the bounds $B(.)$ from \Cref{table:genus_bounds}, we can determine the possible intersection numbers $C\cdot \Gamma$ in terms of $g,d$ and $n$ by
$$ 3n+1 \leq C\cdot \Gamma \overset{\textup{\eqref{eq:genus_of_C_union_Gamma}}}{=} p_a(C\cup\Gamma)+1-g \leq B(d+n)+1-g.$$
Given a divisor $D=aH+bC+c\Gamma\in \Pic(S)$ with $a,b,c\in\Z$, we can then calculate 
\begin{gather*}
H\cdot D = 6a+db+nc,\quad C\cdot D =da+ (2g-2)b+ (C\cdot \Gamma)c\\
\Gamma \cdot D = na + (C\cdot \Gamma)b-2c, \quad D^2 = (H\cdot D)a + (C\cdot D)b+ (\Gamma \cdot D)c.
\end{gather*}
For $(g,d)=(2,8)$, the divisor $D=2H-C-2\Gamma$ has the properties $H\cdot D=0$ and $D^2=-2$. \Cref{lem:Hodge_index_thm} implies that $D=0$ in $\Pic(S)$. However, this contradicts the fact that $\Gamma \cdot D=-1 \neq 0$. \par 
Now consider the other pairs $(g,d)\in\mathcal{P}\setminus \{(1,6), (2,8)\}$, for which the number $n_{\max}$ in \Cref{table:max_degree_of_3n+1_secant} is crossed out. For these pairs, we can find a divisor $D\in\Pic(S)$, which is specified in \Cref{table:finding_contradictory_divisors}, with the following properties: The divisor $D$ satisfies $D^2=-2$ and $\:H\cdot D \geq 1$ and therefore corresponds to an effective rational curve of degree $H\cdot D$ by \Cref{lem:Riemann_Roch_effective_divisor}. Moreover, we either have $\deg(D)<d$ and $C\cdot D <0$, or $\deg(D)<n$ and $\Gamma \cdot D<0$. In other words, the intersection of the effective irreducible curve $C$ or $\Gamma$ with the effective (possibly reducible or non-reduced) curve $D$ is negative. However, this can only be true if $C\sub \supp(D)$ or, respectively, if $\Gamma \sub \supp(D)$. This is impossible since the degree of $D$ is strictly less than the degrees $d$ of $C$ or the degree $n$ of $\Gamma$, respectively. \par 
So in all of these cases, we find a contradiction and we can decrease $n_{\max}$, obtaining the bounds next to the crossed numbers in the fourth column of \Cref{table:max_degree_of_3n+1_secant}. In particular, we get that $n_{\max} \leq 2$. 
\begin{table}[H]
\caption{If we suppose that a smooth irreducible curve $C\sub Y$ of degree $d$ and genus $g$ is contained in a smooth hypercubic section $S$ of a smooth hyperquadric $Y\sub \Pfour$ and that $C$ admits an irreducible curve $\Gamma\sub Y$ of degree $n$ that is at least $3n+1$-secant to $C$, then we can find a contradictory effective divisor $D\in \Pic(S)$ that has a negative intersection with either $C$ or with $\Gamma$.}
\label{table:finding_contradictory_divisors}
\begin{minipage}{\textwidth}
	\centering
	\vspace{0.5cm}
	\begin{tabular}{|l|l|l|l|l|l|l|l|l|}
		\hline
		$g$ & $d$ & $n$ & $C\cdot \Gamma$ & $D\in\Pic(S)$ & $H\cdot D$ & $D^2$ & $C\cdot D $ & $\Gamma \cdot D$ \\
		\hline
		0 & 7&2&7& $-H+C+\Gamma$ & 3 & -2 & -2 & 3\\
		\hline		
		0&8&3 & 10& $H-\Gamma$ & 3 & -2 & -2 & 5\\
		\hline 
		0&8&3 & 11& $H-\Gamma$ & 3 & -2 & -3 & 5\\
		\hline 
		1 & 8 & 3 & 10& $H-\Gamma$ & 3 & -2 & -2 & 5 \\
		\hline
		2 & 7 & 1 & 4& $-H+C+\Gamma$& 2&-2 & -1 & 1 \\
		\hline 
		2 & 8 & 2 & 7& $2H-C-2\Gamma$ & 0 & -2 & 0 & -1\\
		\hline
		2 & 9 & 4 & 13 & $-2H+C+\Gamma$ & 1 & -2 &-3 & 3 \\
		\hline 
		2 & 9 & 3 &10 &$H-\Gamma$&3&-2&-1&5\\
		\hline
		4&9&2&7&$2H-C-\Gamma$& 1 &-2&5&-1\\
		\hline
		5&10&3&10&$-2H+C+\Gamma$&1&-2&-2&2\\
		\hline 
		8&11&2&7&$-2H+C+\Gamma$ &1&-2&-1&1\\
		\hline
	\end{tabular}
\end{minipage}
\end{table}
For all $(g,d)$-pairs, we have thereby shown that $\Gamma$ is a $3n+1$ secant of $C$ with $n\leq n_{\max}$. This contradicts our initial hypotheses. Therefore, our assumption that $-K_X$ is not nef was false. So $-K_X$ must be nef and thus $X$ is a weak Fano threefold by \Cref{cor:weak_Fano_big_nef_basepoinfree} and by the fact that $26-3d+g>0$ for all $(g,d)\in \mathcal{P}$.
\end{proof}
\begin{rmk}
	Let $C\sub \Pthree$ be a smooth irreducible curve of degree $d$ and genus $g$ lying on a smooth quartic surface $S\sub \Pthree$. Then by \cite{Blanc_Lamy_Weak_Fano_threefolds}, the blowup $X$ of $\Pthree$ along $C$ is weak Fano if and only if $4d-30\leq g\leq 14$ or $(g,d)=(19,12)$, and $C$ does not admit a 5-secant line, 9-secant conic nor 13-secant twisted cubic. Moreover by \cite[Prop.\ 5.8]{Blanc_Lamy_Weak_Fano_threefolds}, 13-secant cubics are only possible for $(g,d)\in \{(0,7),\:(2,8),\:(3,8)\}$ and 9-secant conics may only appear for $(g,d)\in\{(6,9),\:(7,9)\}$. \par 
	Supposing that there is an irreducible curve $\Gamma \sub \Pthree$ of degree $n\in \{2,3\}$ that is $4n+1$-secant to $C$, we must have $\Gamma \sub S$ by Bézout's theorem. As in the proof of \Cref{thm:general_case_in_list_implies_weak_Fano}, we can then try to find an effective divisor $D=aH+bC+c\Gamma\in \Z H \oplus \Z C \oplus \Z \Gamma \sub \Pic(S)$ with $a,b,c\in\Z$ such that 
	\begin{itemize}
		\item $D^2=-2$ and $H\cdot D \geq 1$. So $D$ corresponds to an effective rational curve of degree $H\cdot D$ on $S$ by \Cref{lem:Riemann_Roch_effective_divisor}. (The lemma is applicable since $S\sub \Pthree$ is a smooth K3-surface.) 
		\item $\deg(D)<d$ and $C\cdot D <0$, or $\deg(D)<n$ and $\Gamma \cdot D<0$, which both contradicts the fact that the intersection number of two effective curves with no common irreducible component is always positive.
		\item $H\cdot D = 4a+db+nc,\: C\cdot D =da+ (2g-2)b+ (C\cdot \Gamma)c, \: \Gamma \cdot D = na + (C\cdot \Gamma)b-2c$ and $D^2 = (H\cdot D)a + (C\cdot D)b+ (\Gamma \cdot D)c$.
		\item Applying \cite[Prop.\ 2.6]{Blanc_Lamy_Weak_Fano_threefolds}, we can bound the genus of $C\cup \Gamma$, and thereby the possible intersection numbers $C\cdot \Gamma$, in terms of $g,d$ and $n$ by
		$$ 4n+1 \leq C\cdot \Gamma \overset{\ref{lem:genus_formula_union_of_curves}}{\leq} p_a(C\cup\Gamma)+1-g < \frac{1}{8}(d+n)^2+1-g.$$
	\end{itemize}
	For $(g,d)\in\{(0,7),\:(2,8),\:(3,8),\:(7,9)\}$, we find the contradictory divisor $D\in\Pic(S)$ as it is given in \Cref{table:finding_contradictory_divisors_Pthree}. For $(g,d)=(6,9)$, we can only find such a divisor $D$ if $C\cdot \Gamma= 10$, but not if $C\cdot \Gamma =9$.	In other words, we have shown that the assumptions about 13-secant twisted cubics in \cite{Blanc_Lamy_Weak_Fano_threefolds} can be eliminated and that 9-secant conics might only be possible for $(g,d)=(6,9)$. \par 
	This is almost the same result as in the paper \cite{Jean_Dalmeida}, in which minimal free graded resolutions are used to show that 9-secant conics and 13-secant twisted cubics are never possible and it therefore suffices to assume that no 5-secant lines exist in \cite{Blanc_Lamy_Weak_Fano_threefolds}.
\end{rmk}
\begin{table}[H]
	\caption{If we suppose that a smooth irreducible curve $C\sub \Pthree$ of degree $d$ and genus $g$ is contained in a smooth quartic surface $S\sub\Pthree$ and that $C$ admits an irreducible curve $\Gamma\sub \Pthree$ of degree $n$ that is at least $4n+1$-secant to $C$, then we can find a contradictory effective divisor $D\in \Pic(S)$ that has a negative intersection with either $C$ or with $\Gamma$.}
	\label{table:finding_contradictory_divisors_Pthree}
	\begin{minipage}{\textwidth}
		\centering
		\vspace{0.5cm}
		\begin{tabular}{|l|l|l|l|l|l|l|l|l|}
			\hline
			$g$ & $d$ & $n$ & $C\cdot \Gamma$ & $D\in \Pic(S)$ & $H\cdot D$ & $D^2$ & $C\cdot D $ & $\Gamma \cdot D$ \\
			\hline
			0 & 7&3&13& $-2H+C+\Gamma$ & 2 & -2 & -3 & 5\\
			\hline		
			0&7&2 & $m\in\{9,10,11\}$& $H-\Gamma$ & 2 & -2 & $7-m$ & 4\\
			\hline 
			2&8&3 & 14& $3H-C-\Gamma$ & 1 & -2 & 8 & -3\\
			\hline 
			2 & 8 & 3 & 13& $-2H+C+\Gamma$ & 3 & -2 & -1 & 5 \\
			\hline
			2 & 8 & 2 & $m\in\{9,10,11\}$& $H-\Gamma$ & 2 & -2 & $8-m$ & 4\\
			\hline
			3 & 8 & 3 & 13 & $3H-C-\Gamma$ & 1 & -2 &7 & -2 \\
			\hline 
			3 & 8 & 2 &$m\in\{9,10\}$ &$H-\Gamma$&2&-2&$8-m$&4\\
			\hline
			6&9&2&10&$H-\Gamma$& 2 &-2&-1&4\\
			\hline
			7&9&2&9&$3H-C-\Gamma$&1&-2&6&-1\\
			\hline 
		\end{tabular}
	\end{minipage}
\end{table}
The proof of the theorems \ref{thm:A} and \ref{thm:B} is now just a matter of assembling the previous results.
\begin{proof}[Proof of \Cref{thm:A} and \Cref{thm:B}] 
If the threefold $X$ is weak Fano, we have shown that $C$ is contained in a smooth hypercubic section of $Y$ (\Cref{prop:weak_Fano_implies_contained_in_smooth_hypercubic_section}) and that $C$ admits no irreducible curves of degree $n$ that are $3n+1$-secant to $C$ (\Cref{lem:3n+1_secants}). In particular, $C$ has no 4-secant lines and no 7-secant conics. Moreover, we have seen in \Cref{lem:general_case_weak_Fano_implies_in_list} that condition (iii) of \Cref{thm:A} is satisfied. By the equivalence of this condition with the condition $(g,d)\in \mathcal{P}:=\mathcal{P}_{\textup{none}} \cup \mathcal{P}_{\textup{line}} \cup \mathcal{P}_{\textup{conic}}$ (\Cref{lem:numerical_description_of_P}), this gives one direction of \Cref{thm:A} and \ref{thm:B}.\par 
The converse directions follow directly from \Cref{thm:general_case_in_list_implies_weak_Fano} and again by the equivalence of \Cref{thm:A_g_and_d} with $(g,d)\in \mathcal{P}$ (\Cref{lem:numerical_description_of_P}). 
\end{proof}
\begin{rmk} We can ask whether the three geometric conditions on $C$ in \Cref{thm:general_case_in_list_implies_weak_Fano} can be weakened:
\begin{itemize}
\item In the following \Cref{section:hyperplane}, we will show that the assumption that $C$ is contained in a smooth hypercubic section can be eliminated if $C$ is contained in a hyperplane section or in a smooth hyperquadric section of $Y$. \par 
More precisely, for pairs in the set $$\mathcal{P}_{\textup{plane}}:= (g,d)\in\{(0,1),\:(0,2),\:(0,3),\:(0,4),\:(1,4),\:(2,5),\:(4,6)\},$$ \Cref{thm:general_case_in_list_implies_weak_Fano} yields $n_{\max}=0$. This means that it suffices to assume that  $(g,d)\in \mathcal{P}_{\textup{plane}}$ and that $C$ is contained in a smooth K3-surface to conclude that $X$ is weak Fano. In \Cref{section:hyperplane}, however, we will show that this latter assumption is not necessary since the fact that $(g,d)\in \mathcal{P}_{\textup{plane}}$ directly implies that $X$ is weak Fano (\Cref{prop:hyperplane_section_implies_weak_Fano}). \par 
Denote by $\mathcal{P}_{\textup{quadric}} $ the set \begin{gather*}
	\{ (0,4),\:(0,5),\:(0,6),\:(1,5),\:(1,6),\:(2,6),\:(2,7),\:(3,7),\:(3,8)\} \:\cup 
	\\ \{(3,8),\:(4,8),\:(5,8),\:(6,9),\:(8,10),\:(13,12)\}.
\end{gather*} \Cref{thm:general_case_in_list_implies_weak_Fano} gives $n_{\max}=1$ for $(g,d)\in\{(0,6),\:(3,8)\}$ and $n_{\max}=0$ for $(g,d)\in\mathcal{P}_{\textup{quadric}}\setminus \{(0,6),\:(3,8)\}$. In other words, we know that $X$ is weak Fano if $C$ has the following properties: $(g,d)\in \mathcal{P}_{\textup{quadric}}$, $C$ has no 4-secant lines in the case $(g,d)\in\{(0,6),\:(3,8)\}$ and $C$ is contained in a smooth hypercubic section of $Y$. This latter condition can be replaced by the assumption that the curve $C$ is contained in a smooth hyperquadric of $Y$ (\Cref{lem:quartic_del_Pezzo_in_list_implies_weak_Fano}).
\item The assumption that $(g,d)\in \mathcal{P}$ cannot be dropped since there are pairs $(g,d) \not \in \mathcal{P}$ yielding curves that are contained in a smooth hypercubic section of $Y$, have no 4-secant lines or 7-secant conics, but whose corresponding threefold $X$ is not weak Fano. \par 
For instance, \Cref{thm:existence_curves_in_K3_surface_degree6_smooth_quadric} implies that there is a smooth hyperquadric $Y\sub \Pfour$ admitting a smooth hypercubic section $S$ with the following properties: $S$ contains an irreducible smooth curve $C$ of degree 9 and genus 0 and satisfies $\Pic(S)=\Z C\oplus \Z H_S$, where $H_S$ denotes a general hyperplane section of $S$ . $C$ cannot have 4-secant lines nor 7-secant cubics since any such secant-curve would have to be contained in $S$ by Bézout's theorem. But $S$ does not contain any line or integral conic since $d=9\equiv_6 3$ (\Cref{lem:line_conic_cubic_in_K3}).
\item In general, it is not possible to eliminate the condition about 4-secant lines and 7-secant-conics in  \Cref{thm:A}. This will be shown in the the following two examples.
\end{itemize}
\end{rmk} 
\begin{ex}
\label{ex:4_secant_line_neccessary}
It is possible to find a curve $C\sub \Pfour$ of degree $d$ and genus $g$, a smooth hyperquadric $Y\sub \Pfour$ and a smooth hypercubic section $S$ of $Y$ such that $C\sub S$ has a 4-secant line $L \sub S$ and $(g,d)\in \mathcal{P}$. In other words, the assumption that $C$ does not admit 4-secant lines in \Cref{thm:general_case_in_list_implies_weak_Fano} is necessary. \par 
To find an explicit example for such a curve $C$, we denote the polynomial ring in 5 variables by $\kk[u,v,x,y,z]$ and consider the rational cubic $K=\VPfour(z, x^2-vy,vx-uy,v^2-ux)$. $K$ is the image of the embedding
$$\begin{matrix}
\Pone & \hookrightarrow & \Pthree &\hookrightarrow &\Pfour\\
[s:t]&\mapsto&[s^3:s^2t:st^2:t^3]&\mapsto&[s^3:s^2t:st^2:t^3:0]
\end{matrix}.$$
The line $L:=\VPfour(u,v,y)\sub \Pfour$ does not intersect $K$. We can see immediately that the polynomials $f:=z(u^2-vz+xy)+y(x^2-vy)$, $g:=v^2-ux-zy$ and $h:=vx-uy-zu$ vanish on $L\cup K$. So $L\cup K$ is contained in the hypercubic 
$Z:=\VPfour(f)\sub\Pfour$, as well as in the two hyperquadrics $Y:=\VPfour(g), \:Q:=\VPfour(h)\sub\Pfour$.
Moreover, an explicit calculation shows that the hyperquadric $Y\sub \Pfour$ and the hypercubic section $S:=Y\cap Z$ are smooth. \par 
The curve $R:=Y\cap Q\cap Z$ is of degree 12 and genus 13 (\Cref{lem:genus_formula_complete_intersection}). On $S$, we consider the divisor $C:=R-K-L$, which is an effective curve of degree 8. 
Using the primary decomposition of the vanishing ideal $I_{\Pfour}(R)$ of $R$, we can calculate that the vanishing ideal of $C$ is given by $(f,g,h,p_1,p_2,q)$, where
\begin{gather*}
p_1:= u(uv+xy+y^2-xz+yz)-y(y^2+z^2),\\
p_2:=u(u^2+vy+xy-vz)-xy^2,\\
q:=x^4-uxy^2+z(x^3-2uxy-uy^2+y^3)+z^2(y^2-2ux-2uy)+z^3(y-u)+z^4.\\
\end{gather*}
This allows us to check that $C$ is a smooth irreducible curve with intersection numbers $C\cdot L=4$ and $C\cdot K=8$. The genus of $C$ can then be determined using the formula from \Cref{lem:genus_formula_union_of_curves}:
$$13=p_a(C+L+K) = p_a(C)+p_a(L)+p_a(K)+C\cdot L+ C\cdot K+K\cdot L-2=p_a(C)+10.$$
It follows that $C$ is a smooth irreducible curve of type $(g,d)=(3,8)\in\mathcal{P}$, which admits a 4-secant line $L$ and is contained in a smooth hyperquadric $Y\sub \Pfour$, as well as in a smooth hypercubic section $S$ of $Y$. Note that $C$ does not admit a 7-secant conic since $n_{\max}=1$ in \Cref{thm:general_case_in_list_implies_weak_Fano}.
\end{ex}
\begin{ex}
\label{ex:7_secant_conic_neccessary}
It is possible to find a curve $C\sub \Pfour$ of degree $d$ and genus $g$, a smooth hyperquadric $Y\sub \Pfour$ and a smooth hypercubic section $S$ of $Y$ such that $C\sub S$ has a 7-secant conic $\Gamma \sub S$ and $(g,d)\in \mathcal{P}$. In other words, the assumption that $C$ does not admit 7-secant conics in \Cref{thm:general_case_in_list_implies_weak_Fano} is necessary. \par 
To find an explicit example for such a curve $C$, we consider the following polynomials in $\kPfour$:
\begin{gather*}
f:=x_0^2-x_1x_2,\: l_1:= 3x_0+2x_1+8x_2+2x_3,\: l_2:=4x_0-3x_1+x_2-x_3,\\
q_1:=x_1^2-5x_1x_2+2x_2^2+x_3x_1-2x_3^2,\: q_2:=x_0x_2+2x_1x_2-2x_2^2.
\end{gather*}
They define the following hypersurfaces in $\Pfour$:
\begin{gather*}
Y = \VPfour(f+x_3^2-x_1x_4+x_0x_3+x_2x_3+x_0x_4+x_2x_4), \\Z_1=\VPfour(x_3 q_1+x_4 q_2+x_3x_4 l_1+ x_4^2l_2),\: Z_2=\VPfour(fx_3+x_2x_4^2).
\end{gather*}
An explicit calculation shows that the hyperquadric $Y$ and the hypercubic section $S:=Z_1\cap Y$ are smooth. We know moreover that the curve $R:=Y\cap Z_1 \cap Z_2$ is of degree 12. \par 
Observe that $R$ has multiplicity 2 along the conic $\Gamma := \VPfour(f,x_3,x_4) \sub \Pfour$ since the polynomial $fx_3+x_2x_4^2$ determining $Z_2$ vanishes with multiplicity 2 along $\Gamma$. In order to extract the other irreducible components of $R$, we determine the primary decomposition of the vanishing ideal $I:=I_{\Pfour}(Y\cap Z_1\cap Z_2)$ and obtain that $I=I_1\cap I_2 \cap I_2$, where
\begin{enumerate}
\item $I_1$ is the vanishing ideal of a smooth irreducible curve $C\sub \Pfour$ of degree 8 and genus 0.
\item $I_2$ is the vanishing ideal of a smooth irreducible curve $D\sub\Pfour$ of degree $6$ and genus $1$.
\item The radical of $I_3$ is the vanishing ideal of $\Gamma$ and $I_3$ corresponds to the vanishing ideal of $2\Gamma$ (seen as a divisor on $S$).
\end{enumerate}  
Computing the number of intersection points of these curves, we find that $D\cdot C=12, \: D\cdot \Gamma=3,\:C\cdot \Gamma = 7$.
We have therewith found a smooth irreducible curve $C$ of type $(g,d)=(0,8)\in\mathcal{P}$, which admits a 7-secant conic $\Gamma$ and is contained in a smooth hyperquadric $Y\sub \Pfour$, as well as in a smooth hypercubic section $S$ of $Y$. \par 
Note that $C$ does not admit a 4-secant line since such a line would have to be contained in the intersection $Y\cap Z_1 \cap Z_2$ and then its vanishing ideal would have to appear in the primary decomposition of $I$.
\end{ex}
\subsection{Existence of curves yielding a weak Fano threefold} \label{subsection:existence}
The aim of this subsection is to show that, for every pair $(g,d)\in \mathcal{P}$, the weak Fano threefold $X$ given by \Cref{thm:A} is geometrically realizable. In other words, we want to find a smooth irreducible curve $C$ and a smooth hyperquadric $Y \sub \Pfour$ that satisfy the conditions of \Cref{thm:general_case_in_list_implies_weak_Fano}. To achieve this, we will use the K3-surface $S$ given by \Cref{thm:knutsen_existence_curves_in_K3_surface_degree6} to find such a curve $C$ in $S$. \par 
The assumptions of \Cref{thm:general_case_in_list_implies_weak_Fano} are only satisfied if the curve $C$ has no 4-secant-lines or 7-secant-conics. In order to check this condition, we use the following lemma to examine the existence of lines and conics on a K3-surface $S$, which is given by \Cref{thm:knutsen_existence_curves_in_K3_surface_degree6}. \par 
Moreover, the following lemma provides conditions for the existence of a rational cubic in $S$. These conditions will be used in \Cref{thm:existence_curves_in_K3_surface_degree6_smooth_quadric} to show that some of the sextic K3-surfaces of \Cref{thm:knutsen_existence_curves_in_K3_surface_degree6}, which are given by the complete intersection of a quadric and cubic hypersurface of $\Pfour$, are contained in a smooth hyperquadric.
\begin{lem} \label{lem:line_conic_cubic_in_K3}
Let $S\sub \Pfour$ be a smooth K3-surface of degree 6 containing an irreducible smooth curve $C\sub S$ of genus $g$ and degree $d$ such that either $\Pic(S)= \Z H_S$ or $\Pic(S)=\Z H_S \oplus \Z C$, where $H_S$ denotes a general hyperplane section of $S$. (Note that the surfaces given by \Cref{thm:knutsen_existence_curves_in_K3_surface_degree6} satisfy these conditions.) Then:
\begin{enumlem}
\item $S$ contains a line if and only if $d\equiv_6 r$ with $r\in\{-1, 1\}$ and $g = \frac{d^2-1}{12}$. \label{lem:line_in_K3}
\item If $S$ contains an integral conic, then if $d\equiv_6 r$ with $r\in \{2,4\}$ and $g=\frac{d^2-4}{12}$. \label{lem:conic_in_K3}
\item If $S$ contains an integral rational cubic, then $d\equiv_6 3$ and $g=\frac{d^2-9}{12}$. \label{lem:cubic_in_K3}
\end{enumlem}
Furthermore, we are in exactly one of the following cases: (1) $S$ contains a unique line. (2) $S$ contains a unique integral conic. (3) $S$ contains two distinct integral rational cubics. (4) $S$ contains no line, no integral conic and no integral rational cubic.
\end{lem}
\begin{proof}
We first suppose that we are in case that $\Pic(S)=\Z H_S$. Then every effective divisor $D\in\Pic(S)$ is of the form $D=eH_S$ for some $e\in \Z_{>0}$ and thus of degree $D\cdot H_S=eH_S^2=6e$. In particular, we must have $d\equiv_6 0$ and $S$ cannot contain any curve of degree 1, 2 or 3. \par 
We now assume that we are in the second case, where $\Pic(S)=\Z H\oplus \Z C$. Given a rational curve $\Gamma \in \Pic(S)$, there are $a,\:b\in \Z$ such that $\Gamma=aH+bC$. Note that we have $\Gamma^2 =2p_a(\Gamma)-2=-2$ (\Cref{lem:genus_formula_union_of_curves}). Denoting $x:=\deg(\Gamma)=\Gamma\cdot H$, we calculate: 
\begin{gather*} \label{eq:formulas_g_d_curve_in_smoothK3_1}
x = \Gamma \cdot H = a H^2 + b C\cdot H = 6a + db, \tag{$\triangle_1$}\\
-2 = \Gamma^2 = a^2 H^2+ 2abH\cdot C + b^2 C^2= 6a^2 + 2abd + (2g-2)b^2. \nonumber
\end{gather*}
Substituting $a$ for $\frac{x-db}{6}$ in the second equation gives us
$12+x^2=b^2(d^2-12g+12)$. This implies that $ b^2$ divides $12+x^2$ and that \begin{equation} \label{eq:formulas_g_d_curve_in_smoothK3_2}
g =\frac{d^2b^2+12b^2-12-x^2}{12b^2}. \tag{$\triangle_2$}
\end{equation}
Also observe that the linear equivalence class of an integral rational curve in $S$ defines a unique effective curve in $S$: If $D\sub S$ is an integral rational curve and $D'\sub S$ is an effective curve linearly equivalent to $D'$, then $D\cdot D'=D^2=-2$. Since $D$ is integral, this is only possible if $D\sub \supp(D')$. Now the fact that $\deg(D)=D\cdot H=D'\cdot H=\deg(D')$ implies that we must have $D=D'$.   \par
Furthermore, note that if $D=aH+bC$ for some $a,b\in \Z$, then the integers $a$ and $b$ are unique (since $H$ and $C$ generate $\Pic(S)$ by assumption). For both of these reasons, we can prove the uniqueness of an integral rational curve $D$ in $S$, assuming its existence, as follows: It is enough to show that there is a unique pair $(a,b)\in \Z^2$ such that the linear equivalence class of $D$ in $\Pic(S)$ is given by $D=aH+bC$. \par 
We first assume that $\Gamma$ is a line. Then $x=\deg(\Gamma)=1$ and thus $b^2$ divides $12+x^2=13$. So we must have $b\in\{\pm1\}$. Given $b$, the integer $a$ is then uniquely determined by \eqref{eq:formulas_g_d_curve_in_smoothK3_1}. More precisely, we have the following: In $\Pic(S)$, $L$ is uniquely given by $L=aH+bC$, where $a=\frac{1-db}{6}$, and $b=1$ if $d\equiv_6 1$, respectively, $b=-1$ if $d\equiv_6 5$. \par 
From \eqref{eq:formulas_g_d_curve_in_smoothK3_1} and \eqref{eq:formulas_g_d_curve_in_smoothK3_2}, we deduce moreover that $$g=\frac{d^2-1}{12},\quad d\equiv_6 r,\quad r\in \{\pm 1\}.$$ 
\par 
Conversely, this divisor $\Gamma=aH+bC$ is a well-defined element of $\Pic(S)$ if $d\equiv_6 r$ with $r\in\{-1, 1\}$ and $g = \frac{d^2-1}{12}$. Since we then have $\Gamma \cdot H=1$ and $\Gamma^2 = -2$, the divisor $\Gamma$ is a line by \Cref{lem:Riemann_Roch_effective_divisor}.\par 
Similarly, if $\Gamma$ is an integral rational cubic, we have $x=3$. Now the fact that $b^2$ divides $12+x^2=21$ also implies that $b=\pm 1$. The integer $a$ is then again uniquely determined by $b$ and \eqref{eq:formulas_g_d_curve_in_smoothK3_1}. More precisely, we have the following: If $S$ contains integral rational cubics, there exist two unique ones, which are given by $aH+bC$ where $b\in\{-1,\:1\}$ and $a=\frac{3-db}{6}$. \par
Using \eqref{eq:formulas_g_d_curve_in_smoothK3_1} and \eqref{eq:formulas_g_d_curve_in_smoothK3_2} gives
$$ g=\frac{d^2-9}{12}, \quad d\equiv_6 r, \quad r=3.$$
\par 
Lastly, if $\Gamma$ is an integral conic, then $x=2$. As $b^2$ divides $12+x^2=-16$, we have $b\in\{\pm 1, \pm 2, \pm 4\}$. For the case $b=\pm 2$ (respectively $b=\pm 4$), the equality
$$ g\overset{\eqref{eq:formulas_g_d_curve_in_smoothK3_2}}{=} \frac{d^2}{12} +\frac{2}{3}\: \left( \textup{respectively } g\overset{\eqref{eq:formulas_g_d_curve_in_smoothK3_2}}{=}\frac{d^2}{12}+\frac{11}{12}\right)$$
contradicts the statements of \Cref{thm:knutsen_existence_curves_in_K3_surface_degree6}. So we must have $b=\pm 1$, which again uniquely determines $a$. If it exists, an integral conic in $S$ is uniquely given by $aH+bC$, where $a=\frac{2-db}{6}$ and $b=1$ if $d\equiv_6 2$, respectively, $b=-1$ if $d\equiv_6 4$. \par 
It follows from \eqref{eq:formulas_g_d_curve_in_smoothK3_1} and \eqref{eq:formulas_g_d_curve_in_smoothK3_2} that 
$$ g = \frac{d^2-4}{12}, \quad d\equiv_6 r, \quad r\in\{\pm 2\}.$$ 
We have shown that \subcref{lem:line_in_K3}-\subcref{lem:cubic_in_K3} hold and that the existence of a line or integral conic (respectively integral rational cubic) implies uniqueness of the line or conic (respectively that there are exactly two cubics). It is left no note that the cases \subcref{lem:line_in_K3}, \subcref{lem:conic_in_K3} and \subcref{lem:cubic_in_K3} are mutually exclusive since they apply to different $(g,d)$-pairs, as we have shown.
\end{proof}
We can now use the previous lemma to refine the statement of \Cref{thm:knutsen_existence_curves_in_K3_surface_degree6} in the case where the hyperquadric $Y$ containing the smooth irreducible curve $C$ is smooth. The smoothness of $Y$ will eliminate \Cref{thm:knutsen_existence_curves_in_K3_surface_degree6_2} as a possible case, as we have already seen in \Cref{prop:bound_curves_in_K3surface_of_degree_6}. Conversely, assuming that \Cref{thm:knutsen_existence_curves_in_K3_surface_degree6_1} or \Cref{thm:knutsen_existence_curves_in_K3_surface_degree6_3} hold, we will show the existence of a smooth hyperquadric $Y$ containing a corresponding curve $C$.
\begin{prop}
\label{thm:existence_curves_in_K3_surface_degree6_smooth_quadric}
There exists a smooth hyperquadric $Q\sub\Pfour$ and a hypercubic $Z\sub \Pfour$ such that their intersection $S=Q\cap Z$ is a smooth K3-surface and contains an irreducible smooth curve $C$ of degree $d$ and genus $g$ if and only if one of the following cases applies:
\begin{enumprop}
\item $g=\frac{d^2}{12}+1$,\label{thm:existence_curves_in_K3_surface_degree6_smooth_quadric_1}
\item $g\leq\frac{d^2-1}{12}$ and $(g,d)\neq (4,7)$. \label{thm:existence_curves_in_K3_surface_degree6_smooth_quadric_2}
\end{enumprop}
Furthermore, we can assume that $\Pic(S)=\Z H_S$ in case \subcref{thm:existence_curves_in_K3_surface_degree6_smooth_quadric_1} and that $\Pic(S)=\Z H_S \oplus \Z C$ in case \subcref{thm:existence_curves_in_K3_surface_degree6_smooth_quadric_2}, where $H_S$ denotes a general hyperplane section of $S$.
\end{prop}
\begin{proof}
If $Q, Z, S$ and $C$ as in the statement of the theorem exist, then \Cref{prop:bound_curves_in_K3surface_of_degree_6} shows that we must either have $g=\frac{d^2}{12}+1$, or $g\leq\frac{d^2-1}{12}$ and $(g,d)\neq (4,7)$. \par 
Conversely, given that either \subcref{thm:existence_curves_in_K3_surface_degree6_smooth_quadric_1} or \subcref{thm:existence_curves_in_K3_surface_degree6_smooth_quadric_2} hold, \Cref{thm:knutsen_existence_curves_in_K3_surface_degree6} guarantees the existence of a hyperquadric $Q\sub\Pfour$ and a hypercubic $Z\sub \Pfour$ such that their smooth intersection $S=Q\cap Z$ contains an irreducible smooth curve $C$ of degree $d$ and genus $g$. Moreover, \Cref{thm:knutsen_existence_curves_in_K3_surface_degree6} shows that we can assume that $\Pic(S)=\Z H_S$ (respectively $\Pic(S)=\Z H_S \oplus \Z C$) in case \subcref{thm:existence_curves_in_K3_surface_degree6_smooth_quadric_1} (respectively case \subcref{thm:existence_curves_in_K3_surface_degree6_smooth_quadric_2}).\par 
It is left to show that the hyperquadric $Y$ is smooth. To find a contradiction, we suppose to the contrary that $Y$ is singular. Then $Y$ contains a plane $H_1\cap H_2$ for some hyperplanes $H_1,\:H_2 \sub \Pfour$ (\Cref{rmk:singular_hyperquadric_contains_plane}). If $H_1\cap H_2 \sub Z$, then we would have $H_1\cap H_2 \sub S$. This is impossible since $S=Y\cap Z$ is a smooth K3-surface. So the intersection $D := S\cap H_1 \cap H_2 = Z \cap H_1 \cap H_2$ must be an effective divisor in $\Pic(S)$ of degree 3 and hence $D\cdot H_S =3$.\par
We first suppose that $g=\frac{d^2}{12}+1$ and $\Pic(S)= \Z H_S$. Then there is $a\in\Z$ such that $D=aH_S$. We would therefore have $3=D\cdot H_S = a H_S^2 = 6a$, which is impossible. So assume that we are in the second case, in which $g\leq \frac{d^2-1}{12}$ and $\Pic(S)= \Z H_S \oplus \Z C$. \par
We first show that $D$ must be a reduced curve. Suppose to the contrary that $D$ is not reduced.  Since $S$ can contain at most one line and cannot contain a line and an integral conic at the same time (\Cref{lem:line_conic_cubic_in_K3}), we must then have $D=3L$, where $L\sub S$ is a line in $S$. \par 
Note that $3L=D=H_1\cap H_2\cap S =(H_1\cap S) \cap (H_2\cap S)$. So in $\Pic(S)$, we must have $H_i\cap S=3L + R_i$, where $R_i\in\Pic(S)$ is an effective divisor for $i\in\{1,2\}$. Since $S$ is irreducible, it cannot be contained in the hyperplanes $H_i$ for $i \in \{1,2\}$. So $H_i\cap S$ are divisors of degree 6 and thus $R_i\sub S$ effective divisors of degree 3. We now show that we must have $R_i=3L$. \par 
Recall that all integral conics and cubics in $S$ must be rational since integral cubics in $Y$ must have genus 0 (\Cref{lem:curve_degree3_genus1}). By the assumption $H_1\cap H_2\cap S=3L$, the surface $S$ contains the line $L$. But then $S$ cannot contain an integral conic or cubic by \Cref{lem:line_conic_cubic_in_K3}. This lemma also shows that $L$ is the unique effective divisor in $\Pic(S)$ of degree 1. For both of these reasons, we must indeed have $R_1=3L=R_2$. \par 
But this would imply that $H_i\cap S=3L+R_i=6L$ in $\Pic(S)$ for $i\in\{1,2\}$. This, however, contradicts the assumption that $(H_1\cap S) \cap (H_2\cap S)=3L$.\par 
That is to say, the divisor $D$ must indeed be reduced. Moreover, $D$ must be irreducible since $S$ cannot contain a line and an integral conic, nor 2 or 3 distinct lines at the same time (\Cref{lem:line_conic_cubic_in_K3}). That is to say, $D$ is an integral cubic in the plane $H_1\cap H_2\simeq \Ptwo$. So the genus-degree-formula for integral curves in the projective plane implies that $p_a(D)=1$. \par  
However, we also have $D\sub S$ and we have shown in \Cref{lem:divisor_degree_3_genus_1} that all integral cubics in $S$ have genus at most 0. That is to say, the divisor $D$ cannot exist on the surface $S$ and thus, $Y$ must indeed be smooth.
\end{proof}
\begin{rmk} \label{rmk:counterexample}
For $d=14$ and $g=16$, \Cref{thm:existence_curves_in_K3_surface_degree6_smooth_quadric} proves, in particular, the existence of a smooth hyperquadric $Y\sub \Pfour$, a hypercubic $Z\sub \Pfour$ and a smooth irreducible curve $C$ of degree 14 and genus 16 such that $C$ is contained in the smooth intersection $Y \cap Z$. Since the arithmetic and geometric genus of the smooth curve $C$ coincide, this is a counter example for \cite[Thm.\ 1.4]{De_Cataldo-Genus_of_curves_on_3d_quadric}, which gives 15 as an upper bound for such a curve.
\end{rmk}
Now we can prove the main result of this subsection: The existence of a weak Fano threefold arising as the blowup of a smooth hyperquadric in $\Pfour$ along a smooth irreducible curve $C\sub \Pfour$ whose degree and genus appear in the list $\mathcal{P}$. This result will be a corollary of  \Cref{thm:existence_curves_in_K3_surface_degree6_smooth_quadric} and the following lemma.
\begin{lem} \label{lem:existence_gdpairs_yielding_weak_Fano}
Let $(g,d)\in \mathcal{P}$. Assume that $C$ is a smooth irreducible curve $C$ of degree $d$ and genus $g$ on a smooth sextic K3-surface $S\sub \Pfour$ such that: 
\begin{itemize}
\item $S=Y\cap Z\sub \Pfour$ for a smooth hyperquadric $Y\sub \Pfour$ and a hypercubic $Z\sub \Pfour$.
\item Either $\Pic(S)= \Z H_S$ or $\Pic(S)=\Z H_S \oplus \Z C$ where $H_S$ denotes a general hyperplane section of $S$.
\end{itemize}
Then the blowup of $Y$ along $C$ is weak Fano.
\end{lem}
\begin{proof}
We want to apply \Cref{thm:general_case_in_list_implies_weak_Fano}. We already know that $(g,d)\in \mathcal{P}$ and that $C$ lies on the smooth hyperquadric $Y\sub \Pfour$, as well as on the smooth hypercubic section $S$ of $Y$. So it is left to show that $C$ does not admit a $3n+1$-secant of degree $n\leq n_{\max}$, where $n_{\max}$ is given in \Cref{table:max_degree_of_3n+1_secant}. \par
If $S$ contains a line, then $g=\frac{d^2-1}{12}$ and $d\equiv_6 r$ with $r\in\{1,5\}$ (\Cref{lem:line_conic_cubic_in_K3}). The only pairs $(g,d)\in \mathcal{P}$ satisfying these conditions are $(0,1),(2,5)$ and $(14,13)$. However, for all three of these pairs, we have $n_{\max}=0$. \par 
Similarly, if $S$ contains an integral conic, then $g=\frac{d^2-4}{12}$ and $d\equiv_6 r$ with $r\in\{2,4\}$ (\Cref{lem:line_conic_cubic_in_K3}). The only pairs $(g,d)\in \mathcal{P}$ satisfying these conditions are $(0,2),(1,4), (5,8)$ and $(8,10)$. Again, for all four of these pairs, we have $n_{\max}=0$. \par
That is to say, for all pairs $(g,d)\in \mathcal{P}$, the curve $C$ cannot have a 4-secant line or 7-secant conic. So by  \Cref{thm:general_case_in_list_implies_weak_Fano}, the blowup of the smooth hyperquadric $Y\sub \Pfour$ along $C$ is a weak Fano threefold.
\end{proof}
\begin{cor} \label{cor:existence_gdpairs_yielding_weak_Fano}
Let $(g,d)\in\mathcal{P}$. Then there exists a smooth irreducible curve $C\sub \Pfour$ of degree $d$ and genus $g$, as well as a smooth hyperquadric $Y\sub \Pfour$ containing $C$ such that the blowup of $Y$ along $C$ is a weak Fano threefold. 
\end{cor}
\begin{proof}
Observe that every pair $(g,d)\in \mathcal{P}$ either satisfies $g=\frac{d^2}{12}+1$, or $g\leq \frac{d^2-1}{12}$ and $(g,d)\neq(4,7)$. So \Cref{thm:existence_curves_in_K3_surface_degree6_smooth_quadric} gives the existence of a smooth hyperquadric $Y\sub \Pfour$ and a hypercubic $Z\sub \Pfour$ such that the following is satisfied: $S=Y\cap Z$ is a smooth K3-surface that contains a smooth irreducible curve $C$ of degree $d$ and genus $g$ and that satisfies either $\Pic(S)=\Z H_S$, or $\Pic(S)=\Z H_S \oplus \Z C$, where $H_S$ denotes a general hyperplane section of $S$. \par 
Then the blowup of $Y$ along $C$ is weak Fano by \Cref{lem:existence_gdpairs_yielding_weak_Fano}. \par 
\end{proof}
%
%
\section{Curves on a hyperplane or smooth hyperquadric section} \label{section:hyperplane}
In \S \ref{subsection:hyperplane_preliminaries} and \S \ref{subsection:hyperplane_weak_Fano}, we will weaken the assumptions of \Cref{thm:B} for those pairs in $\mathcal{P}$ that come from curves contained in a hyperplane. We will also modify the assumptions of \Cref{thm:B} in \S \ref{subsection:hyperquadric_preliminaries} and \S \ref{subsection:hyperquadric_weak_Fano}, given that the curve $C$ is contained in a smooth hyperquadric section of the smooth hyperquadric $Y\sub  \Pfour$. \par
For this purpose, we introduce the following notation for the rest of this section: $\mathcal{P}=\mathcal{P}_{\textup{plane}} \cup \mathcal{P}_{\textup{quadric}} \cup \mathcal{P}_{\textup{cubic}} $, where
$$\mathcal{P}_{\textup{plane}}:=\{(0,1),\:(0,2),\:(0,3),\:(0,4),\:(1,4),\:(2,5),\:(4,6)\},$$
\begin{gather*}
\mathcal{P}_{\textup{quadric}} :=\{ (0,4),\:(0,5),\:(0,6),\:(1,5),\:(1,6),\:(2,6),\:(2,7),\:(3,7),\:(3,8),\:(4,8),\:(5,8)\}\:\cup\\
\{(6,9),\:(8,10),\:(13,12)\},
\end{gather*}
\begin{gather*}
\mathcal{P}_{\textup{cubic}}:=\{(0,7),\:(0,8),\:(1,7),\:(1,8),\:(2,8),\:(2,9),\:(3,8),\:(3,9),\:(4,9),\:(5,9),\:(5,10)\} \:\cup \\\{(6,10),\:(7,10),\:(8,11),\:(9,11),\:(11,12), \:(14,13) \}.
\end{gather*}
More precisely, instead of a curve $C$ lying on a smooth K3-surface of degree 6, we will be looking at curves lying on a singular or smooth quadric surface (\S4.1) or on a smooth  quartic surface (\S4.3). Using the well-known structure of these type of surfaces, we are able to prove that if $(g,d)\in \mathcal{P}_{\textup{plane}}$, then $X$ is weak Fano (\S4.2); respectively, if $(g,d)\in \mathcal{P}_{\textup{quadric}}$ and $C$ is contained in a smooth hyperquadric section of $Y$, but admits no 4-secant line in the cases $(g,d)\in \{(0,6),\:(3,8)\}$, then $X$ is weak Fano (\S4.4). \par
Note that $(0,4)\in \mathcal{P}_{\textup{plane}} \cap \mathcal{P}_{\textup{quadric}}$ since the blowup along a curve of type $(0,4)$ gives rise to two different weak Fano threefolds, depending on whether the curve lies on a hyperplane or not. Similarly, the blowup along a curve of type $(3,8) \in \mathcal{P}_{\textup{quadric}} \cap \mathcal{P}_{\textup{cubic}}$  yields two different weak Fano threefolds, depending on whether the curve lies on a hyperquadric section or not. Both of these will become apparent when studying the potential Sarkisov links arising from the weak Fano threefold $X$ in \Cref{section:Sarkisov_links}.\par
\subsection{Existence of curves on a hyperplane section} \label{subsection:hyperplane_preliminaries}
In this subsection, we recall some auxiliary results about quadric hyperplane sections in $\Pfour$ and the curves lying on them. In particular, we also consider singular quadric hyperplane sections so that we do not have to assume that the curve $C$ is contained in a smooth hyperplane section of $Y$. 
\begin{lem} \label{lem:contained_in_hyperplane_section}
Let $C \sub Y$ be a smooth irreducible curve of genus $g$ and degree $d$ that lies on a smooth hyperquadric $Y \sub \Pfour$. If $ (g,d) \in \mathcal{P}_{\textup{plane}} \setminus \{(0,4)\}$, then $C$ is contained in a hyperplane section of $Y$. For $d\geq 3$, this hyperplane section is unique.
\end{lem}
\begin{proof}
If $(g,d)\in\{(0,1),\:(0,2),\:(0,3),\:(1,4),\:(2,5)\}$, then $2g-2<d$. By \Cref{lem:dimension_linear_system_hypersurfaces}, the linear system of hyperplanes of $\Pfour$ containing $C$ then has a projective dimension of at least $3+g-d \geq 0$. In other words, $C$ is contained in at least one hyperplane of $\Pfour$. \par 
If $(g,d)=(4,6)$, then $g > \left \lfloor \frac{d(d-5)}{6} \right \rfloor + 1 = 2$ and thus, $C$ must be contained in a hyperplane (\cite[Corollary 2.7(2)]{Blanc_Lamy_Cubic}). \par
The uniqueness of the hyperplane section for the case $d \geq 3$ results from the following fact: Since $Y$ does not contain any plane (\Cref{lem:hyperplane_section_irred}), a curve contained in more than one hyperplane section must be of degree $d\leq 2$ by Bézout's theorem.
\end{proof}
The next two lemmas are a recollection of results on quadric hyperplane sections in $\Pfour$. The lemmas will also introduce the notation that will be used throughout the next two subsections.
\begin{lem} \label{lem:quadric_surface_in_P4}
Let $Y\sub \Pfour$ be a smooth hyperquadric and $S=Y\cap H$ a hyperplane section of $Y$ for some hyperplane $H\sub \Pfour$. Then the following dichotomy applies:
\begin{enumlem}
\item If $S$ is smooth, then $S$ is isomorphic to $\PxP$ and the Picard group of $S$ is equal to $\Pic(S) = \Z f_1 \oplus \Z f_2$, where $f_i$ is a fiber in $\Pone$ of the $i$-th projection for $i\in \{1,2\}$. The intersection form on $\Pic(S)$ is determined by the rules $f_1^2 = 0 = f_2^2$ and $f_1\cdot f_2 =1$. Moreover, a general hyperplane section of $S$ is given by $f_1+f_2$. \label{lem:smooth_quadric_surface_in_P4}
\item If $S$ is singular, then, up to a change of coordinates, $S=\VPfour(x_0x_2-x_1^2,\:x_4)$. So $S$ is isomorphic to a quadric cone in $\Pthree$ and has a unique singular point $p=[0:0:0:1:0]\in S$. \par
Consider the blowup $\eta: \Ftwo \rightarrow S, \: [s_0:s_0;\:t_0:t_1] \mapsto [s_0^2t_0:s_0s_1t_0:s_1^2t_0:t_1:0]$ of $S$ in $p$. Here, the Hirzebruch surface $\Ftwo$ is seen as the set of equivalence classes $\Ftwo = \{[s_0:s_1;\:t_0:t_1] \: | \: [s_0:s_1], [t_0:t_1]\in\Pone\}$ such that $[s_0:s_1;\:t_0:t_1] = [\lambda s_0:\lambda s_1;\:\lambda^{-2} \mu t_0: \mu t_1]$  for all $\lambda, \mu \in \kk \setminus \{0\}$. Moreover, we denote the exceptional divisor by $e=\eta^{-1}(\{p\})$ and a general fiber of the projection $pr_1:\Ftwo \rightarrow \Pone$ on the first coordinate by $f$. Then $\Pic(\Ftwo)=\Z f \oplus \Z e$ is determined by the intersection rules $e \cdot f = 1, \: f^2 = 0$ and $e^2 = -2$. Moreover, the pullback of a general hyperplane section of $S$ under $\eta$ is given by $e+2f$.
\label{lem:singular_quadric_surface_in_P4}
\end{enumlem}
\end{lem}
\begin{lem} \label{lem:curves_in_quadric_surface_in_P4}
Using the same notation as in \Cref{lem:quadric_surface_in_P4}, we can characterize smooth irreducible curves on a quadric hyperplane section $S$ in $\Pfour$ as follows:
\begin{enumlem}
\item If $S\simeq \PxP$ is smooth, curves $\Gamma$ in $S$ are precisely the zero sets of bihomogeneous polynomials $h\in \kPxP_{a,b}$ of bidegree $(a,b)\in \Z^2_{\geq 0} \setminus \{(0,0)\}$. In $\Pic(S)$, we then have $\Gamma=af_1+bf_2$. Moreover, if $\Gamma$ is a smooth irreducible curve, then its degree and genus are given by $\deg(\Gamma) = a+b$ and $p_a(\Gamma) = (a-1)(b-1)$.
\label{lem:curves_in_smooth_quadric_surface_in_P4}
\item Suppose that $S$ is singular and $\Gamma \sub S$ is a smooth irreducible curve with strict transform $\hat{\Gamma}\sub \Ftwo$. Then there is $(a,b)\in \left(\Z_{\geq 0} \times \{0,1\}\right) \setminus \{(0,0)\}$ such that $\hat{\Gamma} = a e + (2a+b)f$ in $\Pic(\Ftwo)$. Moreover, we have $b=1$ if and only if $p\in C$. The degree and genus of $\Gamma$ are given by $\deg(\Gamma)=2a+b$ and $p_a(\Gamma) = (a-1)(a-1+b)$. 
\label{lem:curves_in_singular_quadric_surface_in_P4}
\end{enumlem}	
\end{lem}
\begin{proof}[Proof of \Cref{lem:quadric_surface_in_P4} and \ref{lem:curves_in_quadric_surface_in_P4}]
The hyperplane $H\sub \Pfour$ is isomorphic to $\Pthree$ via a closed embedding $\psi:\Pthree \rightarrow \Pfour$ of degree 1. Via $\psi$, any curve in $H$ is isomorphic to a smooth irreducible curve in $\Pthree$ of the same genus and degree. Furthermore, $S$ is isomorphic to the quadric surface $\psi^{-1}(S)\sub \Pthree$.  By \Cref{lem:hyperplane_section_irred}, $S$ and hence also $\psi^{-1}(S)$ are irreducible. Now the results of the lemma follow from the classical theory of irreducible quadric surfaces in $\Pthree$ (see e.g.\ \cite[Chapter 2]{Miles_Reid_Chapters_on_algebraic_surfaces}).
\end{proof}
If our curve $C$ is contained in a hyperplane section and the corresponding threefold $X$ is weak Fano, then $C$ must also be contained in a smooth and thus irreducible hypercubic section of $Y$ (\Cref{prop:weak_Fano_implies_contained_in_smooth_hypercubic_section}). So by Bézout's theorem, the degree of $C$ is at most $6$. Applying the formulas given in \Cref{lem:quadric_surface_in_P4}, we obtain directly the following potential candidates of $(g,d)$-pairs of a curve $C$ lying on a hyperplane such that the blowup $X$ of $Y$ along $C$ is weak Fano:
\begin{cor} \label{lem:existence_of_curves_in_quadric_hypersurface}
Let $Y\sub \Pfour$ be a smooth hyperquadric and $C\sub Y$ a smooth irreducible curve of degree $d \leq 6$ and genus $g$, which is contained in a hyperplane section $S = Y \cap H$ of $Y$ for some hyperplane $H \sub \Pfour$. Then there are two possibilities:
\begin{enumcor}
\item If $S$ is smooth, then $ (g,d) \in \mathcal{P}_{\textup{plane}}\cup \{(0,5), \:(0,6),\:(3,6)\}.$ Using the notation of \Cref{lem:curves_in_smooth_quadric_surface_in_P4}, the curve $C$ is given by $C=af_1+bf_2$ in $\Pic(S)$, where $a,b\in\Z_{\geq 0}$ with $a\geq b$ are given in \Cref{table:a_b_pairs_curves_in_P1xP1}. \label{lem:existence_of_curves_in_smooth_quadric_hypersurface}
\item If $S$ is singular, then $(g,d) \in\mathcal{P}_{\textup{plane}}\setminus\{(0,4)\}.$
Using the notation of \Cref{lem:curves_in_singular_quadric_surface_in_P4}, the strict transform $\hat{C} \sub \Ftwo$ of $C$ is given by $\hat{C}=ae+(2a+b)f$ in $\Pic(\Ftwo)$, where the pair $(a,b)\in \left(\Z_{\geq 0} \times \{0,1\}\right) \setminus \{(0,0)\}$ is given in \Cref{table:a_b_pairs_curves_in_quadric_cone}. \label{lem:existence_of_curves_in_singular_quadric_hypersurface}

\end{enumcor}
\begin{table}[H] 
\caption{Possible $(g,d)$-pairs if the corresponding smooth irreducible curve $C\sub \Pfour$ of genus $g$ and degree $d\leq 6$ is contained in a smooth hyperplane section $S$ of a smooth hyperquadric $Y\sub \Pfour $. In $\Pic(S)$, we then have $C=af_1+bf_2$ (supposing $a\geq b$ to avoid doubles).}
\begin{minipage}{\textwidth}
\centering
\vspace{6pt}
\begin{tabular}{|l|l|l|l|l|l|l|l|l|l|l|} 
	\hline
	$(a,b)$ & $(1,0)$ & $(1,1)$ & $(2,1)$ & $(2,2)$ & $(3,1)$ & $(3,2)$ & $(3,3)$ & $(4,1)$ & $(4,2)$ & $(5,1)$ \\
	\hline
	$(g,d)$ & $(0,1)$ & $(0,2)$ & $(0,3)$ & $(1,4)$ & $(0,4)$ & $(2,5)$ & $(4,6)$ & $(0,5)$ & $(3,6)$ & $(0,6)$ \\
	\hline
\end{tabular}
\vspace{6pt}
\end{minipage}\label{table:a_b_pairs_curves_in_P1xP1}
\end{table}
\vspace{-0.5cm}
\begin{table}[H] 
\caption{Possible $(g,d)$-pairs if the corresponding smooth irreducible curve of genus $g$ and degree $d\leq 6$ is contained in a singular hyperplane section $S$ of a smooth hyperquadric $Y\sub \Pfour $, whose blowup in its singular point is given by $\eta:\Ftwo\rightarrow S$. In $\Pic(\Ftwo)$, the strict transform of $C$ under $\eta$ is then given by $\hat{C}=ae+(2a+b)f$.}
\begin{minipage}{\textwidth}
\centering
\vspace{6pt}
\begin{tabular}{|l|l|l|l|l|l|l|l|}
	\hline
	$(a,b)$ & $(0,1)$ & $(1,0)$ & $(1,1)$ & $(2,0)$ & $(2,1)$ & $(3,0)$ \\
	\hline
	$(g,d)$ & $(0,1)$ & $(0,2)$ & $(0,3)$ & $(1,4)$ & $(2,5)$ & $(4,6)$ \\
	\hline
\end{tabular}
\vspace{6pt}
\end{minipage}
\label{table:a_b_pairs_curves_in_quadric_cone}
\end{table}
\end{cor}
In the next subsection, we will show that the pairs $(g,d)\in \{(0,5), \:(0,6),\:(3,6)\}$ do not yield a weak Fano threefold. Therefore, we will be left with only the pairs $(g,d)\in \mathcal{P}_{\textup{plane}}$. \par 
Also note that the curve corresponding to the pair $(g,d)=(0,4)$ can be contained only on a smooth hyperplane section $S$ of $Y$, whereas for all the other pairs, $S$ can also be singular. For this reason, we also have to study the case in which $S$ is isomorphic to a singular quadric cone $Q$ in $\Pthree$. To study the blowup of $Q$ along a curve, we first recall the universal property of blowups.
\begin{rmk} \label{rmk:universal_property_of_blowups}
Recall that if $W$ is a smooth projective surface and $Z\sub W$ a smooth irreducible curve, then the blowup of $W$ along $Z$ is an isomorphism. The reason for this fact is the following universal property of the blowup $f:\widetilde{W} \rightarrow W$ of a projective variety $W$ along a subvariety $Z\sub W$: \par 
If $g:V\rightarrow W$ is a morphism such that $g^{-1}(Z)$ is an effective Cartier divisor, then there is a unique morphism $h:V \rightarrow \widetilde{W}$ such that $g=f\circ h$ (see e.g.\ \cite[Prop.\ II.7.14]{Hartshorne}). 
\end{rmk}
The next two lemmas aim to show that the blowup of a singular quadric cone $Q\sub \Pthree$ along a smooth irreducible curve through the singular point $p$ of $Q$ is isomorphic to the blowup of $Q$ in $p$.
\begin{lem} \label{lem:automorphisms_quadric_cone}
Let $L$ be a line on the two-dimensional quadric cone $Q=\VPthree(x_0x_2-x_1^2)\sub \Pthree$. Then there is an automorphism $\alpha \in \Aut(Q)$ such that $\alpha(L)=\VPthree(x_1,x_2)\sub Q$. 
\end{lem}
\begin{proof}
We first construct a subgroup of $\Aut(Q)$ from which we will choose the automorphism $\alpha$. Recall that the singular quadric cone $Q\sub\Pthree$ is uniquely given as the union of lines in $\Pthree$ through the singular point $p:=[0:0:0:1]$ and one point on the planar conic $\Gamma:=\VPtwo(x_0x_2-x_1^2)$ embedded into $Q$. Furthermore, the conic $\Gamma$ is isomorphic to $\Pone$ via the isomorphism 
$$ \Pone \rightarrow \Gamma, \: [u:v]\mapsto [u^2:uv:v^2].$$
This isomorphism induces an inclusion
$$\begin{array}{c c c c c c}
\iota: &\PGL_2(\kk) = \Aut(\Pone) &\overset{\simeq}{\longrightarrow} &\Aut(\Gamma)\sub\GL_3(\kk) &\hookrightarrow &\Aut(Q)\sub \GL_4(\kk)\\
&A=\begin{pmatrix}
a & b\\
c & d
\end{pmatrix}
& \mapsto 
&M_A:=\frac{1}{ad-bc} \begin{pmatrix}
a^2 & 2ab & b^2\\
ac & ad+bc & bd\\
c^2 & 2cd & d^2
\end{pmatrix}
&\mapsto 
&\begin{pmatrix}
& & &0\\
& M_A& &0\\
& & &0\\
0&0&0&1
\end{pmatrix}
\end{array}. $$
This inclusion induces an action of $\PGL_2(\kk)$ on $Q$, which is given by $\PGL_2(\kk)\times Q \rightarrow Q, \: (A,x)\mapsto \iota(A)(x)$. This action fixes the singular point 
$p:=[0:0:0:1]$ of $Q$
and acts transitively on the points of the conic $\Gamma'=\VPthree(x_0x_2-x_1^2,x_3)\simeq \Gamma$. It then acts transitively on the lines in $Q$, which are lines through the singular point $p$ and one point of $\Gamma'$. In particular, we can find an isomorphism $\alpha \in \iota(\PGL_2(\kk))\sub \Aut(Q)$ such that $\alpha(L)=\VPthree(x_1,x_2)$.
\end{proof}
\begin{lem} \label{lem:blowup_quadric_cone_along_curve}
Consider the two-dimensional quadric cone $Q=\VPthree(x_0x_2-x_1^2)\sub \Pthree$ and its singular point $p:=[0:0:0:1]\in Q$. Let $C\sub Q$ be a smooth irreducible curve through $p$. Denote by $\pi: \Bl_C(Q)\rightarrow Q$ the blowup of $Q$ along $C$ and by $\eta:Bl_p(Q)\rightarrow Q$ the blowup of $Q$ in $p$ with exceptional divisors $E:=\pi^{-1}(C)$ and $e:=\eta^{-1}(\{p\})$. \par 
Then there is an isomorphism $\varphi:\Bl_p(Q)\rightarrow \Bl_C(Q)$ satisfying $\eta = \pi \circ \varphi$. In particular, we have $\varphi^{-1}(E) = \hat{C} \cup e$, where $\hat{C} \sub \Bl_p(Q)$ is the strict transform of $C$ under $\eta$.
\end{lem}
\begin{proof} The blowup of $Q$ in $p$ can naturally be embedded into $\Pthree \times \Ptwo$ by 
$$ \Bl_p(Q)=\{([x_0:\dots:x_3], [y_0:y_1:y_2])\in \Pthree \times \Ptwo \:|\: y_0y_2=y_1^2, \: x_iy_j=x_jy_i \textup{ for } i,j\in\{0,1,2\}\}.$$
The blowup map $\eta:\Bl_p(Q)\rightarrow Q,\: (x,y)\mapsto x$ is then just given by the projection on the second factor. Furthermore, the exceptional divisor is $e:=\eta^{-1}(\{p\})=\{p\} \times \Gamma$, where $\Gamma:=\VPtwo(x_0x_2-x_1^2)$ is a conic in $\Ptwo$.\par 
Consider a smooth irreducible curve $C\sub Q$ through $p$. By classical results about blowups, we can construct the blowup of $Q$ along $C$ as follows: Let $g_1,\dots, g_s \in \kPthree$ be homogeneous polynomials such that $I_{\Pthree}(C)=(g_1,\dots,g_s)$. Denoting $d:=\max(\deg(g_1),\dots,\deg(g_s))$, we choose a basis $\{f_0,\dots, f_n\}\sub\kPthree$ of the vector space $$\{f\in I_{\Pthree}(C) \:|\: f \textup{ homogeneous of degree } d\}.$$ \par 
We write $f_i=\sum_{j=1}^d f_{i,j}(x_0,x_1,x_2) \cdot x_3^{d-j}$ for all $i\in\{0,\dots,n\}$ with $f_{i,j}\in \kk[x_0,x_1,x_2]_j$ homogeneous polynomials of degree $j\in \{1,\dots,d\}$. We define a birational map $\tau: Q \dashrightarrow \Pthree \times \Pn$, which is given by
$x \mapsto (x,[f_0(x):{\dots }:f_n(x)]).$
Then $\pi:Bl_C(Q)\rightarrow Q, \:(x,y)\mapsto x$ is the blowup of $Q$ along $C$, where $\Bl_C(Q)= \overline{\tau(Q\setminus C)}\sub \Pthree\times \Pn$.\par 
Observe that the singular fiber $\pi^{-1}(\{p\})$ is a curve since the tangent space $T_p(Q)$ of $Q$ in $p$ is of dimension 3 and the exceptional divisor $E:=\pi^{-1}(C)$ is the projectivized normal bundle of $C$ in $Q$.
So $E$ is not irreducible as it is the union of the singular fiber $\pi^{-1}(\{p\})$ and a section $s$ of the restriction $\pi|_{E}:E \rightarrow C$. \par 
Defining $\varphi:= \tau \circ \eta$, we obtain a birational map
$$ \begin{array}{l l l l l}
\varphi:&\Bl_p(Q)&\dashrightarrow &&\Bl_C(Q), \\ 
&(x,y) &\mapsto &&(x,[f_0(x):{\dots}: f_n(x)])\\
& & =& & \left[\sum_{j=1}^d f_{0,j}(x_0,x_1,x_2)\cdot x_3^{d-j}: {\dots}: \sum_{j=1}^d f_{n,j}(x_0,x_1,x_2)\cdot x_3^{d-j} \right].
\end{array}$$ 
Since $C\setminus \{p\}$ is an effective Cartier divisor on the smooth locus $S\setminus \{p\}$ of the surface $S$, the blowup $\pi$ is an isomorphism outside of $p$ (\Cref{rmk:universal_property_of_blowups}). Therefore, the birational map $\varphi$ is an isomorphism on $\Bl_p(Q)\setminus e \simeq \Bl_C(Q)\setminus \pi^{-1}(\{p\})$, which satisfies $\varphi(\hat{C})=s$, making the following diagram commute: \par
\noindent
\begin{minipage}[h]{\textwidth}
\vspace{0,5cm}
\centering
\begin{tikzcd}[column sep=tiny]
\textcolor{gray}{e\sub} \ar[d,mapsto, gray]  &\textcolor{gray}{\eta^{-1}(C) = \hat{C} \cup e \sub} \ar[d, mapsto, gray] & \Bl_p(Q) \ar[rr, "\varphi", dashed] \ar[dr, "\eta"]& &\Bl_C(Q) \ar[dl, "\pi" swap]& \textcolor{gray}{\supseteq E = \pi^{-1}(\{p\}) \cup s}\ar[d,mapsto, gray]& \textcolor{gray}{ \supseteq \pi^{-1}(\{p\})} \ar[d,mapsto, gray] \\
\phantom{....}\textcolor{gray}{p\in}& 	\phantom{....} \textcolor{gray}{C\sub} & & Q \ar[ru, dashrightarrow, bend right, "\tau" swap]& & \textcolor{gray}{\supseteq C} & \textcolor{gray}{\ni p}
\end{tikzcd}
\vspace{0,5cm}
\end{minipage}
It is left to show that $\varphi$ is defined on $e$ and restricts to an isomorphism $\varphi|_{e}: e \rightarrow \pi^{-1}(\{p\})$. \par 
Since $C$ is smooth in $p$, the tangent space $T_p(C)\sub \Pthree$ of $C$ in $p$ must be a line. Moreover, we can show that $T_p(C)$ is contained in $Q$. To do this, note that the tangent direction of $C$ in $p$ corresponds to a point $(p,y)\in e = \{p\} \times \Gamma$ for some $y\in\Gamma :=\VPtwo(x_0x_2-x_1^2)\sub \Ptwo$. Embedding $\Gamma$ in the cone $Q$, the line $T_p(C)$ contains the point $p$ and $y\in \Gamma \sub Q$. So $T_p(C)$ is the unique line through $p$ and $q$, which is contained in $Q$ (as $Q$ is the union of lines through its singular point $p$ and one point on $\Gamma$ embedded into $Q$). \par 
We have thus shown that $T_p(C)$ is a line on $Q$. By \Cref{lem:automorphisms_quadric_cone}, there exists an automorphism $\alpha\in\Aut(Q)$ such that $\alpha(T_p(C))=\VPthree(x_1,x_2)$. Moreover, recall that the tangent space of $C$ is given by $T_p(C)=\VPtwo(f_{0,1},\dots,f_{n,1})$. We may thus assume, up to composition with a linear automorphism, that $f_{0,1}(x_0,x_1,x_2)=x_1, f_{1,1}(x_0,x_1,x_2)=x_2$ and $f_{i,1}(x_0,x_1,x_2)=0$ for all $(x_0,x_1,x_2)\in \kk^3$ and all $i\in\{2,\dots,n\}$.\par 
We now examine the birational map $\varphi:\Bl_p(Q)\dashrightarrow \Bl_C(Q)$ on $e=\{p\}\times \Gamma$. Consider the open affine subsets $U_i:=\{([x_0:{\dots}:x_3], [y_0:y_1:y_2])\in \Bl_p(Q) \: |\: x_3\neq 0\neq y_i\}$ of $e$ for $i\in\{0,1,2\}$. Every point $u\in U_0$ can be written as $u=(x, y)$ with $x=[t:tu:tv:1]$ and $y=[1:u:v]$ for some $u,v,t\in\kk$. We calculate
\begin{align*}
\varphi(u) &= \left(x,\left[\sum_{j=1}^d f_{0,j}(t,tu,tv): {\dots}: \sum_{j=1}^d f_{n,j}(t,tu,tv) \right]\right) \\
&= \left(x,\left[\sum_{j=1}^d f_{0,j}(1,u,v)\cdot t^j: {\dots}: \sum_{j=1}^d f_{n,j}(1,u,v) \cdot t^j \right]\right) \\
&\overset{:t}{=} \left(x,\left[\sum_{l=0}^{d-1} f_{0,l+1}(y)\cdot t^l: {\dots}: \sum_{l=0}^{d-1} f_{n,l+1}(y) \cdot t^l \right]\right).
\end{align*}
We can carry out analogous computation for the subsets $U_1$ and $U_2$. Note that $U_0\cup U_1\cup U_2$ is an open neighborhood of $e$. Since on each $U_i$ with $i\in\{0,1,2\}$, the exceptional divisor $e=\{p\}\times \Gamma$ corresponds to the substitution $t=0$, for every point $(p,y)\in e$ with $y=[y_0:y_1:y_2]\in\Gamma$, we obtain
\begin{align*}
\varphi((p,y)) &= (p,[f_{0,1}(y_0,y_1,y_2):{\dots} :f_{n,1}(y_0,y_1,y_2)]) \\
&= (p,[y_1:y_2:0:{\dots}:0 ]).
\end{align*}
For all points $y=[y_0:y_1:y_2]$ on the open dense set $\Gamma \setminus (\VPtwo(y_1)\cup \VPtwo(y_2))$, multiplication by $\frac{y_1}{y_2}$ gives $$[y_1:y_2:0:{\dots}:0] = [y_0:y_1:0:{\dots}:0].$$
In other words, $\varphi$ can be defined on $e$ for every $[y_0:y_1:y_2]\in \Gamma$ by 
$$ \varphi|_{e}:e\rightarrow E, \quad (p,[y_0:y_1:y_2])\mapsto \begin{cases}
(p,[y_1:y_2:0:{\dots}:0]) \textup{ if } (y_1,y_2)\neq (0,0),\\
(p,[y_0:y_1:0:{\dots}:0])\textup{ if } (y_0,y_1)\neq (0,0).
\end{cases} $$
So $\varphi|_{e}$ is an embedding. That is to say, $\varphi$ is isomorphic onto its image with inverse
$$ \varphi(e)\rightarrow e, \: (p,[y_0:{\dots}:y_n]) \mapsto (p,[y_0^2:y_0y_1:y_1^2]). $$
Since $\pi^{-1}(\{p\}) \supseteq \varphi(e)\simeq \Gamma:=\VPthree(y_0y_2-y_1^2) \simeq \Pone$ and $\pi^{-1}(\{p\}) \simeq \Pone$, the image of $\varphi|_{e}$ must be the entire singular fiber $\pi^{-1}(\{p\})$. That is to say, $\varphi$ indeed restricts to an isomorphism $\varphi|_{e}: e \rightarrow \pi^{-1}(\{p\})$ which shows that $\varphi$ is globally an isomorphism.
\end{proof}
%
%
\subsection{Curves on a hyperplane section yielding a weak Fano threefold} \label{subsection:hyperplane_weak_Fano}
In the following, we want to prove that if $C$ is contained in a hyperplane, then $X$ is weak Fano if and only if $(g,d)\in \mathcal{P}_{\textup{plane}}$. \par 
In the rest of the section, we assume thus that the curve $C$ lies on a hyperplane section $S$ of $Y$. We denote the strict transform of $S$ under the blowup $\pi:X\rightarrow Y$ of $Y$ along $C$ by $\widetilde{S}:=\overline{\pi^{-1}(S\setminus C)}\sub X$. \par 
We have seen in \Cref{cor:weak_Fano_big_nef_basepoinfree} that in order to show that $X$ is weak Fano, it suffices to prove that the restricted linear system $|-K_X|_{\widetilde{S}}|$ is base-point-free.
We will now show that that this condition can be reduced to verifying whether a linear system on $\PxP$ or $\Ftwo$ is base-point-free, depending on whether $S$ is singular and if so, whether $C$ contains the singular point of $S$ or not.
We will use this to deduce that $\pi|_{\widetilde{S}}$ is an isomorphism if the hyperplane section $S$ is smooth. If $S$ is, however, singular, this only holds if the singular point $p$ of $S$ is not contained in $C$. If $S$ is singular and $p\in C$, the next lemma specifies properties of the morphism $\pi|_{\widetilde{S}} $. It also states how these properties can be used to determine a sufficient condition for showing that $X$ is weak Fano.

\begin{lem} \label{lem:singular_hyperplane_section_p_not_in_C}
Let $C \sub Y$ be a smooth irreducible curve of genus $g$ and degree $d$ that lies on a smooth hyperquadric $Y \sub \Pfour$ and denote the blowup of $Y$ along $C$ by $\pi:X\rightarrow Y$. Suppose that $C$ is contained in a singular hyperplane section $S$ of $Y$. Recall that the blowup $\eta:\Ftwo \rightarrow S$ of $S$ in its singular point $p\in S$ is the Hirzebruch surface $\Ftwo$. Denote the strict transform of $C$ in $\Pic(\Ftwo)=\Z e \oplus \Z f$ by $\hat{C}\sub \Ftwo$, where $e=\eta^{-1}(\{p\})$ and $f$ denotes a general fiber of the projection $\Ftwo \rightarrow \Pone$.\par 
The following statements are satisfied:
\begin{enumlem}
\item If $p\not \in C$, then the map $\pi|_{\widetilde{S}} : \widetilde{S} \rightarrow S$ is an isomorphism, where $\widetilde{S}\sub X$ denotes the strict transform of $S$ under $\pi$. Otherwise, if $p\in C$, then the blowup $\pi|_{\widetilde{S}} : \widetilde{S} \rightarrow S$ of $S$ along $C$ is isomorphic to the blowup $\eta:\Ftwo\rightarrow S$ of $S$ in $p$. \label{lem:singular_hyperplane_section_p_pi}
\item An irreducible curve $\Gamma \sub S$ is an $m$-secant to $C$ for $m\in \Z_{>0}$ if and only if the strict transform $\hat{\Gamma} \sub \Ftwo$ under $\eta$ satisfies $\hat{\Gamma}\cdot \hat{C} + m_p(C) \cdot m_p(\Gamma)\geq m$, where $m_p(\cdot) \in \Z_{\geq 0}$ denotes the multiplicity of a curve in the point $p$.  \label{lem:singular_hyperplane_section_m_secant}
\item Suppose that either $p\not \in C$ and the linear system $|3e+6f-\hat{C}|$ on $\Ftwo$ is base-point-free, or that $p\in C$ and the linear system $|2e+6f-\hat{C}|$ on $\Ftwo$ is base-point-free. Then $-K_X$ is nef.  \label{lem:singular_hyperplane_section_weak_Fano}
\end{enumlem}
\end{lem} 
\begin{proof}
Since $\Ftwo$ is smooth, the effective Weil divisor $\eta^{-1}(C)$ is an effective Cartier divisor on $\Ftwo$. Applying the universal property of blowups (\Cref{rmk:universal_property_of_blowups}) to the blowup $\pi|_{\widetilde{S}} : \widetilde{S} \rightarrow S$ of $S$ along $C$, we deduce that there exists a unique morphism $\varphi: \Ftwo \rightarrow \widetilde{S}$ such that the following diagram commutes: 
\begin{equation*} \label{diagram:eta_on_X_minus_Ei}
\begin{tikzcd}
\Ftwo \ar[d, "\eta"']  \ar[r, "\varphi"]& \widetilde{S} \ar[dl, "\pi|_{\widetilde{S}}"]\\
S &
\end{tikzcd}
\end{equation*}
We first suppose that $p \not \in C$. Then $S$ is smooth in an open neighborhood of $C$ and so $C$ is an effective Cartier divisor of $S$. Using the universal property of blowups once more, it follows that the map $\pi|_{\widetilde{S}} : \widetilde{S} \rightarrow S$ is an isomorphism. So $\varphi$ is the blowup of $\widetilde{S}$ in its unique singular point $\pi^{-1}(\{p\})\in\widetilde{S}$. Furthermore, $\varphi$ restricts to an isomorphism $\Ftwo \setminus e \simeq \widetilde{S}\setminus \pi^{-1}(\{p\})$ that sends $\hat{C}$ to $E|_{\widetilde{S}}$ . \par 
Let $\Gamma \sub S$ now be an irreducible curve with strict transform $\widetilde{\Gamma}\sub \widetilde{S}$ and multiplicity $m_p(\Gamma)=\hat{\Gamma} \cdot e\in \Z_{\geq 0}$ in the point $p$. Then the commutativity of the above diagram implies that $\varphi^*(\widetilde{\Gamma})=\hat{\Gamma}+m_p(\Gamma)\: e$. Noting that $m_p(C)=e\cdot \hat{C}=0$ (since $p\not \in C$), it follows that $\Gamma$ is an $m$-secant of $C$ for some $m\in\Z_{>0}$ if and only if 
$$m\leq  \widetilde{\Gamma}\cdot E =\widetilde{\Gamma} \cdot E|_{\widetilde{S}} = \varphi^*(\widetilde{\Gamma}) \cdot \varphi^*(E|_{\widetilde{S}}) = (\hat{\Gamma}+m_p(\Gamma)\: e)\cdot \hat{C} = \hat{\Gamma}\cdot \hat{C} + m_p(C)\cdot m_p(\Gamma)= \hat{\Gamma}\cdot \hat{C}.$$ \par 
In order to show \subcref{lem:singular_hyperplane_section_weak_Fano}, we first note that the divisor $E|_{\widetilde{S}}=\varphi(\hat{C})$ is a section of $\pi|_E:E\rightarrow C$ not containing any fiber of $\pi|_E$ since $S$ is smooth along $C$. We can also observe that $\varphi(H_{\Ftwo})=H_X|_{\widetilde{S}}$ corresponds to the pullback of a general hyperplane section of $S$ under $\pi$, where $H_{\Ftwo}=e+2f$ denotes the general hyperplane section of $\Ftwo$ under $\eta$ (\Cref{lem:singular_quadric_surface_in_P4}). We now consider the restricted anticanonical divisor $$-K_X|_{\widetilde{S}}\overset{\textup{\ref{lem:formula_for_anticanonical_divisor}}}{=}(3H_X-E)|_{\widetilde{S}}=3\varphi(e+2f)-\varphi(\hat{C}).$$
We can move the support of this divisor away from $p$ (because $p\not \in C$ and the linear system $|H_X|_{\widetilde{S}}|$ is base-point-free). The isomorphism $\Ftwo \setminus e \simeq \widetilde{S}\setminus \pi^{-1}(\{p\})$ then implies that if (and only if) $|3e+6f-\hat{C}|$ is base-point-free, then $|-K_X|_{\widetilde{S}}|$ is base-point-free. \Cref{lem:restricted_divisor_bpf_implies_nef} now gives the desired result that $-K_X$ is nef. \par 
We now suppose that $p\in C$. By \Cref{lem:blowup_quadric_cone_along_curve} and the uniqueness of $\varphi$, we know that $\varphi$ is an isomorphism such that $\varphi^{-1}(E|_{\widetilde{S}}) = \hat{C} + e$. It follows directly from this that $\pi|_{\widetilde{S}} : \widetilde{S} \rightarrow S$ is the blowup of $S$ in $p$.\par 
Given an irreducible curve $\Gamma \sub S$ with strict transform $\widetilde{\Gamma}\sub \widetilde{S}$, we denote its multiplicity in $p$ by $m_p(\Gamma)=\hat{\Gamma} \cdot e\in \Z_{\geq 0}$. Recalling that $e\cdot \hat{C}=1$ (as $p\in C$) and that $e^2=-2$, we deduce that $\Gamma$ is an $m$-secant of $C$ for some $m\in\Z_{>0}$ if and only if 
$$m\leq \widetilde{\Gamma}\cdot E= \widetilde{\Gamma}\cdot E|_{\widetilde{S}} = \varphi^{-1}(\widetilde{\Gamma}) \cdot \varphi^{-1}(E|_{\widetilde{S}}) = \hat{\Gamma}\cdot (\hat{C} + e)=\hat{\Gamma}\cdot \hat{C}+m_p(C)\cdot m_p(\Gamma).$$ \par 
It is left to prove \subcref{lem:singular_hyperplane_section_weak_Fano} for the case that $S$ is singular. We again consider the restricted anticanonical divisor $$-K_X|_{\widetilde{S}}\overset{\textup{\ref{lem:formula_for_anticanonical_divisor}}}{=}(3H_X-E)|_{\widetilde{S}}=3H_X|_{\widetilde{S}}-E|_{\widetilde{S}} = 3 \varphi(H|_{\Ftwo})-\varphi(\hat{C}+e).$$
The fact that $\varphi$ is an isomorphism and $H_{\Ftwo}=e+2f$ now allows us to deduce that if (and only if) $|2e+6f-\hat{C}|$ is base-point-free, then $|-K_X|_{\widetilde{S}}|$ is base-point-free. This implies that $-K_X$ is nef as we have shown in \Cref{lem:restricted_divisor_bpf_implies_nef}. \par 
\end{proof}
Using the previous lemma, we can weaken the assumptions of \Cref{thm:B} for all pairs $(g,d)\in \mathcal{P}_{\textup{plane}}$.
\begin{prop} \label{prop:hyperplane_section_implies_weak_Fano}
Let $C \sub Y$ be a smooth irreducible curve of genus $g$ and degree $d$ that lies on a smooth hyperquadric $Y \sub \Pfour$ and denote the blowup of $Y$ along $C$ by $\pi:X\rightarrow Y$.  Suppose that either 
\begin{itemize}
\item $(g,d)\in\mathcal{P}_{\textup{plane}}\setminus\{(0,4)\}$, or
\item $(g,d)=(0,4)$ and $C$ is contained in a hyperplane.
\end{itemize} 
Then $C$ is contained in a hyperplane and $X$ is weak Fano.
\end{prop}
\begin{proof}
In all cases, $C$ is contained in a hyperplane $H \sub \Pfour$: For $(g,d)\in\mathcal{P}_{\textup{plane}}\setminus\{(0,4)\}$, this follows from \Cref{lem:contained_in_hyperplane_section} and for $(g,d)=(0,4)$ by assumption. Denote the corresponding hyperplane section of $Y$ by $S:=Y\cap H$ and its strict transform under $\pi$ by $\widetilde{S}\sub X$.\par 
We want to show that the restricted linear system $|-K_{\widetilde{S}}|$ is base-point-free.\par 
We first suppose that $S$ is smooth. Then the strict transform $\widetilde{S}$ is isomorphic to $S$ via $\pi$ by the universal property of blowups (\Cref{rmk:universal_property_of_blowups}). Via this isomorphism, the linear system $ \left |-K_X|_{\widetilde{S}}\right|\overset{\textup{\ref{lem:anticanonical_divisor}}}{=} \left|(3 H_X-E )|_{\widetilde{S}}\right|$ corresponds to the linear system $R:= |3H_S-C|$ on $S$, where $E=\pi^{-1}(C)$, $H_S$ is a general hyperplane section of $S$ and $H_X$ is the pullback of a general hyperplane section of $Y$ under $\pi $. In particular, it is enough to show that $R$ is base-point-free.\par 
To do this, we can apply the previously introduced notation and results: In $\Pic(S)$, we have $H_S=f_1+f_2$ (\Cref{lem:smooth_quadric_surface_in_P4}) and $C=af_1+bf_2$, where $a,b\in\{0,1,2,3\}$ are given in \Cref{table:a_b_pairs_curves_in_P1xP1}. So the residual system $R=|(3-a)f_1+(3-b)f_2|$ is given by the zero sets of bihomogeneous polynomials in $\kPxP_{3-a,3-b}$ of bidegree $(3-a,\:3-b)\in \Z_{\geq 0}^2\setminus \{(0,0)\}$. Note that $|f_i|$ is base-point-free for $i\in\{1,2\}$ since $f_i$ is a fiber of the $i$-th projection $\PxP\rightarrow \Pone$ and two distinct such fibers are disjoint. The residual system $R$ is then also base-point-free as a sum of base-point-free linear systems (\Cref{lem:sum_of_base_point_free_linear_systems}). 
\par 
We now suppose that $S$ is singular and denote its unique singular point by $p\in S$. We first note that we then must have $(g,d)\in \mathcal{P}_{\textup{plane}}\setminus\{(0,4)\}$ by \Cref{lem:existence_of_curves_in_quadric_hypersurface}. \par
We again use the results and notation of \Cref{lem:singular_quadric_surface_in_P4} and \ref{lem:curves_in_singular_quadric_surface_in_P4}: Considering the blowup $\eta:\Ftwo \rightarrow S$ of $S$ in its singular point $p\in S$, the strict transform $\hat{C}\sub \Ftwo$ of $C$ is given by $\hat{C}=ae + (2a+b)f$ where $(a,b)\in \left(\Z_{\geq 0} \times \{0,1\}\right) \setminus \{(0,0)\}$ is given in \Cref{table:a_b_pairs_curves_in_quadric_cone}.  Recall that $p\not\in C$ if and only if $b=0$, or equivalently, if and only if the degree $d$ of $C$ is even. \par 
So if $p\not \in C$, then $(g,d)\in\{(0,2),\:(1,4),\:(4,6)\}$. In that case, \Cref{lem:singular_hyperplane_section_weak_Fano} implies that it suffices to show that $R=|3e+6f-\hat{C}|$ is base-point-free. Otherwise, if $p\in C$, we must have that $(g,d)\in \{(0,1),\:(0,3),\:(2,5)\}$. Again by \Cref{lem:singular_hyperplane_section_weak_Fano} , we then only have to show that $R=|2e+6f-\hat{C}|$ is base-point-free.\par
For $(g,d)=(4,6)$, the residual system $R$ is trivial, hence $R$ is base-point-free. For $(g,d)=(1,4)$, the same conclusion follows from the fact that $R=|e+2f|$ corresponds to the linear system of a hyperplane section of $\Ftwo$ (\Cref{lem:singular_quadric_surface_in_P4}). If $(g,d)=(2,5)$, then $R=|f|$ is base-point-free as $f$ is a general fiber of the projection $\Ftwo \rightarrow \Pone$ and two distinct such fibers are disjoint. For the other $(g,d)$-pairs, the residual system $R$ is given as a sum of the base-point-free-systems $|e+2f|$ and $|f|$ (see \Cref{table:residual_system_curves_in_quadric_cone}) and is thus also base-point-free (\Cref{lem:sum_of_base_point_free_linear_systems}). \par
In all cases, we have shown that $|-K_X|_{\widetilde{S}}|$ is indeed base-point-free. Moreover, note that $S$ is normal since it is either smooth or isomorphic to a singular quadric cone in $\Pthree$ (\Cref{lem:quadric_surface_in_P4}). So \Cref{lem:restricted_divisor_bpf_implies_nef} implies that $-K_X$ is nef. As moreover, all pairs $(g,d)\in\mathcal{P}_{\textup{plane}}$ satisfy $ 26-3d+g>0$, it follows from \Cref{cor:weak_Fano_big_nef_basepoinfree} that $X$ is weak Fano.
\end{proof}
\vspace{-0.5cm}
\begin{table}[H] 
\caption{Residual system $R$ on $\Ftwo$ if the smooth irreducible curve $C\sub \Pfour$ of degree $d$ and genus $g$ is contained in a smooth hyperquadric $Y\sub \Pfour$ and in a singular hyperplane section $S$ of $Y$. Here $f$ denotes a general fiber of the projection $\Ftwo \rightarrow \Pone$ and $e=\eta^{-1}(C)$ is the exceptional divisor of the blowup $\eta:\Ftwo\rightarrow S$ of $S$ in its singular point $p\in S$} \label{table:residual_system_curves_in_quadric_cone}
\begin{minipage}{\textwidth}
\centering
\vspace{6pt}
\begin{tabular}{|l|l|l|l|l|l|l|l|}
\hline
$(g,d)$ & $(0,1)$ & $(0,2)$ & $(0,3)$ & $(1,4)$ & $(2,5)$ & $(4,6)$ \\
\hline
$\hat{C}$ & $f$ & $e+2f$ & $e+3f$ & $2e+4f$ & $2e+5f$ & $3e+6f$ \\
\hline
$R$ & $|2e+5f|$ & $|2e+4f|$ & $|e+3f|$ & $|e+2f|$ & $|f|$ & $|0|$\\
\hline
\end{tabular}
\vspace{6pt}
\end{minipage}
\end{table}

Given that the curve $C$ is contained in a hyperplane of $\Pfour$, we now show the converse of \Cref{prop:hyperplane_section_implies_weak_Fano}.
\begin{lem} \label{lem:weak_Fano_hyperplane_weak_Fano_implies_in_list}
Let $C \sub Y$ be a smooth irreducible curve of genus $g$ and degree $d$ that lies on a smooth hyperquadric $Y \sub \Pfour$ and denote the blowup of $Y$ along $C$ by $\pi:X\rightarrow Y$. \par
If $C$ is contained in a hyperplane, then $X$ is weak Fano if and only if $ (g,d) \in \mathcal{P}_{\textup{plane}}$.
\end{lem}
\begin{proof}
By assumption, $C$ is contained in a hyperplane $H\sub \Pfour$ and thus also in the hyperplane section $S:=Y \cap H$. The quadric surface $S$ is either isomorphic to $\PxP$ or to a quadric cone in $\Pthree$ (\Cref{lem:quadric_surface_in_P4}).\par
If $X$ is weak Fano, then $C$ is contained in a smooth hypercubic section $Y\cap Z$ for some irreducible hypercubic $Z\sub \Pfour$ by \Cref{prop:weak_Fano_implies_contained_in_smooth_hypercubic_section}. In particular, $Y \cap Z$ is irreducible (\Cref{lem:smooth_implies_irred}) and thus cannot contain the quadric surface $S=Y\cap H$. Hence, Bézout's theorem yields $d\leq \deg(Z)\cdot \deg(Y)\cdot \deg(H)=6$.\par
We first consider the case that $S$ is smooth. Using the same notation as in \Cref{lem:smooth_quadric_surface_in_P4}, we can find $(a,b)\in \Z^2_{\geq 0} \setminus\{(0,0)\}$ such that $C\simeq af_1+bf_2$ in $\Pic(S)$. By the formulas of \Cref{lem:smooth_quadric_surface_in_P4}, the curves $f_1$ and $f_2$ correspond to lines on $S$. They satisfy $E\cdot \widetilde{f_1}=C\cdot f_1 = a$ and $E\cdot \widetilde{f_2}=C \cdot f_2 = b$, where $E:=\pi^{-1}(C)$ is the exceptional divisor and $\widetilde{f_i}\sub X$ denotes the strict transforms of $f_i$ under $\pi$ for $i\in\{1,2\}$. \par 
Therefore, we must have $a,b\leq 3$. Otherwise, $f_1$, or $f_2$ respectively, would be a 4-secant line of $C$. This is in contradiction to \Cref{lem:3n+1_secants} and the fact that $X$ is weak Fano. The $(g,d)$-pairs of \Cref{lem:existence_of_curves_in_smooth_quadric_hypersurface} satisfying this condition are exactly the asserted pairs $(g,d)\in\mathcal{P}_{\textup{plane}}$ (see \Cref{table:a_b_pairs_curves_in_P1xP1}). \par
We now suppose that $S$ is singular. We use the same notation as in \Cref{lem:singular_quadric_surface_in_P4} and denote by $\hat{C}\sub \Ftwo$ the strict transform of $C$ under the blowup $\eta:\Ftwo \rightarrow S$ of $S$ in its singular point $p\in S$.
\par 
By \Cref{lem:curves_in_singular_quadric_surface_in_P4}, we can find $(a,b)\in \left(\Z_{\geq 0} \times \{0,1\}\right) \setminus \{(0,0)\}$ such that $\hat{C} = a e + (2a+b)f$ in $\Pic(\Ftwo)$. Consider a line $\ell\sub S$ and its strict transform $\hat{\ell}\in \Ftwo$ under $\eta$. Since $\deg(\ell)=1$ and $p_a(\ell)=0$, we must have $\hat{\ell} \cdot e=1$ and $\hat{\ell} \cdot f=0$ in $\Ftwo$. That is to say, $\hat{\ell} =f$ in $\Pic(\Ftwo)$. \par
The line $\ell$ must satisfy $\hat{\ell}\cdot \hat{C} + m_p(C)\cdot m_p(\Gamma )\leq 3$. Otherwise, $\ell$ would be a 4-secant line of $C$ by \Cref{lem:singular_hyperplane_section_m_secant}, which again contradicts the assumption that $X$ is weak Fano. We therefore have $a= \hat{\ell}\cdot \hat{C} \leq \hat{\ell}\cdot \hat{C} + m_p(C)\cdot m_p(\Gamma )\leq 3$. The $(g,d)$-pairs of \Cref{lem:existence_of_curves_in_singular_quadric_hypersurface} satisfying this condition are precisely the asserted pairs $(g,d)\in\mathcal{P}_{\textup{plane}}\setminus \{(0,4)\}$ (see \Cref{table:a_b_pairs_curves_in_quadric_cone}).\\
In both cases, we have shown that if $X$ is weak Fano, then $ (g,d) \in \mathcal{P}_{\textup{plane}}$. Conversely, if $C$ is contained in a hyperplane and $(g,d)\in \mathcal{P}_{\textup{plane}}$, then $X$ is weak Fano by \Cref{prop:hyperplane_section_implies_weak_Fano}.
\end{proof}
\subsection{Existence of curves on a smooth hyperquadric section} \label{subsection:hyperquadric_preliminaries}
In comparison to \S\ref{subsection:hyperplane_preliminaries} and \S\ref{subsection:hyperplane_weak_Fano}, we will assume in the next two subsections that the smooth irreducible curve $C\sub \Pfour$ is contained in a hyperquadric section $S$ of $Y$, rather than a hyperplane section. We additionally require $S$ to be smooth. \par
Under these assumptions, we will see in the following lemma that $S$ is a smooth quartic del Pezzo surface. Or equivalently, $S$ is the blowup of $\Ptwo$ in five points in general position. 
\begin{lem} \label{lem:quartic_del_Pezzo_surface}
Let $S$ be a smooth del Pezzo surface of degree 4. There are two equivalent representations of smooth quartic del Pezzo surfaces, which will be used interchangeably in the following:
\begin{enumlem}
\item $S$ can be embedded into $\Pfour$ as a smooth complete intersection of two hyperquadrics. \label{quartic_del_Pezzo_surface_intersection_hyperquadrics}
\item $S$ is the blowup of $\Ptwo $ in five distinct points, three of which are not collinear. \label{quartic_del_Pezzo_surface_blowup_in_five_points}
\end{enumlem}
Let $p_1,\dots,p_5\in \Ptwo$ be five such distinct points, no three of them collinear, such that $S\sub \Pfour$ is given as the blowup $\eta: S \rightarrow \Ptwo $ of $\Ptwo$ in $p_1,\dots,p_5$. \par 
Denoting the pullback of a general line under $\eta$ by $L$ and the exceptional divisors by $E_i=\eta^{-1}(\{p_i\})$ for $i \in \{1,\dots,5\}$, we have $\textup{Pic}(S) = \Z L \oplus \Z E_1 \oplus {\dots} \oplus \Z E_5 $. Moreover, the intersection form on $\textup{Pic}(S)$ is given by $L^2=1$, $E_i^2 = -1$ and $E_i \cdot E_j = 0 = E_i \cdot L$ for all $i,\:j\in\{1,\dots,5\}$ with $i\not=j$.\par
Consider an irreducible curve $\Gamma \sub \Ptwo$ of degree $k>0$ in $\Ptwo$ with multiplicities $m_i = \textup{mult}_{p_i}(\Gamma) \geq 0$ in $p_i$ for $i\in\{1,\dots,5\}$. In $\textup{Pic}(S)$, its strict transform $\hat{\Gamma}:= \overline{\eta^{-1}(\Gamma \setminus \{p_1,\dots,p_5\})}$ in $ S$ is equal to $ \hat{\Gamma} = k L - m_1 E_1 - {\dots} - m_5 E_5$. 
Any divisor of this form is a curve in $S\sub\Pfour$, whose genus and degree are given by
$$ \deg(\hat{\Gamma}) = 3k - \sum_{i=1}^5 m_i \quad \textup{and} \quad g(\hat{\Gamma})= \frac{(k-1)(k-2)}{2} - \sum_{i=1}^5 \frac{m_i(m_i-1)}{2}.$$
In particular, $S$ contains exactly 16 lines. Namely:
\begin{itemize}
\item the five exceptional divisors $E_1,\dots,E_5$,
\item the strict transform $2 L - \sum_{i=1}^5 E_i$ of the unique conic in $\Ptwo $ passing through $p_1,\dots,p_5$, and
\item the strict transforms $L - E_i-E_j$ of the ten lines in $\Ptwo$ passing through two points $p_i, p_j $ with $i,j\in\{1,\dots,5\}$ and $i\not = j$.
\end{itemize}
\end{lem}
\begin{proof}
The equivalence of \subcref{quartic_del_Pezzo_surface_intersection_hyperquadrics} and \subcref{quartic_del_Pezzo_surface_blowup_in_five_points}, as well as the characterizations of lines on $S$, is shown in \cite[Chapter 4.4, p.\ 550-552]{Griffiths_Harris_Principles_of_Alg_Geom}. The statements regarding the Picard group and its intersection form, as well as the genus and degree of an irreducible curve in $S$, can be proved analogously to the corresponding statements about cubic surfaces in $\Pthree$ (which are in 1-to-1-correspondence with the blowup of $\Ptwo$ in six points in general position), as they are proved in \cite[Prop.\ V.4.8]{Hartshorne}.
\end{proof}
\begin{lem} \label{lem:degree_k_and_multiplicities_mi_of_curve_in_quartic_del_Pezzo}
Let $C\sub \Pfour$ be a smooth irreducible curve of degree $d$ and genus $g$ with $(g,d)\in\mathcal{P}_{\textup{quadric}}$. Suppose that $C$ is contained in a smooth del Pezzo surface $S\sub \Pfour$ of degree 4.
Then there are
\begin{itemize}
\item 5 distinct points $p_1,\dots,p_5\in\Ptwo$, three of which are not collinear
\item a birational morphism $\eta:S\rightarrow \Ptwo$ corresponding to the blowup of $\Ptwo$ in $p_1,\dots,p_5$, and
\item an irreducible curve $\Gamma \sub \Ptwo $ with $k:=\deg(\Gamma)$ and $m_i:=\textup{mult}_{p_i}(\Gamma)$ for $i\in\{1,\dots,5\}$, as they are given in \Cref{table:g_d_pairs_quartic_del_Pezzo_surface},
\end{itemize} 
such that $C=\overline{\eta^{-1}(\Gamma\setminus \{p_1,\dots,p_5\})}\sub S$ is the strict transform of $\Gamma$ under the blowup $\eta$.\par
\end{lem}
\begin{proof}
By \Cref{lem:quartic_del_Pezzo_surface}, there are 5 distinct points $p_1,\dots,p_5\in\Ptwo$, three of which are not collinear, such that $S$ is the blowup of $\Ptwo$ in $p_1,\dots,p_5$. We denote the blowup map (which is induced by $|-K_S|$) by $\eta:S\rightarrow \Ptwo$.  \par 
We use the notation of \Cref{lem:quartic_del_Pezzo_surface}. Note that $C$ cannot be contained in a hyperplane of $\Pfour$. Otherwise, Bézout's theorem yields the contradiction that $C$ is contained in the curve $S\cap H$ of degree 4. This is only possible if either $d\leq 3$ or $(g,d)=(1,4)$ (\Cref{lem:genus_formula_complete_intersection}), which is both not satisfied for the pairs $(g,d)\in\mathcal{P}_{\textup{quadric}}$. \par 
In particular, $C$ cannot be a line and thus, it does not correspond to one of the exceptional divisors $E_1,\dots E_6$ on $S$. So $C$ is of the form $C=\hat{\Gamma}=kL-m_1E_1-{\dots} -m_5E_5$, where $\Gamma \sub \Ptwo$ is an irreducible curve of degree $k\in \Z_{>0}$ with multiplicities $m_i:=\textup{mult}_{p_i}(\Gamma)\in \Z_{\geq 0}$ in $p_i$ for $i\in\{1,\dots,5\}$. \par
We can reorder the indices such that $m_1\leq m_2 \leq m_3 \leq m_4 \leq m_5$. Moreover, we may assume that $k\geq m_3 + m_4 + m_5$, which can be argued in the same way as in \cite[Set-Up 4.1]{Blanc_Lamy_Weak_Fano_threefolds}. \par 
The degree and genus of $C$ are then given by $$ d = 3k - \sum_{i=1}^5 m_i \quad \textup{and} \quad g= \frac{(k-1)(k-2)}{2} - \sum_{i=1}^5 \frac{m_i(m_i-1)}{2}.$$
Given a pair $(g,d)\in \mathcal{P}_{\textup{quadric}}$, we can compute numerically which $k\in \Z_{>0}$ and $(m_1,\dots,m_5)\in \Z_{\geq 0}^5$ satisfy the above equalities for $d$ and $g$, as well as $m_1\leq m_2 \leq m_3 \leq m_4 \leq m_5$ and $k\geq m_3 + m_4 + m_5$. In this way, we obtain all the possibilities listed in \Cref{table:g_d_pairs_quartic_del_Pezzo_surface}.
\end{proof}
\vspace{-0.5cm}
\begin{table}[H]
\caption{If a smooth irreducible curve $C\sub \Pfour$ of degree $d$ and genus $g$ is contained in a smooth quartic del Pezzo surface $S\sub \Pfour$ and satisfies $(g,d)\in \mathcal{P}_{\textup{quadric}}$, then $C=kL-m_1E_1-{\dots} -m_5E_5$ in $\Pic(S)$ for some $k\in \Z_{>0}$ and $m_1,\dots m_5\in\Z_{\geq 0}$. Denoting a general hyperplane section of $S$ by $H_S$, the linear system $R:=|H_S-C|$ on $S$ is defined and used in the proof \Cref{lem:quartic_del_Pezzo_in_list_implies_weak_Fano}. A general member of $\Lambda(a;b_1, \dots,b_5)$ is given by the strict transform in $S$ of a curve in $\Ptwo$ which has degree $a\in \Z_{>0}$ and multiplicities $b_i\in\Z_{\geq 0}$ in $p_i$ for all $i\in\{1,\dots,5\}$.}
\label{table:g_d_pairs_quartic_del_Pezzo_surface}
\begin{minipage}{\textwidth}
\vspace{5pt}
\centering
\begin{tabular}{|l|l|l|l|l|}
\hline
$g$ & $d$ & $k$ & $(m_1,\dots,m_5)$  & residual system $R$ (see proof of \Cref{lem:quartic_del_Pezzo_in_list_implies_weak_Fano}) \\
\hline
0& 4 & 2 & $(0,0,0,1,1)$  &  $\Lambda(7;3,3,3,2,2)=$\\
& & & &$\Lambda(3;1,1,1,1,1)+\Lambda(2;1,1,1,1,0)+\Lambda(2;1,1,1,0,1)$  \\
0& 5 & 2 & $(0,0,0,0,1)$ & $\Lambda(7;3,3,3,3,2)=$ \\
& & & & $\Lambda(3;2,1,1,1,1)+\Lambda(2;0,1,1,1,1)+\Lambda(2;1,1,1,1,0)$  \\
0& 6 & 2 & $(0,0,0,0,0)$ & $\Lambda(7;3,3,3,3,3)$ not base-point-free\\
& & 3 & $(0,0,0,1,2)$ & $\Lambda(6;3,3,3,2,1)=\Lambda(2;1,1,1,1,0)+\Lambda(2;1,1,1,0,1)$   \\
1 & 5 & 3 & $(0,1,1,1,1)$ & $\Lambda(6;3,2,2,2,2)= \Lambda(3;2,1,1,1,1)+\Lambda(3;1,1,1,1,1)$  \\
1 & 6 & 3 & $(0,0,1,1,1)$ & $\Lambda(6;3,3,2,2,2)=\Lambda(3;2,1,1,1,1)+\Lambda(3;1,2,1,1,1)$   \\
2 & 6 & 4 & $(1,1,1,1,2)$ &  $\Lambda(5;2,2,2,2,1)=\Lambda(3;1,1,1,1,1)+\Lambda(2;1,1,1,1,0)$  \\
2 & 7 & 4 & $(0,1,1,1,2)$ & $\Lambda(5;3,2,2,2,1)=\Lambda(3;2,1,1,1,1)+\Lambda(2;1,1,1,1,0)$   \\
3 & 7 & 4 & $(1,1,1,1,1)$ &  $\Lambda(5;2,2,2,2,2)=\Lambda(3;2,1,1,1,1)+\Lambda(2;0,1,1,1,1)$  \\
3 & 8 & 4 & $(0,1,1,1,1)$ & $\Lambda(5;3,2,2,2,2)$ not base-point-free \\
& & 5 & $(1,1,1,1,3)$ & $\Lambda(4;2,2,2,2,0)=\Lambda(2;1,1,1,1,0)+\Lambda(2;1,1,1,1,0)$   \\
4 & 8 &5  & $(1,1,1,2,2)$ &  $\Lambda(4;2,2,2,1,1)=\Lambda(2;1,1,1,1,0)+\Lambda(2;1,1,1,0,1)$   \\
5& 8 & 6 & $(2,2,2,2,2)$ & $\Lambda(3;1,1,1,1,1)$   \\
6& 9 & 6 & $(1,2,2,2,2)$ & $\Lambda(3;2,1,1,1,1)$   \\
8& 10 & 7 & $(2,2,2,2,3)$ & $\Lambda(2;1,1,1,1,0)$  \\
13 & 12 & 9 & $(3,3,3,3,3)$ & $\Lambda(0;0,0,0,0,0)$ \\
\hline
\end{tabular}
\end{minipage}
\end{table}

\begin{rmk} \label{rmk:quadric_del_Pezzo_4_secant_lines}
Note that the pairs $(g,d)\in \{(0,6),\: (3,8)\}\sub \mathcal{P}_{\textup{quadric}}$ are the only pairs that cannot be reduced to a unique representation in $\Pic(S)$. In fact, these two pairs are the only pairs admitting 4-secant lines in $S$: If $(g,d)=(0,6)$ and $C=2L$, or $(g,d)=(3,8)$ and $C=4L-E_2-E_3-E_4-E_5$, then $D=2L-\sum_{i=1}^5 E_i$ is a 4-secant line of $C$. However, if $(g,d)\in \mathcal{P}_{\textup{quadric}}\setminus \{(0,6),\: (3,8)\}$, then $C$ has no 4-secant line, as we will see in the proof of \Cref{lem:quartic_del_Pezzo_weak_Fano_implies_in_list}.
\end{rmk}
%
%
%
%
%
\subsection{Curves on a hyperquadric section yielding a weak Fano threefold} \label{subsection:hyperquadric_weak_Fano}
In the following, we want to prove that if a smooth irreducible curve $C\sub \Pfour$ of degree $d$ and genus $g$ is contained in a smooth hyperquadric $Y\sub \Pfour$ and in a smooth hyperquadric section of $Y$, $(g,d)\in \mathcal{P}_{\textup{quadric}}$ and $C$ admits no 4-secant lines in the cases $(g,d)\in \{(0,6),\: (3,8)\}$, then the blowup $X$ of $Y$ along $C$ is weak Fano. \par 
The assumption that $C$ is contained in a smooth hyperquadric section $S$ of $Y$ is stronger than the assumption from \Cref{thm:B} that $C$ lies on a smooth hypercubic section of $Y$. Nevertheless, the former is often easier to verify, since in that case, $S$ is a smooth quartic del Pezzo surface whose structure is well understood \par 
\begin{lem}\label{lem:quartic_del_Pezzo_in_list_implies_weak_Fano} \label{lem: }
Let $C \sub Y$ be a smooth irreducible curve of genus $g$ and degree $d$ that lies on a smooth hyperquadric $Y \sub \Pfour$ and denote the blow up of $Y$ along $C$ by $\pi:X\rightarrow Y$. \par 
Suppose that $C$ is contained in a smooth hyperquadric section $S$ of $Y$. If $(g,d)\in \mathcal{P}_{\textup{quadric}}$ and $C$ does not have a 4-secant line in the cases $(g,d)\in \{(0,6),(3,8)\}$, then $X$ is weak Fano.  
\end{lem}
\begin{proof}
By \Cref{lem:quartic_del_Pezzo_surface}, $S$ is a smooth quartic del Pezzo surface and the blowup $\eta:S\rightarrow \Ptwo$ of $\Ptwo$ in five distinct points $p_1,\dots,p_5\in \Ptwo$, three of which are not collinear. Denote the pullback of a general line in $\Ptwo$ under $\eta$ by $L$ and the exceptional divisors by $E_i=\eta^{-1}(\{p_i\})$ for $i\in \{1,\dots,5\}$.
Moreover, we introduce the notation 
$$\Lambda(a;b_1,b_2,b_3,b_4,b_5):=\left|aL-\sum_{i=1}^5 b_i E_i\right|$$ to denote the linear system on $S$, whose general member is the strict transform of an irreducible curve in $\Ptwo$ of degree $a\in \Z_{>0}$ and with multiplicities $b_i\in\Z_{\geq0}$ in $p_i$ for all $i\in\{1,\dots,5\}$. \par 
We want to show that the restricted linear system $|-K_X|_{\widetilde{S}}|$ is base-point-free, where $\widetilde{S}\sub X$ denotes the strict transform of $S$ under $\pi$. \par 
Since $S$ is smooth, the blowup $\pi:X\rightarrow Y$ induces an isomorphism $\widetilde{S}\simeq S$ by the universal property of blowups (\Cref{rmk:universal_property_of_blowups}). Moreover, the smoothness of $S$ implies that $E|_{\widetilde{S}}$ is a section of $\pi|_E:E\rightarrow C$, where $E:=\pi^{-1}(C)$. That is to say, the section $E|_{\widetilde{S}}$ is isomorphic to $C$ via $\pi$. For both of these reasons, the linear system $|-K_X|_{\widetilde{S}}|\overset{\textup{\ref{lem:formula_for_anticanonical_divisor}}}{=}|3H_X|_{\widetilde{S}}- E|_{\widetilde{S}}|$ on $\widetilde{S}$ corresponds to the linear system $R:= |3H_S-C|$ on $S$, where $H_S$ denotes a general hyperplane section of $S$. \par 
In particular, to prove that $|-K_X|_{\widetilde{S}}|$ is base-point-free, it suffices to show that $R$ is base-point-free. To do this, note that we can use the adjunction formula to deduce that a general hyperplane section $H_S$ of $S$ is given by $H_S=-K_S=3L-E_1-{\dots}-E_5$. In other words, we have $|H_S| = \Lambda(3;1,1,1,1,1)$. \par 
Moreover, \Cref{lem:degree_k_and_multiplicities_mi_of_curve_in_quartic_del_Pezzo} gives us that $|C|=\Lambda(k;m_1,\dots,m_5)$, where $m_1,\dots,m_5\in\Z_{\geq 0}$ and $k\in \Z_{>0}$ are given in \Cref{table:g_d_pairs_quartic_del_Pezzo_surface}. So $R=|3H_S - C| =\Lambda(9-k;3-m_1,\dots,3-m_5)$ is the linear system given in the last column of \Cref{table:g_d_pairs_quartic_del_Pezzo_surface}, where $k$ and $m_1,\dots,m_5$ are given by the third and fourth column of \Cref{table:g_d_pairs_quartic_del_Pezzo_surface} in dependence of the pair $(g,d)$.\par
In all cases, $R$ can be written as a sum of base-point-free linear systems and is thus itself base-point-free (\Cref{lem:sum_of_base_point_free_linear_systems}). Indeed, $\Lambda(l;0,0,0,0,0)$ corresponds to the linear system of hypersurfaces in $\Ptwo$ of degree $l\in\{1,2\}$ not containing $C$. This is base-point-free by Bertini's theorem. \par
Furthermore, $\Lambda(2;1,1,1,1,0)$ corresponds to the linear system of plane conics through four distinct points in $\Ptwo$, three of which are not collinear; and $\Lambda(3;2,1,1,1,1)$ corresponds to the linear system of plane cubics through 5 distinct points in $\Ptwo$ that have a double point in one them such that no three points of $p_1,\dots, p_5$ are collinear. Both of these linear systems are base-point-free (follows from \cite[Prop.\ V.4.1 + V.4.3]{Hartshorne}).\par 
We have thus shown that the linear system $R$ and hence also $|-K_X|_{\widetilde{S}}|$ is indeed base-point-free. This implies that $-K_X$ is nef (\Cref{lem:restricted_divisor_bpf_implies_nef}). As moreover, all pairs $(g,d)\in\mathcal{P}_{\textup{quadric}}$ satisfy the inequality $26-3d+g > 0$, we can use \Cref{cor:weak_Fano_big_nef_basepoinfree} to deduce that $X$ is weak Fano.
\end{proof}
We can also show the converse direction of \Cref{lem:quartic_del_Pezzo_in_list_implies_weak_Fano}.

\begin{lem} \label{lem:quartic_del_Pezzo_weak_Fano_implies_in_list}
Let $C \sub Y$ be a smooth irreducible curve of genus $g$ and degree $d$ that lies on a smooth hyperquadric $Y \sub \Pfour$ and denote the blow up of $Y$ along $C$ by $\pi:X\rightarrow Y$. Suppose furthermore that $C$ is contained in a smooth hyperquadric section $S$ of $Y$. \par
Then $X$ is weak Fano if and only if \begin{itemize}
\item either $ (g,d)\in \{(0,1), (0,2), (0,3), (1,4)\}$,
\item or $(g,d)\in \mathcal{P}_{\textup{quadric}}$ and $C$ does not have a 4-secant line if $(g,d)\in \{(0,6),(3,8)\}$.
\end{itemize} 
\end{lem}
\begin{proof}
If $(g,d)\in \{(0,1), (0,2), (0,3), (1,4)\}\sub \mathcal{P}_{\textup{plane}}$, we have shown in \Cref{prop:hyperplane_section_implies_weak_Fano} that $X$ is weak Fano.	If $(g,d)\in \mathcal{P}_{\textup{quadric}}$ and $C$ does not have a 4-secant line in the cases $(g,d)\in \{(0,6),(3,8)\}$, it follows from \Cref{lem:quartic_del_Pezzo_in_list_implies_weak_Fano} that $X$ is weak Fano.  \par
Suppose conversely that $X$ is weak Fano. Then \Cref{lem:3n+1_secants} shows that $C$ cannot admit any 4-secant lines. We want to show that $(g,d)\in \mathcal{P}_{\textup{plane}}\cup \{(0,1), (0,2), (0,3), (1,4)\}$. \par 
By \Cref{lem:quartic_del_Pezzo_surface}, $S$ is a smooth quartic del Pezzo surface and the blowup $\eta:S\rightarrow \Ptwo$ of $\Ptwo$ in five distinct points $p_1,\dots,p_5\in \Ptwo$, three of which are not collinear. Denote the pullback of a general line under $\eta$ by $L$ and the exceptional divisors by $E_i=\eta^{-1}(\{p_i\})$ for $i\in \{1,\dots,5\}$. \par
It follows from \Cref{lem:quartic_del_Pezzo_surface} that $S$ contains exactly 16 lines and that $C= kL - \sum_{i=1}^5 m_i E_i$ is the strict transform of an irreducible curve in $\Ptwo$ of degree $k$ with multiplicities $m_i$ in $p_i$ for $i\in\{1,\dots,5\}$. \par
Since $X$ is weak Fano, $C$ has at most 3-secant lines in $Y$ (\Cref{lem:3n+1_secants}). That is to say, for each of the 16 lines $D \sub S$ of $S$, we have $C\cdot D = E \cdot \widetilde{D} \leq 3$, where $E:=\pi^{-1}(C)$ and $\widetilde{D}\sub X$ denotes the strict transform of $D$ under the blowup $\pi:X\rightarrow Y$ (and where the intersection is taken in $S$). \par 
This imposes the following conditions on $k$ and $m_1,\dots,m_5$:
\begin{itemize}
\item if $D \simeq E_j$ for some $j\in\{1,\dots,5\}$, then $m_j\leq 3$;
\item if $D \simeq 2 L - \sum_{i=1}^5 E_i $, then $2k-\sum_{i=1}^5 m_i \leq 3$;
\item if $D \simeq L-E_i-E_j$ for $i,j\in\{1,\dots,5\}$ with $i\not=j$, then $k-m_i-m_j\leq 3$.
\end{itemize}
Running numerically through all possible $k$ and $m_1,\dots,m_5$ satisfying these conditions and computing the corresponding degree $d\geq 1$ and genus $g\geq 0$ with the formulas from \Cref{lem:quartic_del_Pezzo_surface}, it follows that $(g,d)\in \mathcal{P}_{\textup{quadric}} \cup \{(0,1), (0,2), (0,3), (1,4)\}$.
\end{proof}
\section{Extremal contractions and Sarkisov links} \label{section:Sarkisov_links}
In this section, we will analyze which extremal rays and which contractions of these rays may appear on the weak Fano threefold $X$ that arises as the blowup of a smooth hyperquadric $Y\sub\Pfour$ along a smooth irreducible curve that is contained in $Y$. We will then study the question whether these extremal contractions give rise to a Sarkisov link and if so, which type of links may appear. \par 
Recall that every birational map between Mori fiber spaces over the complex numbers $\mathbb{C}$ with terminal $\mathbb{Q}$-factorial singularities is a composition of Sarkisov links and automorphisms (\cite[Thm.\ 1.1]{Sarkisov_program_Hacon_McKernan}). We will use the classical definition of Sarkisov links via rank 2 fibrations as can be found in e.g.\ \cite[Def.\ 3.8]{Quotients_of_higher_dim_Cremona_groups-Blanc_Lamy_Zimmermann}. Recall, in particular, that projective varieties $X_1$ and $X_2$ appearing in a Sarkisov link $\chi: X_1 \dashrightarrow X_2$ are $\mathbb{Q}$-factorial and terminal. \par 
In \S \ref{subsection:extremal_contractions}, we give an overview of the possible cases that may occur and we introduce the notation that will be used throughout the rest of the section. In \S\ref{subsection:Fano_threefolds}, we consider the case in where $X$ is a Fano threefold, which always results in a Sarkisov link of type I or II. The case where $X$ is not Fano is divided into two subsections: Either $X$ admits two divisorial contractions and does not give rise to a Sarkisov link (\S\ref{subsection:no_flop}). Otherwise, $X$ yields a small anticanonical map, which can be flopped to again obtain a Sarkisov link of type I or II (\S\ref{subsection:flop}). \par 
We will construct the Sarkisov link coming from a curve of type $(g,d)=(1,4)$ (see \Cref{ex:Sarkisov_link_1_4}) since we could not find a reference for it elsewhere. All other Sarkisov links have already been treated by other authors. So we will not construct these links explicitly nor prove why they are of their corresponding type. Rather, we will just compile the results and refer to the relevant sources.
\subsection{Extremal contractions on weak Fano threefolds of Picard rank two} \label{subsection:extremal_contractions}
Let $X$ be a smooth weak Fano threefold of Picard rank two. Recall that $N_1(X)_{\R}$ is the $\R$-vector space of 1-cycles on $X$ with real coefficients modulo numerical equivalence and that the closed convex cone of curves $\NE(X)$ is the closure in $N_1(X)_{\R}$ of the set of classes of effective 1-cycles on $X$. Also recall that an extremal ray of $\NE(X)$ is a one-dimensional subcone $R \sub \NE(X)$ such that any $r,s\in \NE(X)$ with $r+s\in R$ satisfy $r,s\in R$. \par 
Since the Picard rank of $X$ is two, $N_1(X)_{\R}$ is a two-dimensional $\R$-vector space. As a closed convex cone not containing any lines, $\NE(X)\sub N_1(X)_{\R}$ must be the convex hull of its extremal rays. For both of these reasons, $\NE(X)$ must have exactly two extremal rays.\par 
As the anticanonical divisor $-K_X$ is nef, the intersection of $K_X$ with any effective curve is non-negative. So one of the extremal rays $R_1\sub N_1(X)_{\R}$ must be $K_X$-negative, which means that it consists of classes of effective curves intersecting $K_X$ negatively. If $X$ is a Fano threefold, the second extremal ray $R_2\sub N_1(X)_{\R}$ is also $K_X$-negative. Otherwise, if $X$ is weak Fano but not Fano, the second extremal ray $R_2\sub N_1(X)_{\R}$ must be $K_X$-trivial. That is to say, it consists of classes of effective curves that are trivial against the canonical divisor $K_X$. \par 
In both cases, the Cone Theorem (\cite[Thm.\ 6.1]{Debarre-Higher_dimensional_Alg_Geom}) implies that $\NE(X)=R_1+R_2$. Recall that a contraction of an extremal ray $R$ is a projective morphism to a normal variety with connected fibers, which contracts all curves whose numerical equivalence classes belong to the extremal ray $R$. If it exists, such a contraction is unique (\cite[6.11]{Kollar-Intro_to_Moris_Program}).\par 
If the extremal ray $R$ is $K_X$-negative, the Contraction Theorem (\cite[Thm.\ 7.39]{Debarre-Higher_dimensional_Alg_Geom}) guarantees the existence of a contraction morphism $\phi_R:X\rightarrow Z$ of $R$ for some projective variety $Z$. We can use Mori's classification of $K_X$-negative extremal contractions from smooth projective threefolds (\cite[Thm.\ 3.3-3.5]{Mori-Extremal_contractions_threefolds}) to classify $\phi_R:X\rightarrow Z$: Since $X$ is smooth and of Picard rank two (so it cannot be a primitive Fano threefold), there are two possible cases: Either $\phi_R$ is a divisorial contraction and then $\phi_R:X\rightarrow Z$ is the blowup of $Z$ in either a point or along a smooth irreducible curve. The second possibility is that $\phi_R$ is of fibring type and $Z$ is smooth. Then $\phi_R$ is either a conic bundle over $\Ptwo$ (if $\dim(Z)=2$) or $\phi_R$ is a del Pezzo fibration over $\Pone$ (if $\dim(Z)=1$).   \par
We now suppose that $X$ is a weak Fano threefold that is obtained by blowing up a smooth hyperquadric $Y\sub \Pfour$ along a smooth irreducible curve $C\sub Y$. As before, we denote the blowup map by $\pi:X\rightarrow Y$, the exceptional divisor by $E:=\pi^{-1}(C)\sub X$ and a general hyperplane section of $Y$ by $H_Y$. A curve $\Gamma \sub X$ is contracted by $\pi$ if and only if it is a fiber of $\pi|_E:E\rightarrow C$. Every such fiber $f\sub E$ satisfies 
$$K_X\cdot f\overset{\ref{lem:anticanonical_divisor}}{=}(-3\pi^*(H_Y)+E)\cdot f=E\cdot f=-1.$$
So $\pi$ is the contraction of the $K_X$-negative extremal ray $R_1$. More precisely, $\pi$ is a divisorial contraction as the exceptional locus of $\pi$ is the prime divisor $E\sub X$. \par 
The type of the second extremal contraction on $X$ depends on the anticanonical map $\psi:X\rightarrow X'$ associated to the anticanonical linear system $|-mK_X|$ for $m\gg0$, where $X'$ is some projective threefold. The map $\psi$ is an isomorphism if and only if $-K_X$ is ample, or equivalently, if $X$ is a Fano threefold. In this case, there is a second $K_X$-negative extremal contraction, which yields a Sarkisov link of type I or II (\S \ref{subsection:Fano_threefolds}). \par 
If $\psi$ is not an isomorphism, we can still deduce a lot about the anticanonical map $\psi:X\rightarrow X'$: Since $-K_X$ is big and nef, it follows from the Base Point Free Theorem (\cite[Thm.\ 1.3.6]{Shafarevich_Parshin_AlgGeom_Fano_varieties}) that $|-mK_X|$ is base-point-free for $m\gg0$. On the one hand, this guarantees that $\psi$ is a morphism. On the other hand, this implies that $\psi$ has connected fibers and that its image $X'$ is a normal variety (\cite[Prop.\ 7.6(b)]{Debarre-Higher_dimensional_Alg_Geom}). That is to say, $\psi$ is a contraction morphism. \par 
In fact, $\psi$ is the contraction of the second $K_X$-trivial extremal ray: Let $\Gamma \sub X$ be an effective curve that is trivial against $K_X$. Suppose to the contrary that it is not contracted by $\psi$ and mapped to a curve $\Gamma':=\psi(\Gamma) \sub Y$. As $\psi$ is the map associated to the linear system $|-mK_X|$, the pullback of the general hyperplane section $H_Y$ of $Y$ under $\psi$ is precisely $-mK_X$. So we find the contradiction $$0<H_Y\cdot \Gamma' =\psi^*(H_Y)\cdot \psi^*(\Gamma') = (-mK_X) \cdot \Gamma =0.$$ 
Therefore, the morphism $\psi$ is indeed the contraction of the $K_X$-trivial extremal ray. Moreover, we know that $\psi$ is birational  since $-K_X$ is big. By the classification of the contraction of extremal rays (see e.g.\ \cite[Prop.\ 2.5]{Mori_Kollar_Birational_geometry}), $\psi$ is thus either a divisorial contraction or small. \par 
If $\psi$ is a divisorial contraction, it does not yield a Sarkisov link (\S\ref{subsection:no_flop}). If the map $\psi$ is however small, it is a flopping contraction. So we can perform a flop $\chi:X\rightarrow X^+$ (\cite[Thm\ 1.4.15 + Lemma 4.1.1]{Shafarevich_Parshin_AlgGeom_Fano_varieties}), which gives rise to another weak Fano threefold $X^+$ having one $K_{X^+}$-negative and one $K_{X^+}$-trivial extremal ray. Depending on the type of $K_{X^+}$-negative extremal contraction on $X^+$, we again obtain a Sarkisov link of type I or II (\S\ref{subsection:flop}).
\subsection{Fano threefolds} \label{subsection:Fano_threefolds}
Suppose that the anticanonical map $\psi$ is an isomorphism. Then $-K_X$ is ample and so, $X$ is a Fano threefold. As we have argued in \Cref{subsection:extremal_contractions}, $X$ then admits two $K_X$-negative extremal contractions. One of them is the blowup map $\pi:X\rightarrow Y$, which is a divisorial contraction. Denoting the second extremal contraction by $\phi: X\rightarrow Z$ for a projective variety $Z$, we use Mori's classification (\cite[Thm.\ 3.3-3.5]{Mori-Extremal_contractions_threefolds}) to deduce that we are in one of the following two cases:
\begin{enumerate}
\item[(1)] $\phi:X\rightarrow Z$ is a divisorial contraction and then it is the blowup of $Z$ in a point or along a smooth irreducible curve. This case yields a Sarkisov link $Y \dashrightarrow Z$ of type II and appears for the four pairs $(g,d)\in \{(0,3), \: (0,4),\:(1,4),\: (1,5)\}$ (see \Cref{table:Fano_threefolds_divisorial}). 
\item[(2)] $Z$ is smooth and $\phi:X\rightarrow Z$ is of fibring type. More precisely, we have the following dichotomy:
\begin{enumerate}
\item If $\dim(Z)=2$, then $\phi$ is a conic bundle over $\Ptwo$.
\item If $\dim(Z)=1$, then $\phi$ is a del Pezzo fibration over $\Pone$.
\end{enumerate}
Both cases yields a Sarkisov link $Y\dashrightarrow X$ of type I. They appear for the four pairs $(g,d)\in \{(0,1),\:(0,2),\:(2,6),\:(5,8)\}$ (see \Cref{table:Fano_threefolds_fibring_type}).
\end{enumerate}
\begin{minipage}{\textwidth }
\hspace{1pt}
Sarkisov links (dashed arrows) if $X$ is a Fano threefold, depending on whether the second extremal contraction $\phi:X\rightarrow Z$ on $X$ is of fibring type (fib) or a divisorial contraction (div):
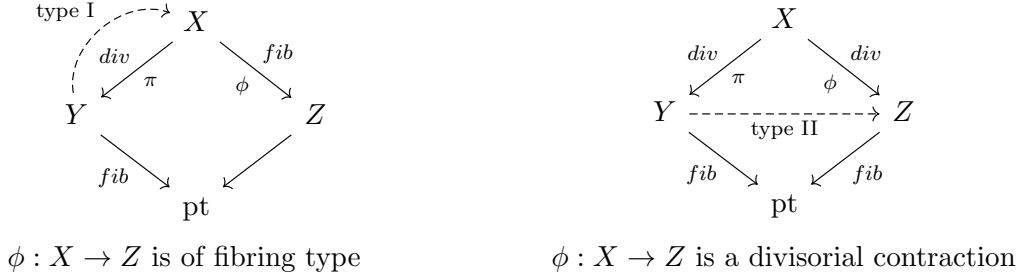
\begin{multicols}{2}
\centering
\begin{equation*}
\begin{tikzcd} 
	& X \ar[dl, "div"', "\pi"]  \ar[dr, "\phi"', "fib"]& \\
	Y \ar[dr, "fib"']\ar[ur, bend left = 60, dashed, "\textup{type I}"]& & Z \ar[dl,]\\
	&\textup{pt} &
\end{tikzcd} 
\end{equation*}	
$\phi:X \rightarrow Z$ is of fibring type

\begin{equation*}
\begin{tikzcd}
	& X \ar[dl, "div"',"\pi"]  \ar[dr, "div","\phi"']& \\
	Y \ar[dr, "fib"']\ar[rr, dashed, "\textup{type II}"']& & Z \ar[dl,"fib"]\\
	& \textup{pt}&
\end{tikzcd}
\end{equation*}
$\phi:X\rightarrow Z$ is a divisorial contraction
\end{multicols}
\hspace{2pt}
\end{minipage}
All of the previously mentioned cases are listed in Mori and Mukai's classification of Fano threefolds of Picard rank 2 in \S 12.3 of \cite{Mori_Mukai-Fano_threefolds_Picard_rank_2}. We have found an explicit construction for all Sarkisov links arising from these cases in other papers, except for the case $(g,d)=(1,4)$. This is why we give the construction and existence of the corresponding link in the following example.
\begin{ex} \label{ex:Sarkisov_link_1_4}
Let $C\sub \Pfour$ be a smooth irreducible curve of degree 4 and genus 1 that is contained in a smooth hyperquadric $Y=\VPfour(f)\sub \Pfour$ for a homogeneous polynomial $f\in\kPfour_2$ of degree 2. We know that $C$ is then contained in a hyperplane $H=\VPfour(\ell)\sub \Pfour$ for some homogeneous linear polynomial $\ell\in\kPfour_1$ (\Cref{lem:contained_in_hyperplane_section}), as well as in at least 7 other hyperquadrics of $\Pfour$ (\Cref{lem:dimension_linear_system_hypersurfaces}). A general curve $C$ of degree 4 and genus 1 is thus be the complete intersection $Y\cap Q\cap H$ for a hyperquadric $Q=\VPfour(g)\neq Y$ and a homogeneous polynomial $g\in\kPfour_2$ of degree 2. \par 
We denote the blowup of $Y$ by $\pi:X\rightarrow Y$ and the exceptional divisor by $E=\pi^{-1}(C)$. Recall that the 1-cycles $N_1(X)$ on $X$ are generated by the pullback $l\sub X$ of a general line in $Y$ and by an exceptional curve $f\sub E$. It is a well-known fact that a ray $al-bf$ for $a\in\Z_{>0},\: b\in \Z_{\geq 0}$ corresponds to the strict transforms in $X$ of curves of degree $a$ that are a $b$-secant of $C$. Recall also that such a ray is extremal if and only if $\frac{b}{a}$ is maximal. \par 
We want to show that the extremal ray is given by $r=l-2f$. Suppose to the contrary that there are  $a\in\Z_{>0},\: b\in \Z_{\geq 0}$ and a curve $\Gamma \sub Y$ of degree $a$ that is a $b$-secant of $C$ such that $\frac{b}{a}> 2$. Then $\Gamma$ is a curve of degree $a$ that is at least a $2a+1$-secant of $C$. We can use Bézout's theorem to derive the contradiction that $\Gamma$ must be contained in the complete intersection $C=Y\cap Q\cap H$. So the extremal ray $r=l-2f$ is indeed given by 2-secant lines of $C$. \par 
The corresponding extremal contraction is given by the linear system $|2H_X-E|$ where $H_X$ denotes the pullback of a general hyperplane section of $Y$ under $\pi$, that is, by hyperquadric sections $S$ of $Y$ containing $C=Y\cap Q\cap H=\VPfour(f,g,\ell)$. Every such quartic surface $S$ must be of the form $\VPfour(f,h)$, where $h=\lambda g + \ell' \ell$ for some constant $\lambda\in\kk$ and a linear polynomial $\ell'\in\kPfour_1$. The linear span of these polynomials $h$ is of dimension 5. Choosing a basis of this span, we obtain the rational map
$$ \chi': \Pfour \dashrightarrow \Pfive,\quad [x_0:{\dots}:x_5 ]\mapsto [x_0\ell:{\dots}:x_4\ell:g]=\left[x_0:{\dots}:x_4:\frac{g}{\ell}\right]$$
The image of $\chi'$ is given by the quartic threefold $Z:=V_{\Pfive}(f, g-x_5\ell)\sub\Pfive$. An easy calculation shows that $Z$ is singular in the point $p:=[0:0:0:0:1]$ and that $Z$ is terminal. We can also verify explicitly that the restriction $\chi:=\chi'|_Y:Y\dashrightarrow Z$ factors as $\chi = \phi \circ \pi^{-1}$, where $\phi:X\rightarrow Z$ is the blowup of $Z$ in its singular point $p$. \par
The morphism $\phi$ contracts the strict transform of the quartic surface $\VPfour(f,\ell)$ (which is either isomorphic to $\PxP$ or to a singular quadric cone in $\Pthree$ by \Cref{lem:quadric_surface_in_P4}) to a point. In particular, $\pi$ and $\phi$ are divisorial contractions and so $\chi:Y\dashrightarrow Z$ is a Sarkisov link of type II.
\end{ex}
\vspace{-0.5cm}
\begin{table}[H]
\caption{$(g,d)$-pairs yielding a Fano threefold $X$ such that 1) The anticanonical map $\psi$ is an isomorphism. 2) One $K_X$-negative extremal ray is the blowup of a smooth hyperquadric $Y\sub \Pfour$ along a smooth irreducible curve $C\sub Y$ of genus $g$ and degree $d$. 3) The second extremal contraction $\phi:X\rightarrow Z$ is a divisorial contraction. More precisely, it is the blowup of a $\mathbb{Q}$-factorial, terminal threefold $Z$ along a smooth irreducible curve or in a point. 
}
\label{table:Fano_threefolds_divisorial}
\begin{minipage}[h]{\textwidth}
\vspace{5pt}
\centering
\begin{tabular}{|l|l|l|l|l|l|}
\hline
$g$ & $d$ & $-K_X^3$ & Z and 2nd extremal contraction $\phi$ & \cite{Mori_Mukai-Fano_threefolds_Picard_rank_2} & Reference, Existence\\
\hline
$0$ & $3$ & 34 & blowup of a Fano threefold $Z$ of Picard rank& Nr.\ 26& \cite[p.\ 117 6)]{Mori_Mukai_On_Fano_Threefolds},
\\
& & & 1, degree 5 and index 1 along a line & & Cor.\ \ref{cor:existence_gdpairs_yielding_weak_Fano}\\
\hline
$0$ & $4$ & 28 & blowup of a smooth hyperquadric $Z=Y\sub \Pfour$ &Nr.\ 21 & \cite[p.\ 117 4)]{Mori_Mukai_On_Fano_Threefolds},
\\
& & & along a smooth rational curve of degree 4 & & Cor.\ \ref{cor:existence_gdpairs_yielding_weak_Fano} \\
\hline
$1$ & $4$ & 30 & blowup of a singular intersection $Z$ of two & Nr.\ 23 & \Cref{ex:Sarkisov_link_1_4}, \\
& & & hyperquadrics in $\Pfour$ in its singular point& & Cor.\ \ref{cor:existence_gdpairs_yielding_weak_Fano}\\
\hline
$1$ & $5$ & 24 & blowup of $Z=\Pthree$ along a smooth elliptic& Nr.\ 17 & \cite[p.\ 117 2)]{Mori_Mukai_On_Fano_Threefolds},
\\
& & & curve of degree 5 & & \cite[Thm.\ 3.7]{Motivic_invariants-Shinder_Lin} \\ 
\hline
\end{tabular}		
\end{minipage}
\end{table}	
\vspace{-0.5cm}
\begin{table}[H]
\caption{$(g,d)$-pairs yielding a Fano threefold $X$ such that 1) The anticanonical map $\psi$ is an isomorphism. 2) One $K_X$-negative extremal ray is the blowup of a smooth hyperquadric $Y\sub \Pfour$ along a smooth irreducible curve $C\sub Y$ of genus $g$ and degree $d$. 3) The second extremal contraction $\phi:X\rightarrow Z$ is of fibring type, in which case $\phi$ is either a conic bundle over $\Ptwo$ or a del Pezzo fibration over $\Pone$ ((abbreviated with "dPf").}
\label{table:Fano_threefolds_fibring_type}
\begin{minipage}[h]{\textwidth}
\vspace{5pt}
\centering
\begin{tabular}{|l|l|l|l|l|l|l|}
\hline
$g$ & $d$ & $-K_X^3$ & 2nd extremal contraction & \cite{Mori_Mukai-Fano_threefolds_Picard_rank_2} & Reference & Existence\\
\hline
$0$ & $1$ & 46 & $\Pone$-bundle over $\Ptwo$ & Nr.\ 31 &  \cite[Lemma 3.22]{Fano_threefolds} & Cor.\ \ref{cor:existence_gdpairs_yielding_weak_Fano}\\
\hline
$0$ & $2$ & 40 & dPf.\ of degree 8  & Nr.\ 29 & \cite[(2.3.1)]{Takeuchi-Weak_Fano_threefolds_del_Pezzo_fibration} & \cite[(4.6.1.)]{Takeuchi-Weak_Fano_threefolds_del_Pezzo_fibration}
\\
\hline
$2$ & $6$ & 20 & conic bundle over $\Ptwo$& Nr.\ 13 & \cite[\S 5.5]{Fano_threefolds} & Cor.\ \ref{cor:existence_gdpairs_yielding_weak_Fano}\\
\hline
$5$ & $8$ & 14 & dPf\ of degree 4 & Nr.\ 7 & \cite[(2.11.3.)]{Takeuchi-Weak_Fano_threefolds_del_Pezzo_fibration} 
& \cite[(7.6.3.)]{Takeuchi-Weak_Fano_threefolds_del_Pezzo_fibration} 
\\ 
\hline
\end{tabular}		
\end{minipage}
\end{table}	
\subsection{Weak Fano threefolds not yielding a Sarkisov link} \label{subsection:no_flop}
Suppose that the anticanonical map $\psi:X\rightarrow X'$ is not an isomorphism, or equivalently, the weak Fano threefold $X$ is not Fano. As we have seen in \Cref{subsection:extremal_contractions}, in that case, the birational morphism $\psi$ corresponds to the contraction of the $K_X$-trivial extremal ray of $X$. \par 
In this section, we suppose furthermore that there exists an effective divisor on $X$ that is contracted by $\psi$. Then there are infinitely many $K_X$-trivial curves and the exceptional locus $E\sub X$ of $\psi$ is of codimension 1. So $\psi$ cannot be small (i.e.\ isomorphic in codimension 1). By the classification of the contraction of extremal rays (see e.g.\ \cite[Prop.\ 2.5]{Mori_Kollar_Birational_geometry}), $\psi:X\rightarrow X'$ is then a divisorial contraction. That is, the exceptional divisor $E\sub X$ of $\psi$ is a prime divisor of $X$. \par
By the classification of \cite{Weak_Fano_threefolds_I-Jahnke_Peternell_Radloff}, we know that $\psi:X\rightarrow X'$ is the blowup of $X'$ along a smooth curve $C'\sub X'$ and that the normal threefold $X'$ is a Fano threefold of Picard rank 1. Moreover, $X'$ has compound Du Val singularities along $C'$ and $X'$ is Gorenstein (meaning $X'$ is Cohen-Macaulay and $K_{X'}$ is a Cartier divisor). \par
We will now show that $X'$ is canonical but not terminal. Since $K_{X'}$ is Cartier, it can be pulled back under the resolution $\psi:X\rightarrow X'$ and there exists $a\in \mathbb{Q}$ such that $K_X=\psi^*(K_{X'})+aE$.
Considering an effective curve $\Gamma \sub X$ contracted by $\psi$, we note that $\psi^*(K_{X'})\cdot \Gamma = K_{X'}\cdot \psi_*(\Gamma)=0$. Since $\psi$ is a contraction of the $K_X$-trivial extremal ray, we deduce that $0=K_X\cdot \Gamma= a E\cdot \Gamma$. This implies that the $\psi$-exceptional divisor $ aE$ is $\psi$-nef. By the negativity lemma (\cite[Thm.\ 1.3.9]{Complex_Algebraic_Threefolds_Kawakita}), it follows that $-aE$ is effective. However, $E$ is an effective prime divisor by assumption. So we must have $a=0$.  \par 
We have therewith shown that the normal threefold $X'$ is indeed only canonical and not terminal. In particular, there exists no Sarkisov link to $X'$. This case can occur for the eight pairs $(g,d)\in \{(0,4),\:(2,5),\:(3,8),\:(4,6),\:(6,10),\:(8,10),\:(11,12),\:(13,12)\}$ (see \Cref{table:anticanonical_map_divisorial}). They were all treated in \cite{Weak_Fano_threefolds_I-Jahnke_Peternell_Radloff} explicitly.\par
The following two examples study when the anticanonical map corresponding to a curve $C\sub \Pfour$ of type $(g,d)\in \{(0,4),\:(3,8)\}$ is divisorial.\par  These examples also explain why these two pairs yield two different weak Fano threefolds: If $(g,d)=(0,4)$, then the blowup $X$ of a smooth hyperquadric $Y\sub\Pfour$ along $C\sub Y$ might be Fano (see \Cref{table:Fano_threefolds_divisorial}) or $X$ might not give rise to a Sarkisov link (see \Cref{ex:rational_quartic_curve}). Similarly, if $(g,d)=(3,8)$, then $X$ might give rise to a Sarkisov link of type II (see \Cref{table:Sarkisov_link_divisorial})  or to no link at all (see \Cref{ex:curve_degre8_genus_3}).
\begin{ex} \label{ex:rational_quartic_curve}
Let $Y\sub \Pfour$ be a smooth hyperquadric and $C\sub Y$ a smooth irreducible curve of degree 4 and genus 0. Suppose that $C$ is contained in a hyperplane $H\sub \Pfour$. Then the hyperplane section $S=H\cap Y$ of $Y$ is isomorphic to $\PxP$ and in $\Pic(S)$, we have $C=3f_1+f_2$, where $f_i$ is the linear equivalence class of a fiber of the $i-$th-projection $pr_i:\PxP\rightarrow \Pone$ for $i\in\{1,2\}$ (\Cref{lem:existence_of_curves_in_quadric_hypersurface}).\par 
Every fiber $f_p:=pr_2^{-1}(\{p\})=\Pone \times \{p\}$ of a point $p\in \Pone$ is a line in $S$ and is contained in the linear system $|f_2|$. It thus satisfies $f_p\cdot C=f_2\cdot(3f_1+f_2)=3$ by the intersection rules from \Cref{lem:smooth_quadric_surface_in_P4}. So there are infinitely many 3-secant lines of $C$ in $S$.\par 
Denoting the blowup of $Y$ along $C$ by $\pi:X\rightarrow Y$ and the exceptional divisor by $E=\pi^{-1}(C)$, we know by \Cref{prop:hyperplane_section_implies_weak_Fano} that $X$ is weak Fano. Moreover, the discussion in \Cref{subsection:extremal_contractions} shows that the anticanonical map $\psi:X\rightarrow X'$ to some projective normal threefold $X'$ is the contraction of the $K_X$-trivial extremal ray of $\NE(X)$. \par 
The fibers $f_p$ with $p\in\Pone$ satisfy $E\cdot \widetilde{f_p}=C\cdot f_p=3$, where $\widetilde{f_p}\sub X$ denotes the strict transform of $f_p$ under $\pi$. It follows that all $\widetilde{f_p}$ intersect the canonical divisor $K_X$ trivially (\Cref{lem:3n+1_secants}) and are thus contracted by $\psi$. The union of all these contracted curves is equal to the strict transform of $ \bigcup_{p\in\Pone} f_p = \PxP \simeq S $. Therefore, $\psi$ contracts the surface $\widetilde{S}\sub X$, which is the strict transform of $S$ under $\pi$ and a prime divisor of $X$.
In other words, the anticanonical map $\psi$ is divisorial and therefore does not give rise to a Sarkisov link.
\end{ex}
\begin{ex}
\label{ex:curve_degre8_genus_3}
Let $Y\sub \Pfour$ be a smooth hyperquadric and $C\sub Y$ a smooth irreducible curve of degree 8 and genus 3. Suppose that $C$ is contained in a smooth hyperquadric section $S\sub \Pfour$ of $Y$. Then $S$ is a quartic del Pezzo surface that is given as the blowup $\eta:S\rightarrow \Ptwo$ of $\Ptwo$ in five distinct points $p_1,\dots,p_5\in\Ptwo$ such that no three of them are collinear (\Cref{lem:quartic_del_Pezzo_surface}). Denote the pullback of a general line under $\eta$ by $L$ and the exceptional divisors by $E_i=\eta^{-1}(\{p_i\})$ for $i\in\{1,\dots,5\}$. \par 
Assume that $C$ does not admit a 4-secant line. Then $C$ is given in $\Pic(S)$ by $C=5L-E_1-E_2-E_3-E_4-3E_5$ (\Cref{lem:degree_k_and_multiplicities_mi_of_curve_in_quartic_del_Pezzo} and \Cref{rmk:quadric_del_Pezzo_4_secant_lines}).\par 
Using the facts from \Cref{lem:quartic_del_Pezzo_surface}, we can observe that the divisor $D:=2L-E_1-E_2-E_3-E_4$ satisfies $D\cdot C=10+4(-1)=6$ and corresponds to the strict transform in $S$ of a conic in $\Pfour$ passing through the four points $p_1,\dots,p_4$ but not through $p_5$. Since every point $q\in\Ptwo\setminus\{p_1,\dots,p_5\}$ uniquely defines such a conic $c_q$, there are infinitely many 6-secant conics $s_q\sub S$ of $C$, which are the strict transforms in $S$ of the conics $c_q$. \par 
We can now argue analogously to \Cref{ex:rational_quartic_curve}: \Cref{lem:quartic_del_Pezzo_in_list_implies_weak_Fano} guarantees that the blowup $X$ of $Y$ along $C$ is weak Fano and we have argued in \Cref{subsection:extremal_contractions} that the anticanonical map $\psi:X\rightarrow X'$ to some projective normal threefold $X'$ is the contraction of the $K_X$-trivial extremal ray of $\NE(X)$. \par 
All the strict transforms $s_q\sub S$ of the conics $c_q\sub \Ptwo$ passing through $p_1,\dots,p_4$ and $q\in\Ptwo\setminus \{p_1,\dots,p_5\}$ satisfy $E\cdot \widetilde{s_q}=C\cdot s_q=6$, where $E=\pi^{-1}(C)$ denotes the exceptional divisor of the blowup $\pi:X\rightarrow Y$ of $Y$ along $C$ and $\widetilde{s_q}$ is the strict transforms of $s_p$ in $X$. Such strict transforms $\widetilde{s_q}$ intersect thus the canonical divisor $K_X$ trivially (\Cref{lem:3n+1_secants}). So they are contracted by $\psi$. \par 
Furthermore, the union of all these contracted curves is the strict transform of the union of all $s_q$ with $q\in\Ptwo\setminus \{p_1,\dots,p_5\}$, which is equal to $S$. Therefore, $\psi$ contracts the surface $\widetilde{S}\sub X$, which is the strict transform of $S$ under $\pi$ and a prime divisor of $X$.
That is to say, the anticanonical map $\psi$ is divisorial and therefore does not yield a Sarkisov link.\par 
Note that $C$ admits only finitely many 3-secant-lines since these lines have to be contained in the quartic del Pezzo surface $S$ and there are finitely many lines on $S$ (\Cref{lem:quartic_del_Pezzo_surface}). To obtain the divisor contracted by $\psi$, we thus really need to consider conics.
\end{ex}
\vspace{-0.5cm}
\begin{table}[H]
\caption{$(g,d)$-pairs yielding a weak Fano threefold $X$ such that 1) One $K_X$-negative extremal ray is the blowup of a smooth hyperquadric $Y\sub \Pfour$ along a smooth irreducible curve $C\sub Y$ of genus $g$ and degree $d$. 2) The anticanonical map $\psi:X\rightarrow X'$ is a divisorial contraction to a smooth curve. More precisely, $\psi$ is the blowup of a non-$\mathbb{Q}$-factorial normal Gorenstein Fano threefold $X'$ of index 1 having compound Du Val singularities along a smooth curve $C'\sub X'$ of degree $d'$ and genus $g'$.
}
\label{table:anticanonical_map_divisorial}
\begin{minipage}[h]{\textwidth}
\vspace{5pt}
\centering
\begin{tabular}{|l|l|l|l|l|l|l|}
\hline
$g$ & $d$ & $-K_X^3$ & $X'$ & $g'$ & $d'$ & Reference + Existence \\
\hline
$0$ & $4$ & $28$ & antican.\ model of $X \sub \mathbb{P}(\mathcal{O}_Y^{\oplus 3} \oplus \mathcal{O}_Y(2))$ & 0 &2 & \cite[A.4 Nr.\ 18]{Weak_Fano_threefolds_I-Jahnke_Peternell_Radloff}\\
\hline
$2$ & $5$ & $26$ & antican.\ model of $X \sub \mathbb{P}(\mathcal{O}_Y^{\oplus 2} \oplus \mathcal{O}_Y(2))$ & 0 & 
1 & \cite[A.4 Nr.\ 17]{Weak_Fano_threefolds_I-Jahnke_Peternell_Radloff}\\
\hline
$3$ & $8$ & $10$ & a canonical Gorenstein threefold & 0 & 2 & \cite[A.4 Nr.\ 16]{Weak_Fano_threefolds_I-Jahnke_Peternell_Radloff}\\
& & & from the list \cite[A.1]{Weak_Fano_threefolds_I-Jahnke_Peternell_Radloff} & & & \\ \hline
$6$ & $10$ & $4$ & quartic hypersurface in $\Pfour$ & 1 & 5& \cite[A4 Nr.\ 14]{Weak_Fano_threefolds_I-Jahnke_Peternell_Radloff}\\
& & & or 2:1 over $Y$, ramified quartic & & & \\
\hline
$8$ & $10$ & $8$ & complete intersection of 3 quadrics in $\mathbb{P}^6$ & 0 & 1 & \cite[A.4 Nr.\ 15]{Weak_Fano_threefolds_I-Jahnke_Peternell_Radloff}\\
\hline
$11$ & $12$ & $2$ & 2:1 over $\Pthree$, ramified sextic & 3 & 6 & \cite[A.4 Nr.\ 13]{Weak_Fano_threefolds_I-Jahnke_Peternell_Radloff}\\			
\hline
\end{tabular}		
\end{minipage}
\end{table}	
\vspace{-0.5cm}
\begin{table}[H]
	\caption{$(g,d)$-pairs yielding a weak Fano threefold $X$ such that 1) One $K_X$-negative extremal ray is the blowup of a smooth hyperquadric $Y\sub \Pfour$ along a smooth irreducible curve $C\sub Y$ of genus $g$ and degree $d$. 2) The anticanonical map $\psi:X\rightarrow X'$ is a divisorial contraction to a point. More precisely, $\psi$ is the blowup of a non-$\mathbb{Q}$-factorial canonical Gorenstein Fano threefold $X'$ from the list \cite[A.1]{Weak_Fano_threefolds_I-Jahnke_Peternell_Radloff} in a point.
	}
	\label{table:anticanonical_map_divisorial}
	\begin{minipage}[h]{\textwidth}
		\vspace{5pt}
		\centering
		\begin{tabular}{|l|l|l|l|l|l|l|}
			\hline
			$g$ & $d$ & $-K_X^3$ & Reference + Existence \\
			\hline
			4 & 6 & 24 & \cite[A.4 Nr.\ 25]{Weak_Fano_threefolds_I-Jahnke_Peternell_Radloff} with $r=3, H^3=2$ and $r'=1$ \\
			\hline
			13 & 12 & 6 & \cite[A.4 Nr.\ 25]{Weak_Fano_threefolds_I-Jahnke_Peternell_Radloff} with $r=3, H^3=2$ and $r'=2$ \\
			\hline
\end{tabular}		
\end{minipage}
\end{table}	
\subsection{Weak Fano threefolds yielding a flop and a Sarkisov link} \label{subsection:flop}
We now suppose that the birational anticanonical map $\psi: X\rightarrow X'$ neither is an isomorphism nor contracts an effective divisor of $X$. That is to say, the weak Fano threefold $X$ is not Fano and $\psi$ is not a divisorial contraction. 
Then $\psi$ contracts only finitely many curves, or equivalently, there are only finitely many $K_X$-trivial curves. This is the case if and only if the exceptional locus $E\sub X$ of $\psi$ is of codimension at least 2.\par 
In other words, $\psi$ is small and a flopping contraction. For this reason, \cite[Thm.\ 1.4.15]{Shafarevich_Parshin_AlgGeom_Fano_varieties}) guarantees the existence of a $\mathbb{Q}$-factorial terminal threefold $X^+$, a flop $\chi:X\dashrightarrow X^+$ and a birational morphism $\psi^+:X^+\rightarrow X'$ such that $\chi = (\psi^+)^{-1} \circ \psi$. Moreover, $\psi^+$ is the anticanonical map of $X^+$ and $X^+$ is also a smooth weak Fano threefold of Picard rank 2 and $-K_X^3=-K_{X^+}^3$ (\cite[Prop.\ 2.2]{Weak_Fano_threefolds_II-Jahnke_Peternell_Radloff}).\par
Using the same arguments as in \Cref{subsection:extremal_contractions}, we know that $X^+$ has a $K_{X+}$-negative extremal contraction $\phi:X^+\rightarrow Y^+$ to a normal Fano threefold $Y^+$ of Picard rank 1. Analogously to the Fano case in \Cref{subsection:Fano_threefolds}, there are two possible cases:
\begin{enumerate}
\item[(1)] $\phi:X^+\rightarrow Y^+$ is a divisorial contraction and then it is the blowup of $Y^+$ in a point or along a smooth irreducible curve. This case yields a Sarkisov link $Y \dashrightarrow Y^+$ of type II. This case appears for the 18 pairs listed in \Cref{table:anticanonical_map_small_flopped_extremal_contraction_divisorial} and they were all treated in a series of papers by \cite{Cutrone_Marshburn_Classification}, \cite{Arap_Cutrone_Marshburn_weak_Fano_threefolds_Existence} and \cite{Cutrone_Marshburn_Update}.
\item[(2)] $Y^+$ is smooth and $\phi:X^+\rightarrow Y^+$ is of fibring type. More precisely, one of the following applies:
\begin{enumerate}
\item If $\dim(Y^+)=2$, then $\phi$ is a conic bundle over $\Ptwo$.
\item If $\dim(Y^+)=1$, then $\phi$ is a del Pezzo fibration over $\Pone$.
\end{enumerate}
Both cases yields a Sarkisov link $Y\dashrightarrow X^+$ of type I. They appear for the six pairs in \Cref{table:anticanonical_map_small_flopped_extremal_contraction_fibring} and were studied in the paper \cite{Weak_Fano_threefolds_II-Jahnke_Peternell_Radloff}, and for the case that $\phi$ is a del Pezzo fibration also in \cite{Takeuchi-Weak_Fano_threefolds_del_Pezzo_fibration}.
\end{enumerate}
The construction of each of these weak Fano threefolds $X$ and the Sarkisov links arising from them is given in \cite{Weak_Fano_threefolds_II-Jahnke_Peternell_Radloff} or \cite{Takeuchi-Weak_Fano_threefolds_del_Pezzo_fibration} for all pairs $(g,d)$, except for the pairs $(g,d)\in \{(0,7),\:(5,9)\}$. For these two pairs, \Cref{cor:existence_gdpairs_yielding_weak_Fano} guarantees the existence of a corresponding weak Fano threefold and Sarkisov link, which was not known previously. \\

\begin{minipage}{\textwidth }
\vspace{1cm}
\centering
Sarkisov links if $X$ is a weak Fano threefold with small anticanonical map admitting a flop $\chi:X\dashrightarrow X^+$, depending on whether the $K_{X^+}$-negative extremal contraction $\phi:X^+\rightarrow Y^+$ on $X^+$ is of fibring type (fib) or a divisorial contraction (div):\\
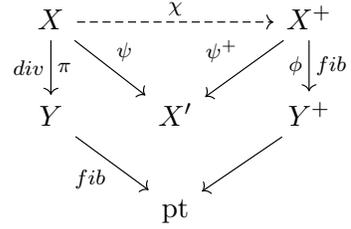
\begin{equation*}
\begin{tikzcd} 
X \ar[d, "div"', "\pi"] \ar[dr, "\psi"] \ar[rr, "\chi",dashed]&& X^+ \ar[dl, "\psi^+"']\ar[d, "\phi"', "fib"] \\
Y \ar[dr, "fib"']
& X' & Y^+ \ar[dl,]\\
&\textup{pt} &
\end{tikzcd} 
\end{equation*}	
Sarkisov link $\chi:Y\dashrightarrow X^+$ of type I: $\phi:X^+ \rightarrow Y^+$ is of fibring type.

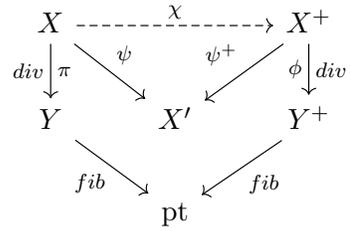
\begin{equation*}
\begin{tikzcd}
X \ar[d, "div"', "\pi"] \ar[dr, "\psi"] \ar[rr, "\chi",dashed]&& X^+ \ar[dl, "\psi^+"']\ar[d, "\phi"', "div"] \\
Y \ar[dr, "fib"']
& X' & Y^+ \ar[dl,"fib"]\\
&\textup{pt} &
\end{tikzcd}
\end{equation*}
Sarkisov link $\chi:Y\dashrightarrow Y^+$ of type II: $\phi:X^+\rightarrow Y^+$ is a divisorial contraction.
\end{minipage}\\
\vspace{-0.5cm}
\begin{table}[H] \label{table:Sarkisov_link_divisorial}
\caption{$(g,d)$-pairs yielding a weak Fano threefold $X$ such that 1) One $K_X$-negative extremal ray is the blowup of a smooth hyperquadric $Y\sub \Pfour$ along a smooth irreducible curve $C\sub Y$ of genus $g$ and degree $d$. 2) The anticanonical map $\psi$ is small and admits a flop $\chi:X\dashrightarrow X^+$ to a smooth weak Fano threefold $X^+$ of Picard rank 2 such that the $K_{X^+}$-negative extremal contraction $\phi:X^+\rightarrow Y^+$ is of fibring type.\\
More precisely, the contraction $\phi:X^+\rightarrow Y^+$ is the blowup of a Fano threefold $Y^+$ of Picard rank 1 along a curve or in a point. Denoting the exceptional divisor of the blowup $\pi:X\rightarrow Y$ by $E$, the image $\phi(E)$ is either a smooth curve in $Y^+$ of degree $d^+$ and genus $g^+$ and $Y^+$ is smooth (the first 16 pairs), or $\phi(E)$ is a singular point of $Y^+$ (the last two pairs). 
}
\label{table:anticanonical_map_small_flopped_extremal_contraction_divisorial}
\begin{minipage}[h]{\textwidth}
\vspace{5pt}
\centering
\begin{tabular}{|l|l|l|l|l|l|l|}
\hline
$g$ & $d$ & $-K_X^3$ & $-K_{Y^+}^3$ & $g^+$ & $d^+$ & Reference and Existence\\
\hline
0 & 6 & 16 & 22 & 0 & 2 & \cite[(2.8)]{Takeuchi_Some_birational_maps} \\
\hline
0 & 8 & 4 & 54 & 0 & 8 & \cite[No.\ 71, Thm.\ 3.2]{Arap_Cutrone_Marshburn_weak_Fano_threefolds_Existence} \\
\hline
1 & 7 & 12 & 40 & 1 & 7 & \cite[No.\ 105, Thm.\ 3.2]{Arap_Cutrone_Marshburn_weak_Fano_threefolds_Existence} \\
\hline
1 & 8 & 6 & 22 & 1 & 8 & \cite[No.\ 86, Thm.\ 3.2]{Arap_Cutrone_Marshburn_weak_Fano_threefolds_Existence} \\
\hline
2 & 7 & 14 & 18 & 0 & 1 & \cite[Thm.\ 4.3.3(vii) + Thm.\ 4.3.7(ii) ]{Shafarevich_Parshin_AlgGeom_Fano_varieties} \\
\hline
2 & 8 & 8 & 18 & 0 & 4 & \cite[No.\ 97, Thm.\ 3.2]{Arap_Cutrone_Marshburn_weak_Fano_threefolds_Existence} \\
\hline
2 & 9 & 2 & 54 & 2 & 9 & \cite[No.\ 44, Thm.\ 3.2]{Arap_Cutrone_Marshburn_weak_Fano_threefolds_Existence}\\ 
\hline
3 & 8 & 10 & 54 & 3 & 8 & \cite[No.\ 102, Thm.\ 4.1]{Cutrone_Marshburn_Update} \\
\hline
3 & 9 & 4 & 16 & 0 & 5 & \cite[No.\ 70, Thm.\ 3.2]{Arap_Cutrone_Marshburn_weak_Fano_threefolds_Existence} \\
\hline
4 & 9 & 6 & 54 & 4 & 9 & \cite[No.\ 88, Thm.\ 3.2]{Arap_Cutrone_Marshburn_weak_Fano_threefolds_Existence}\\
\hline
5 & 10 & 2 & 54 & 5 & 10 & \cite[No.\ 45, Thm.\ 3.2]{Arap_Cutrone_Marshburn_weak_Fano_threefolds_Existence}\\
\hline 
6 & 10 &  4 & 54 & 6 & 10 & \cite[No.\ 72, Thm.\ 4.1]{Cutrone_Marshburn_Update} \\
\hline
7 & 10 & 6 & 12 & 0 & 2 & \cite[(2.8)]{Takeuchi_Some_birational_maps} \\
\hline
8 & 11 & 2 & 54 & 8 & 11 & \cite[No.\ 46, Thm.\ 3.2]{Arap_Cutrone_Marshburn_weak_Fano_threefolds_Existence} \\
\hline 
11 & 12 & 2 & 54 & 11 & 12 & \cite[No.\ 47, Thm.\ 4.1]{Cutrone_Marshburn_Update} \\ 
\hline 
14 & 13 & 2 & 54 & 14 & 13 & \cite[No.\ 48, Thm.\ 3.2]{Arap_Cutrone_Marshburn_weak_Fano_threefolds_Existence}\\
\hline
\end{tabular}		\par
\vspace{1cm}
\begin{tabular}{|l|l|l|l|l|l|l|}
\hline
g&d&$-K_X^3$&$-K_{Y^+}^3$ & E& $\phi(E)$ & Reference + Existence \\
\hline 
4&8&12&14 &$E\simeq \PxP$& ordinary & \cite[3.2.2, No.\ 5]{Cutrone_Marshburn_Classification}\\
& & & & & double point& \\
\hline
6 &9&10& $\frac{21}{2}$ & $E\simeq \Ptwo$& quadruple non- & \cite[3.2.3, No.\ 5]{Cutrone_Marshburn_Classification} \\
& & & & & Gorenstein &\\
\hline
\end{tabular}
\end{minipage}
\end{table}
\begin{table}[H]
\caption{$(g,d)$-pairs yielding a weak Fano threefold $X$ such that 1) One $K_X$-negative extremal ray is the blowup of a smooth hyperquadric $Y\sub \Pfour$ along a smooth irreducible curve $C\sub Y$ of genus $g$ and degree $d$. 2) The anticanonical map $\psi$ is small and admits a flop $\chi:X\dashrightarrow X^+$ to a smooth weak Fano threefold $X^+$ of Picard rank 2 such that the $K_{X^+}$-negative extremal contraction $\phi:X^+\rightarrow Y^+$ is of fibring type. More precisely, $\phi$ is either a conic bundle over $\Ptwo$ or a del Pezzo fibration over $\Pone$ (abbreviated with "dPf").}
\label{table:anticanonical_map_small_flopped_extremal_contraction_fibring}
\begin{minipage}[h]{\textwidth}
\vspace{5pt}
\centering
\begin{tabular}{|l|l|l|l|l|l|}
\hline
$g$ & $d$ & $-K_X^3$ & Fibring type of $\phi$& Reference & Existence\\
\hline
0 & 5 & 22 & conic bundle & \cite[Prop.\ 7.13 Nr.\ 17]{Weak_Fano_threefolds_II-Jahnke_Peternell_Radloff} & \cite[Thm.\ 7.14]{Weak_Fano_threefolds_II-Jahnke_Peternell_Radloff}\\
\hline
1 & 6 & 18 & dPf of degree 6 & \cite[7.4 Nr.\ 3]{Weak_Fano_threefolds_II-Jahnke_Peternell_Radloff} & \cite[Prop.\ 6.5 Nr.\ 21]{Weak_Fano_threefolds_II-Jahnke_Peternell_Radloff}\\
\hline 
0 & 7 & 10 & conic bundle & \cite[Prop 7.13 Nr.\ 13]{Weak_Fano_threefolds_II-Jahnke_Peternell_Radloff} & \Cref{cor:existence_gdpairs_yielding_weak_Fano}\\
\hline
3 & 7 & 16 & dPf of degree 5 & \cite[7.4 Nr.\ 5]{Weak_Fano_threefolds_II-Jahnke_Peternell_Radloff}, & \cite[Prop.\ 6.5 Nr.\ 20]{Weak_Fano_threefolds_II-Jahnke_Peternell_Radloff},\\
& & & & \cite[(2.13.4)]{Takeuchi-Weak_Fano_threefolds_del_Pezzo_fibration}
& \cite[(3.5.1)]{Takeuchi-Weak_Fano_threefolds_del_Pezzo_fibration}\\
\hline
5 & 9 & 8 & conic bundle & \cite[Prop.\ 7.13 Nr.\ 10]{Weak_Fano_threefolds_II-Jahnke_Peternell_Radloff} & \Cref{cor:existence_gdpairs_yielding_weak_Fano}\\
\hline
9 & 11 & 4 & dPf of degree 5 & \cite[7.4 Nr.\ 17]{Weak_Fano_threefolds_II-Jahnke_Peternell_Radloff}, & \cite[(8.6.10)]{Takeuchi-Weak_Fano_threefolds_del_Pezzo_fibration}\\
& & & & \cite[(2.13.10)]{Takeuchi-Weak_Fano_threefolds_del_Pezzo_fibration} &\\
\hline 
\end{tabular}		
\end{minipage}
\end{table}	
%
%

\bibliographystyle{alpha}
\bibliography{references}{}

\newcommand{\etalchar}[1]{$^{#1}$}
\begin{thebibliography}{ACC{\etalchar{+}}23}

\bibitem[ACC{\etalchar{+}}23]{Fano_threefolds}
Carolina Araujo, Ana-Maria Castravet, Ivan Cheltsov, Kento Fujita, Anne-Sophie
  Kaloghiros, Jesus Martinez-Garcia, Constantin Shramov, Hendrik S\"u\ss, and
  Nivedita Viswanathan.
\newblock {\em The {C}alabi problem for {F}ano threefolds}, volume 485 of {\em
  London Mathematical Society Lecture Note Series}.
\newblock Cambridge University Press, Cambridge, 2023.

\bibitem[ACM17]{Arap_Cutrone_Marshburn_weak_Fano_threefolds_Existence}
Maxim Arap, Joseph Cutrone, and Nicholas Marshburn.
\newblock On the existence of certain weak {F}ano threefolds of {P}icard number
  two.
\newblock {\em Math. Scand.}, 120(1):68--86, 2017.

\bibitem[BL12]{Blanc_Lamy_Weak_Fano_threefolds}
J\'er\'emy Blanc and St\'ephane Lamy.
\newblock Weak {F}ano threefolds obtained by blowing-up a space curve and
  construction of {S}arkisov links.
\newblock {\em Proc. Lond. Math. Soc. (3)}, 105(5):1047--1075, 2012.

\bibitem[BL15]{Blanc_Lamy_Cubic}
J\'{e}r\'{e}my Blanc and St\'{e}phane Lamy.
\newblock On birational maps from cubic threefolds.
\newblock {\em North-West. Eur. J. Math.}, 1:55--84, 2015.

\bibitem[BLZ21]{Quotients_of_higher_dim_Cremona_groups-Blanc_Lamy_Zimmermann}
J\'er\'emy Blanc, St\'ephane Lamy, and Susanna Zimmermann.
\newblock Quotients of higher-dimensional {C}remona groups.
\newblock {\em Acta Math.}, 226(2):211--318, 2021.

\bibitem[CM13]{Cutrone_Marshburn_Classification}
Joseph~W. Cutrone and Nicholas~A. Marshburn.
\newblock Towards the classification of weak {F}ano threefolds with {$\rho=2$}.
\newblock {\em Cent. Eur. J. Math.}, 11(9):1552--1576, 2013.

\bibitem[CM25]{Cutrone_Marshburn_Update}
Joseph~W. Cutrone and Nicholas~A. Marshburn.
\newblock {An Update on the Classification of Rank 2 Weak Fano Threefolds}.
\newblock {\em Taiwanese Journal of Mathematics}, pages 1 -- 13, 2025.

\bibitem[D'A17]{Jean_Dalmeida}
Jean D'Almeida.
\newblock {Volumes de Fano faibles obtenus par éclatement d’une courbe de
  $\mathbf{P}^{3}$}.
\newblock {\em Kyoto Journal of Mathematics}, 57(1):97 -- 106, 2017.

\bibitem[dC97]{De_Cataldo-Genus_of_curves_on_3d_quadric}
Mark Andrea~A. de~Cataldo.
\newblock The genus of curves on the three-dimensional quadric.
\newblock {\em Nagoya Math. J.}, 147:193--211, 1997.

\bibitem[Deb01]{Debarre-Higher_dimensional_Alg_Geom}
Olivier Debarre.
\newblock {\em Higher-dimensional algebraic geometry}.
\newblock Universitext. Springer-Verlag, New York, 2001.

\bibitem[Dol12]{Dolgachev}
Igor~V. Dolgachev.
\newblock {\em Classical algebraic geometry}.
\newblock Cambridge University Press, Cambridge, 2012.
\newblock A modern view.

\bibitem[GH78]{Griffiths_Harris_Principles_of_Alg_Geom}
Phillip Griffiths and Joseph Harris.
\newblock {\em Principles of algebraic geometry}.
\newblock Pure and Applied Mathematics. Wiley-Interscience [John Wiley \&
  Sons], New York, 1978.

\bibitem[Har77]{Hartshorne}
Robin Hartshorne.
\newblock {\em Algebraic geometry}, volume No. 52 of {\em Graduate Texts in
  Mathematics}.
\newblock Springer-Verlag, New York-Heidelberg, 1977.

\bibitem[HM13]{Sarkisov_program_Hacon_McKernan}
Christopher~D. Hacon and James McKernan.
\newblock The {S}arkisov program.
\newblock {\em J. Algebraic Geom.}, 22(2):389--405, 2013.

\bibitem[Isk77]{Iskovskikh-Fano_threefolds_I}
V.~A. Iskovskikh.
\newblock Fano threefolds. {I}.
\newblock {\em Izv. Akad. Nauk SSSR Ser. Mat.}, 41(3):516--562, 717, 1977.

\bibitem[Isk78]{Iskovskikh-Fano_threefolds_II}
V.~A. Iskovskikh.
\newblock Fano threefolds. {II}.
\newblock {\em Izv. Akad. Nauk SSSR Ser. Mat.}, 42(3):506--549, 1978.

\bibitem[Jac85]{Jacobson-Quadratic_form_diagonalizable}
Nathan Jacobson.
\newblock {\em Basic algebra. {I}}.
\newblock W. H. Freeman and Company, New York, second edition, 1985.

\bibitem[JPR05]{Weak_Fano_threefolds_I-Jahnke_Peternell_Radloff}
Priska Jahnke, Thomas Peternell, and Ivo Radloff.
\newblock Threefolds with big and nef anticanonical bundles. {I}.
\newblock {\em Math. Ann.}, 333(3):569--631, 2005.

\bibitem[JPR11]{Weak_Fano_threefolds_II-Jahnke_Peternell_Radloff}
Priska Jahnke, Thomas Peternell, and Ivo Radloff.
\newblock Threefolds with big and nef anticanonical bundles {II}.
\newblock {\em Cent. Eur. J. Math.}, 9(3):449--488, 2011.

\bibitem[Kaw24]{Complex_Algebraic_Threefolds_Kawakita}
Masayuki Kawakita.
\newblock {\em Complex algebraic threefolds}, volume 209 of {\em Cambridge
  Studies in Advanced Mathematics}.
\newblock Cambridge University Press, Cambridge, 2024.

\bibitem[KM98]{Mori_Kollar_Birational_geometry}
J\'anos Koll\'ar and Shigefumi Mori.
\newblock {\em Birational geometry of algebraic varieties}, volume 134 of {\em
  Cambridge Tracts in Mathematics}.
\newblock Cambridge University Press, Cambridge, 1998.
\newblock With the collaboration of C. H. Clemens and A. Corti, Translated from
  the 1998 Japanese original.

\bibitem[Knu02]{Knutsen-Smooth_curves_on_projective_K3_surface}
Andreas~Leopold Knutsen.
\newblock Smooth curves on projective {$K3$} surfaces.
\newblock {\em Math. Scand.}, 90(2):215--231, 2002.

\bibitem[Kol87]{Kollar-Intro_to_Moris_Program}
J\'anos Koll\'ar.
\newblock The structure of algebraic threefolds: an introduction to {M}ori's
  program.
\newblock {\em Bull. Amer. Math. Soc. (N.S.)}, 17(2):211--273, 1987.

\bibitem[Laz04]{Lazarsfeld_positivity_in_AlgGeom_I}
Robert Lazarsfeld.
\newblock {\em Positivity in algebraic geometry. {I}}, volume~48 of {\em
  Ergebnisse der Mathematik und ihrer Grenzgebiete. 3. Folge. A Series of
  Modern Surveys in Mathematics [Results in Mathematics and Related Areas. 3rd
  Series. A Series of Modern Surveys in Mathematics]}.
\newblock Springer-Verlag, Berlin, 2004.
\newblock Classical setting: line bundles and linear series.

\bibitem[Liu02]{Qing_Liu-Algebraic_Geometry}
Qing Liu.
\newblock {\em Algebraic geometry and arithmetic curves}, volume~6 of {\em
  Oxford Graduate Texts in Mathematics}.
\newblock Oxford University Press, Oxford, 2002.
\newblock Translated from the French by Reinie Ern\'e, Oxford Science
  Publications.

\bibitem[LS24]{Motivic_invariants-Shinder_Lin}
Hsueh-Yung Lin and Evgeny Shinder.
\newblock Motivic invariants of birational maps.
\newblock {\em Ann. of Math. (2)}, 199(1):445--478, 2024.

\bibitem[MM83]{Mori_Mukai_On_Fano_Threefolds}
Shigefumi Mori and Shigeru Mukai.
\newblock On {F}ano {$3$}-folds with {$B\sb{2}\geq 2$}.
\newblock In {\em Algebraic varieties and analytic varieties ({T}okyo, 1981)},
  volume~1 of {\em Adv. Stud. Pure Math.}, pages 101--129. North-Holland,
  Amsterdam, 1983.

\bibitem[MM82]{Mori_Mukai-Fano_threefolds_Picard_rank_2}
Shigefumi Mori and Shigeru Mukai.
\newblock Classification of {F}ano {$3$}-folds with {$B\sb{2}\geq 2$}.
\newblock {\em Manuscripta Math.}, 36(2):147--162, 1981/82.

\bibitem[Mor82]{Mori-Extremal_contractions_threefolds}
Shigefumi Mori.
\newblock Threefolds whose canonical bundles are not numerically effective.
\newblock {\em Ann. of Math. (2)}, 116(1):133--176, 1982.

\bibitem[Mum95]{Mumford-Algebraic_Geometry_I}
David Mumford.
\newblock {\em Algebraic geometry. {I}}.
\newblock Classics in Mathematics. Springer-Verlag, Berlin, 1995.
\newblock Complex projective varieties, Reprint of the 1976 edition.

\bibitem[PS99]{Shafarevich_Parshin_AlgGeom_Fano_varieties}
A.~N. Parshin and I.~R. Shafarevich.
\newblock {\em Algebraic geometry. {V}}, volume~47 of {\em Encyclopaedia of
  Mathematical Sciences}.
\newblock Springer-Verlag, Berlin, 1999.

\bibitem[Rei97]{Miles_Reid_Chapters_on_algebraic_surfaces}
Miles Reid.
\newblock Chapters on algebraic surfaces.
\newblock In {\em Complex algebraic geometry ({P}ark {C}ity, {UT}, 1993)},
  volume~3 of {\em IAS/Park City Math. Ser.}, pages 3--159. Amer. Math. Soc.,
  Providence, RI, 1997.

\bibitem[Sha13]{Shafarevich_Basic_Alg_Geom_1}
Igor~R. Shafarevich.
\newblock {\em Basic algebraic geometry. 1}.
\newblock Springer, Heidelberg, third edition, 2013.
\newblock Varieties in projective space.

\bibitem[Shi89]{Shin_3dimensional_Fano_varieties_with_canonical_singularities}
Kil-Ho Shin.
\newblock {$3$}-dimensional {F}ano varieties with canonical singularities.
\newblock {\em Tokyo J. Math.}, 12(2):375--385, 1989.

\bibitem[Tak89]{Takeuchi_Some_birational_maps}
Kiyohiko Takeuchi.
\newblock Some birational maps of {F}ano {$3$}-folds.
\newblock {\em Compositio Math.}, 71(3):265--283, 1989.

\bibitem[Tak22]{Takeuchi-Weak_Fano_threefolds_del_Pezzo_fibration}
Kiyohiko Takeuchi.
\newblock Weak {F}ano threefolds with del {P}ezzo fibration.
\newblock {\em Eur. J. Math.}, 8(3):1225--1290, 2022.

\bibitem[Vak25]{Rising_Sea-Vakil}
Ravi Vakil.
\newblock {\em The rising sea---foundations of algebraic geometry}.
\newblock Princeton University Press, Princeton, NJ, [2025] \copyright 2025.

\end{thebibliography}

\end{document}